\documentclass[12pt]{article}
\usepackage{amssymb}
\usepackage{amsmath}
\usepackage{graphicx}

\begin{document}
\newtheorem{theorem}{Theorem} [section]
\newtheorem{tha}{Theorem}
\newtheorem{conjecture}[theorem]{Conjecture}
\renewcommand{\thetha}{\Alph{tha}}

\newtheorem{corollary}[theorem]{Corollary}
\newtheorem{lemma}[theorem]{Lemma}
\newtheorem{proposition}[theorem]{Proposition}
\newtheorem{construction}[theorem]{Construction}
\newtheorem{claim}[theorem]{Claim}
\newtheorem{definition}[theorem]{Definition}
\newtheorem{observation}{Observation}
\newtheorem{question}{Question}
\newtheorem{remark}[theorem]{Remark}
\newtheorem{algorithm}[theorem]{Algorithm}
\newtheorem{example}[theorem]{Example}

\def\endproofbox{\hskip 1.3em\hfill\rule{6pt}{6pt}}
\newenvironment{proof}%
{%
\noindent{\it Proof.}
}%
{%
 \quad\hfill\endproofbox\vspace*{2ex}
}
\def\qed{\hskip 1.3em\hfill\rule{6pt}{6pt}}

\parindent=15pt

\def\dim{{\rm dim}}
\def \e{\epsilon}
\def\vec#1{{\bf #1}}
\def\va{\vec{a}}
\def\vb{\vec{b}}
\def\vc{\vec{c}}
\def\vx{\vec{x}}
\def\vy{\vec{y}}
\def\cM{{\mathcal M}}
\def\cH{{\cal H}}
\def\cF{{\mathcal F}}
\def\bE{\mathbb{E}}
\def\bP{\mathbb{P}}
\def\cA{{\mathcal A}}
\def\cB{{\mathcal B}}
\def\cJ{{\mathcal J}}
\def\cM{{\mathcal M}}
\def\cS{{\mathcal S}}
\def\cT{{\mathcal T}}
\def\wt#1{\widetilde{#1}}
%%%%%%%%%%%%%%%%%%%%%%%%%%%%%%%%%%%%%%%%%%%%%%%%%%%%%%%%%%%%%

\title{Embedding multidimensional grids into optimal hypercubes}
\author{Zevi Miller \thanks{Dept. of
Mathematics, Miami University, Oxford, OH 45056, USA,
millerz@muohio.edu} \, Dan Pritikin
\thanks{Dept. of Mathematics,
Miami University, Oxford, OH 45056, USA, pritikd@muohio.edu.}\, I.H. 
Sudborough \thanks{Program in Computer Science, University of Texas-Dallas, 
Richardson, Texas} }

\maketitle

\begin{abstract}
Let $G$ and $H$ be graphs, with $|V(H)|\geq |V(G)| $, and $f:V(G)\rightarrow V(H)$ a one to one map of their vertices.   Let $dilation(f) = max\{ dist_{H}(f(x),f(y)): xy\in E(G)  \}$, where $dist_{H}(v,w)$ is 
the distance between vertices $v$ and $w$ of $H$.  Now let $B(G,H)$ = $min_{f}\{ dilation(f) \}$, over all such maps $f$.

The parameter $B(G,H)$ is a generalization of the classic and well studied ``bandwidth" of $G$, defined as $B(G,P(n))$, where $P(n)$ is the path on $n$ points 
and $n = |V(G)|$.  Let $[a_{1}\times a_{2}\times \cdots \times a_{k}  ]$ be the $k$-dimensional grid graph with integer values $1$ through $a_{i}$ in the 
$i$'th coordinate.  In this paper, we study $B(G,H)$ in the case when $G = [a_{1}\times a_{2}\times \cdots \times a_{k}  ]$ and 
$H$ is the hypercube $Q_{n}$ of dimension $n = \lceil log_{2}(|V(G)|)   \rceil$, the hypercube of smallest dimension having at least as many points as $G$.  Our main result is that 
 $$B( [a_{1}\times a_{2}\times \cdots \times a_{k}  ],Q_{n}) \le 3k,$$
 provided $a_{i} \geq 2^{22}$ for each $1\le i\le k$.  For such $G$, the bound $3k$ improves on the previous best upper bound 
 $4k+O(1)$.  Our methods include an application of Knuth's result on two-way rounding and of the existence of spanning regular cyclic caterpillars in the hypercube.

\end{abstract}

%%%%%%%%%%%%%%%%%%%%%%%%%%%%%%%%%%%%%%%%%%%%%%%%%%%%%%%%%%%%%%%%%%%%%%%%%%%%%
\section{Introduction}\label{intro}

In this paper we will usually follow standard graph theoretic terminology, as may be found for example in \cite{W}.  We let $P(t)$ stand for the path 
on $t$ vertices.  The cartesian product $G\times H$ of two graphs $G$ and $H$ is the graph with vertex set $V = \{ (v,w): v\in V(G), w\in V(H) \}$ and edge set 
$E = \{ (v,w)(v',w'):$ either $v = v'$ and $ww'\in E(H)$, or $vv'\in E(G)$ and $w = w'  \}$.  All logarithms are taken base $2$.  

\subsection{Background and main result}\label{subintro}
The analysis of how effectively one network can simulate another, and the resulting implications for optimal design of parallel computation networks, are important topics in graph 
theoretic aspects of computer science.  One of the measures of the effectiveness of a simulation is the \textit{dilation} of the corresponding map (or ``embedding") of networks, defined as follows.  Let 
$G$ and $H$ be two graphs and $f:V(G)\rightarrow V(H)$ a map from the vertices of $G$ to those of $H$.  As a convenience we typically write such a map 
as $f:G\rightarrow H$, with the meaning that it is a map from vertices to vertices.  Similarly we sometimes write $|G|$ for $|V(G)|$. 
Apart from an exception indicated below in a review of 
previous research on our topic,  we will suppose that $|V(G)|\le |V(H)|$ and that $f$ 
is one to one.  Whether $f$ is one to one or not, we let $dilation(f) = max\{ dist_{H}(f(x),f(y)): xy\in E(G)  \}$, where $dist_{H}(v,w)$ is the distance between vertices $v$ and $w$ of $H$, defined 
as the minimum number of edges in any path of $H$ joining $v$ and $w$.  Thus $dilation(f)$ is the maximum ``stretch" 
experienced by any edge of $G$ under the map $f$.  Now define $B(G,H)$ to be $min_{f}\{ dilation(f) \}$, over all such maps $f$.  
Note that $B(G,H)$ is a generalization of the classic and well studied ``bandwidth" of $G$, 
defined as $B(G,P(n))$, where $n = |V(G)|$.   

The study of $B(G,H)$ arises when each of $G$ and $H$ is a computation network, and the goal is to have $H$ simulate a computation in $G$.  
A given map $f$ indicates how the vertices of $H$ play the roles of the vertices of $G$, and $dilation(f)$ is a measure 
of the communication delay in this roleplaying. A message between adjacent vertices $x$ and $y$ in $G$ 
taking unit time would become a message between $f(x)$ and $f(y)$ in $H$ taking 
time $dist_{H}(f(x),f(y))$, which in the worst case is $dilation(f)$ if a shortest path in $H$ joining $f(x)$ and $f(y)$ for this message is used.  Indeed the delay may 
be worse when one considers the full simulation, requiring in addition 
to $f$ a \textit{routing path} for each edge $xy\in E(G)$, namely, a path in $H$ (not necessarily shortest) joining $f(x)$ and $f(y)$. 
So let the \textit{edge congestion} of $f$ be the maximum, over all 
edges $vw\in E(H)$, of the number of routing paths in $H$ that contain $vw$.  The edge congestion of $f$ is then 
an additional contribution to the communication delay of the embedding $f$. 

In this paper we obtain upper bounds on $B(G,H)$ when $G$ is a multidimensional grid and $H$ is the smallest hypercube 
having at least $|V(G)|$ vertices.  
To clarify, let $a_{i}\geq 2$, $1\le i\le k$, be integers.  
The $k$-dimensional grid $G = [a_{1}\times a_{2}\times \cdots \times a_{k}  ]$ is the graph with  
vertex set $V(G) = \{ x = (x_{1}, x_{2}, \ldots, x_{k}): x_{i}$ an integer, $1\le x_{i}\le a_{i} \}$ and edge set 
$E(G) = \{ xy: \sum_{i=1}^{k}|x_{i} - y_{i}| =1   \}$.  So two vertices of $G$ are joined by an edge precisely when they disagree in exactly one 
coordinate, and in that coordinate they differ by $1$.  
It follows that for $x,y\in V(G)$ we have $dist_{G}(x,y) = \sum_{i=1}^{k}|x_{i} - y_{i}|$. 
One can also write $G$ as the cartesian product of paths $G = P(a_{1})\times P(a_{2})\times \cdots \times P(a_{k})$.

The \textit{n-dimensional hypercube} $Q_{n}$ is the $n$-dimensional grid $[2\times 2\times \cdots \times 2 ]$. We follow the traditional view whereby 
$V(Q_{n})$ is the set of all strings of length $n$ over the alphabet $\{0,1\}$, where two such strings are joined by an edge if they disagree 
in exactly one coordinate. This departs in a trivial way from our notation above, where we would have required $1\le x_{i}\le 2 $.  
Clearly $|V(Q_{n})| = 2^{n}$ and we let $Opt(G)$ be the smallest hypercube containing at least $|V(G)|$ vertices, so $Opt(G) = Q_{t}$ where $t =  \lceil log_{2}(|V(G)|)  \rceil$.

There is a substantial literature on the simulation of various networks by hypercubes and their related networks; the butterfly, shuffle exchange and DeBruijn graphs.   
See the books \cite{Lei} and  \cite {RH} for excellent expositions on these topics, both emphasizing bounds on dilation and congestion in graph embeddings; where the first also includes 
routing and implementation of various algorithms while the second gives a unified approach to applying separator theorems for deriving such bounds. 
An early survey on embedding graphs into hypercubes \cite{LSt}  
mentions necessary and sufficient conditions (originating in  \cite{HM}) for a graph to be a subgraph of some hypercube, and also 
the fact that for the complete binary tree $T_{n}$ on $2^{n}-1$ vertices there is an embedding $f: T_{n}\rightarrow Q_{n}$ such that for every edge $xy\in E(T_{n})$ we have $dist_{Q_{n}}(f(x),f(y)) = 1$ with 
the exception of a single edge where this distance is $2$  \cite{HL}.  In \cite{BCLR} it is shown how to embed any $2^{n}$ node bounded degree tree into $Q_{n}$ with $O(1)$ dilation  and $O(1)$ 
edge congestion, as $n$ grows.  In the same paper these results are extended to embedding bounded degree graphs with $O(1)$ separators.  In \cite{Lei} many-to-one maps of binary trees into hypercubes are considered, letting the 
\textit{load} be the maximum number of tree nodes mapped onto a hypercube node.  Using probabilistic methods 
and error correcting codes it is shown how to embed an $M$ node binary tree in an 
$N$ node hypercube with dilation 1 and load $O(\frac{M}{N} + log(N)  )$, and how to perform the same embedding with dilation $O(1)$ and load $O(\frac{M}{N} + 1)$.

Concerning the embedding of multidimensional grids into hypercubes, observe first that if $p_{1}, p_{2}, \ldots,p_{r}$ are 
positive integers summing to $n$, and $G = [P(2^{p_{1}})\times P(2^{p_{2}})\times \cdots \times P(2^{p_{r}})]$, then $Q_{n} = Opt(G)$ and $Q_{n}$ contains $G$ as a spanning subgraph.   Thus $B(G,Opt(G)) = 1$ 
in this case.  In fact one can show that $ [a_{1}\times a_{2}\times \cdots \times a_{k} ]$ is a 
subgraph of $Q_{n}$ if and only if $n \geq  \lceil log(a_{1}) \rceil + \lceil log(a_{2}) \rceil + \cdots + \lceil log(a_{k}) \rceil $; see Problem 3.20 
in \cite{Lei}.  Answering a question posed  in  \cite{LSt} about $2$-dimensional grids $G = [a_{1}\times a_{2}]$, 
it is shown in \cite{Ch} and in \cite{CHL} that $B(G,Opt(G)) \le 2$.  In \cite{Ch} it is also shown for arbitrary 
multidimensional grids $G = [a_{1}\times a_{2}\times \cdots \times a_{k} ]$ that $B(G,Opt(G)) \le 4k+1$.  Independently it was shown in \cite{JL} that $B(G,Opt(G)) \le 4k-1$ for such $G$, this upper bound 
being realized by a parallel algorithm on the hypercube. Still for such $G$, it was shown in \cite{BMS} that $B(G,Opt(G))\le k$, assuming quite involved and restrictive inequality 
constraints on the $a_{i}$. 
Returning to dimension $2$,  it was shown in \cite{KL} that determining whether a given graph $G$ can be embedded in $Opt(G)$ with 
edge congestion $1$ is NP-complete.  Subsequently it was shown   
 in \cite{RS} that any $G = [a_{1}\times a_{2}]$ can be embedded in $Opt(G)$ with edge congestion 
at most $2$ and dilation at most $3$.  Following up on a question posed in \cite{LSt}, the issue of many-to-one 
embeddings of $2$ and $3$ dimensional grids $G$ into hypercubes was explored in \cite{MS}.  For these results, let $Opt(G)/2^{t}$ denote 
the hypercube of dimension $\lceil log(|G|) \rceil - t$.  If $f: G\rightarrow Opt(G)/2^{t}$ is a many-to-one map, then as above let the \textit{load} of $f$ be 
max$\{|f^{-1}(z)|: z\in Opt(G)/2^{t}\}$. It was shown in \cite{MS} that for a $2$-dimensional 
grid $G$ there is a many-to-one map $f: G\rightarrow Opt(G)/2^{t}$ of dilation $1$ 
and load at most $1+2^{t}$, and when $G$ is $3$-dimensional there is a map $f: G\rightarrow Opt(G)/2$ of 
dilation at most $2$ and load at most $3$, and a map 
$f: G\rightarrow Opt(G)/4$ of dilation at most $3$ and load at most $5$.

The main result of the present paper is that $B( [a_{1}\times a_{2}\times \cdots \times a_{k}  ],Q_{n}) \le 3k$,  provided $a_{i} \geq 2^{22}$ for each $1\le i\le k$.  
This improves on the $4k-1$ bound above under this condition on the $a_i$.  We construct a one to one map $H^{k}:G \rightarrow Opt(G)$ 
realizing $dilation(H^{k})\le 3k$ and having congestion $O(k)$.  
Our construction uses the technique of two way rounding and the existence of regular spanning cyclic caterpillars in the hypercube.

\subsection{Some notation}\label{notation}                                       

We will need to consider multidimensional grids for which each factor (in the cartesian product) with one possible exception is a path $P(m)$, where $m$ is a power of $2$ and varies with the factor, 
as these will play the role 
of successive approximations to $Opt(G)$.  Let $e_{i} = \lceil log_{2}(a_{1}a_{2}\cdots a_{i})  \rceil$ for $1\le i\le k$, with $e_{0} = 0$.   
Letting $p_{i} = e_{i} - e_{i-1}$ for $1\le i\le k$, we 
let $Opt'(G) =  P(2^{p_{1}})\times P(2^{p_{2}})\times \cdots \times P(2^{p_{k}})$. So $Opt'(G)$ is a spanning subgraph of $Opt(G)$. For any $1\le t\le k$ let 
$\langle Y_{t} \rangle = P(2^{p_{1}})\times P(2^{p_{2}})\times \cdots \times P(2^{p_{t}}) $, with these $p_{i}$.  We let $Y_{t} = \langle Y_{t-1} \rangle \times P(l)$, where 
$l$ is large enough.  Thus $Y_{t}$ is 
the $t$-dimensional 
grid $P(2^{e_{1}})\times P(2^{e_{2}-e_{1}})\times \cdots \times P(2^{e_{t-1}-e_{t-2}})\times P(l)$.  The grids $Y_{2}$, $Y_{3}$,$\ldots$ will be the 
aforementioned successive approximations to $Opt(G)$.  We will construct one to one maps $f_{i}:G\rightarrow Y_{i}$, $2\le i\le k$.  The final map $f_{k}$ 
will satisfy $f_{k}(G)\subseteq Opt'(G)\subseteq Opt(G)$.

For any point $x$ in a multidimensional grid, we let $x_{i}$ be its $i$'th coordinate (as suggested above), and when $i\le j$ we let $x_{i\rightarrow j}$ be the $(j-i+1)$-tuple 
$(x_{i}, x_{i+1}, \ldots,x_{j})$.  So for example, let $f:G\rightarrow H$ be a map where $G$ and $H$ are both multidimensional grids and $H$ is of dimension $r$.  
Let $f(x) = (b_{1}, b_{2}, \cdots , b_{r})\in V(H)$ for some $x\in V(G)$.  
Then by our notation $f(x)_{2} = b_{2}$, while $f(x)_{1\rightarrow i} = (b_{1}, b_{2}, \cdots ,  b_{i})$.  We can also express $dilation(f)$ as 
$dilation(f) =$ max $\{ \sum_{i=1}^{r}|f(x)_{i} -f(y)_{i}|: xy\in E(G)\}$.

For $1\le t < i$, a $t$-\textit{level} of $Y_{i}$ is any $t$-dimensional subgrid of 
$Y_{i}$ obtained by fixing the last $i-t$ coordinates of points in $Y_{i}$. Recalling that $Y_{i} = \langle Y_{i-1} \rangle \times P(l)$, note that there are $l$ pairwise disjoint $(i-1)$-levels of $Y_{i}$, each isomorphic to 
$\langle Y_{i-1} \rangle$, and we denote by $Y_{i}^{c}$ the $(i-1)$-level all of whose points have last (that is, $i$'th) coordinate $c$, $1\le c\le l$. 
We also let $Y_{i}^{(r)} = \bigcup_{c=1}^{r}Y_{i}^{c}$.  For fixed $i$ and $j$, $2\le i\le k$ and $j\geq 1$,   
we denote by $S_{i}^{j}$ the subgraph 
of $Y_{i}$ induced by the vertices 
$\{ (x_{1},x_{2},\ldots ,x_{i-1}, y)\in Y_{i}: 1\le x_{q}\le 2^{e_{q}-e_{q-1}}$ for $1\le q\le i-1, 1+(j-1)2^{e_{i}-e_{i-1}}\le y\le j2^{e_{i}-e_{i-1}} \}$, and call  $S_{i}^{j}$  
an \textit{i-section} of $Y_{i}$, more precisely, the $j$'th \textit{i-section}.    
Thus we have $S_{i}^{j}\cong \langle   Y_{i-1}    \rangle \times P(2^{e_{i} - e_{i-1}}) \cong \langle Y_{i} \rangle$ for each $j$, and we let $S_{i}^{(r)} = \bigcup_{j=1}^{r} S_{i}^{j}$.

As an example, let $G = [3\times 7\times 5\times 9]$.  Then $e_{1} = 2$, $e_{2} = 5$, $e_{3} = 7$, $\langle Y_{2} \rangle = [4\times 8]$, and $Y_{2} = 4\times l$ for large $l$.  Further, 
$Y_{2}^{3} = \{ (x_{1},3): 1\le x_{1} \le 4 \} $, 
so $Y_{2}^{3}$ can be thought of as the third column of $Y_{2}$, and $Y_{2}^{(3)} = \cup_{j=1}^{3}Y_{2}^{j}$ is the graph induced 
by the union of the first three columns of $Y_{2}$.  We have $S_{2}^{j}\cong \langle Y_{2} \rangle = [4\times 8]$ for each $j\geq 1$, and $S_{2}^{j}$ is the $[4\times 8]$ 
subgrid of $Y_{2}$ induced by columns $8(j-1)+1$ through $8j$.  We have $\langle Y_{3} \rangle = [4\times 8\times 4]$, $Y_{3} = [4\times 8\times l]$ for large $l$, and 
$Y_{3}^{c} = \{ (x_{1}, x_{2}, c): 1\le x_{1}\le 4, 1\le x_{2}\le 8 \}$ for any fixed integer $c \geq 1$.  So $Y_{3}^{c}$ is the $c$'th $2$-level of $Y_{3}$ (ordered by altitude, or third coordinate), 
and $Y_{3}^{(j)} =\cup_{t=1}^{j}Y_{3}^{t}$ is the $3$-dimensional subgrid of $Y_{2}$ induced by the first $j$ many $2$-levels of $Y_{3}$.  Further 
$S_{3}^{t}\cong [4\times 8\times 4]$ is the $3$-dimensional subgrid of $Y_{3}$ induced by the set of $2$-levels $\{ Y_{3}^{c}: 4(t-1)+1\le c\le 4t \}$.  So $S_{3}^{(j)} = \cup_{t=1}^{j}S_{3}^{t}$ is the 
$3$-dimensional subgrid $[4\times 8\times 4j]$ of $Y_{3}$ induced by the first $4j$ many $2$-levels of $Y_{3}$.              

We let $u_{i}(G) = \lceil \frac{|G|}{|\langle Y_{i-1} \rangle|} \rceil = \lceil \frac{|G|}{2^{e_{i-1}}} \rceil$, which we abbreviate with just $u_{i}$ when $G$ is 
understood by context or is an arbitrary grid.  Since each $(i-1)$-level 
of $Y_{i}$ has size $|\langle Y_{i-1} \rangle| = 2^{e_{i-1}}$, $u_{i}$ is the minimum number of $(i-1)$-levels of $Y_{i}$ whose union could contain 
the image $f_{i}(G)$.  Our maps $f_{i}$ will satisfy $f_{i}(G)\subseteq Y_{i}^{(u_{i})}$ for each $2\le i\le k$.

We will need the analogue of a $t$-level for $G =  [a_{1}\times a_{2}\times \cdots \times a_{k}  ]$.   For $2\le i\le k-1$, an \textit{i-page} of $G$ is any $i$-dimensional subgrid of $G$ induced by all vertices of 
$G$ having the same last $k-i$ coordinate values.  Let $P_{i} = a_{i+1}a_{i+2}\cdots a_{k}$, which is the number distinct $i$-pages in $G$.  We define a linear ordering  $\prec_{i}$ on these $i$-pages 
as follows.  Let $D_{i}$ and  $D_{i}'$ be two $i$-pages, with fixed last $k-i$ coordinate values $c_{i+1}, c_{i+2},\ldots ,c_{k}$ and $c_{i+1}', c_{i+2}',\ldots ,c_{k}'$ 
respectively.  Then $D_{i} \prec_{i} D_{i}' $ in this ordering if at the maximum index $r$, $i+1\le r\le k$, where $c_{r}\ne c_{r}'$ we have $c_{r} < c_{r}'$.  Now index the 
 $i$-pages of $G$ relative to this ordering by $D_{i}^{j}$, $1\le j\le P_{i}$, where $r < s$ if and only if $D_{i}^{r} \prec_{i} D_{i}^{s} $.  Also let $D_{i}^{(r)} = \bigcup _{j=1}^{r}D_{i}^{j}$.
 
 As an example, consider the $4$-dimensional grid $H = [3\times 7\times 3\times 2]$, containing $6$ many $2$-pages each isomorphic to $[3\times 7]$.  
 For fixed $i$ and $j$, $1\le i\le 3$ and $1\le j\le 2$,  
 denote by $D_{2}\langle i,j \rangle$ the $2$-page of $G$ given by $D_{2}\langle i,j \rangle = \{ (x_{1}, x_{2}, i, j)\in H: 1\le x_{1}\le 3, 1\le x_{2}\le 7 \}$.  Then the above ordering of $2$-pages of $H$ is given by 
 $D_{2}^{1} = D_{2}\langle 1,1 \rangle \prec_{2} D_{2}^{2} = D_{2}\langle 2,1 \rangle \prec_{2} D_{2}^{3} = D_{2}\langle 3,1 \rangle \prec_{2} D_{2}^{4} = D_{2}\langle 1,2 \rangle \prec_{2} 
 D_{2}^{5} = D_{2}\langle 2,2 \rangle \prec_{2} D_{2}^{6} = D_{2}\langle 3,2 \rangle $.  There are two $3$-pages in $H$, given by 
 $D_{3}^{1} = \{D_{2}^{1}, D_{2}^{2}, D_{2}^{3}\}$ and  $D_{3}^{2} = \{D_{2}^{4}, D_{2}^{5}, D_{2}^{6}\}$, and we have $D_{3}^{1} \prec_{3} D_{3}^{2}$.  As this example illustrates, for $3\le i\le k$ 
 the ordering $\prec_{i-1}$ is a refinement of the ordering $\prec_{i}$ in that if $D_{i-1}^{p}\subseteq D_{i}^{p'}$ and $D_{i-1}^{q}\subseteq D_{i}^{q'}$ with $p' < q'$, then $p < q$.    
 
From here on we fix $G = [a_{1}\times a_{2}\times \cdots \times a_{k}  ]$ to be a  $k$-dimensional grid.  We summarize the above notation, together with selected notation items to be introduced later, in Appendix 2 
 for convenient reference.

\section{The $2$-dimensional mapping}\label{the 2D mapping}

Recall the $2$-dimensional grid $Y_{2} = P(2^{e_1})\times P(l)$, with $e_{1} = \lceil log_{2}(a_1) \rceil$ and $l$ sufficiently large.  
In this section we construct a map $f_{2}: G\rightarrow Y_{2}^{(u_{2})}$, and we abbreviate $m = u_{2}$ throughout this section, so $m = \lceil \frac {|G|}{2^{e_{1}}} \rceil$.  So $f_{2}$ 
will satisfy $f_{2}(G)\subseteq P(2^{e_{1}})\times P(2^{\lceil log_{2}(|G|) \rceil - e_{1}})\subset Opt(G)$.
Additional work will show that  for any edge $vw\in E(G)$ we have that $|f_{2}(v)_{1} - f_{2}(w)_{1}| + |f_{2}(v)_{2} - f_{2}(w)_{2}|$ is small.
This map $f_{2}$, resembling a map 
constructed in \cite{Ch} and \cite{RS}, will be the first step 
in an inductive construction leading to a low dilation embedding $f_{k}:G\rightarrow Opt'(G)\subseteq Opt(G)$.

We use the following notation.  Let $G(r)$ denote the infinite 2-dimensional grid having $r$ rows, so the vertex and edge 
sets of $G(r)$ are $V(G(r)) = \{ (x,y)\in \mathbb{Z}^2: 1\le x\le r, 1\le y < \infty \}$, and
$E(G(r)) = \{(x_1, y_1)(x_2, y_2): |x_1- x_2| + |y_1 - y_2| = 1\}$.  We let $C_{i}$ denote the set of vertices $(x,y)$ of $G(r)$ 
with $x = i$, and refer to this set as ``chain $i$", or the ``$i$'$th$ chain" of $G(r)$.  

We can view 
$V(G)$ as a subset of $V(G(a_{1}))$ by a natural correspondence 
$\kappa: V(G)\rightarrow V(G(a_{1}))$ defined as follows.  Let $W_{i} = \prod_{t=2}^{i}a_t$ for $2\le i\le k$.   For any vertex $x = (x_1 , x_2, \ldots, x_k)$ of $G$, let 
$\kappa(x) = (x_1 ,y)$, where $y = (x_{k} - 1)W_{k-1} + (x_{k-1} - 1)W_{k-2} + \cdots + (x_{3} - 1)W_{2} + x_{2}$.  To see the action 
of $\kappa$, let $\rho_{t}$ be the subset of $V(G)$ consisting of vertices $x\in V(G)$ with $x_{1} = t$, $1\le t\le a_{1}$. Then $\kappa$ maps the points of 
$\rho_{t}$ to the first $W_{k}$ points of $C_{t}$ in lexicographic order; that is, if $z = (t, x_{2}, x_{3},\ldots, x_{k})$ and $z' = (t, x_{2}', x_{3}',\ldots, x_{k}')$ 
are two points of $\rho_{t}$, then $\kappa (z)_{1} = \kappa (z')_{1} = t$ and $\kappa (z)_{2} < \kappa (z')_{2}$ if and only if at the largest 
index $r$ where $x_{r}\ne x_{r}'$ we have $x_{r} < x_{r}'$.  Recall now the ordering $D_{i}^{j}$, $1\le j\le P_{i}$, of $i$-pages defined at the end of the last section.  Then 
for fixed $2\le i\le k$ and $1\le j\le P_{i}$, $\kappa (D_{i}^{j})$ is the subset of points in $G(a_{1})$ given by 
$\{ (t,y): 1\le t\le a_{1}, (j-1)a_{2}a_{3}a_{4}\cdots a_{i} + 1\le y\le ja_{2}a_{3}a_{4}\cdots a_{i}             \}$.

Our method in this section is to first construct a low dilation map $f:G(a_{1})\rightarrow Y_{2}$.  We then obtain the desired $f_{2}$ as the composition 
$f_{2} = f\circ \kappa : G\rightarrow Y_{2}^{(m)}$.  For the rest of this section we literally identify any point $x\in V(G)$ with $\kappa(x)$, dropping further references to 
$\kappa$ itself.  Thus, once $f:G(a_{1})\rightarrow Y_{2}$ is constructed, our map $f_{2}: G\rightarrow Y_{2}^{(m)}$ will henceforth be viewed as the restriction of 
$f$ to $G\subset G(a_{1})$ (under the identification $x\leftrightarrow \kappa(x)$).

The map $f:G(5)\rightarrow Y_{2}$ is shown in Figure  \ref{layout of G(5)}(b), and the map $f_{2}: [3\times 7\times 4\times a_{4}]\rightarrow Y_{2}$, obtained 
by restriction from $f:G(3)\rightarrow Y_{2}$, is shown in Figure \ref{3by7}.

To begin the description of $f$, let positive integers $1\le i\le a_{1}$ and $j\geq 1$ be given.  Then $f$ will map either $1$ or $2$ points of 
chain $C_i$ to $Y_{2}^{j}$ (which recall is column $j$ of $Y_2$).  We encode this information by defining a $(0,1)$ matrix $R$ having $a_1$ rows and infinitely 
many columns indexed by the positive integers, where $R_{ij} = 1$ (resp. $R_{ij} = 0$)  means that $C_i$ has $2$ (resp. $1$) image points in 
$Y_{2}^{j}$ under the map $f$.  In the first case these $2$ image points are successive in $Y_{2}^{j}$.   

Define the first column of $R$ by 
$$R_{i1} = \lfloor (\frac {2^{e_1} - a_{1}}{a_1})i \rfloor -  \lfloor (\frac {2^{e_1} - a_{1}}{a_1})(i-1) \rfloor, 1\le i \le a_{1}.$$
For $j > 1$, let $R_{i,j} = R_{i-1,j-1}$  where the row index is viewed modulo $a_1$. 
Thus $R$ is just a circulant matrix whose columns are obtained by successive 
downward shifts of the first column with wraparound, and is illustrated for $a_{1} = 5$ in Figure \ref{layout of G(5)}(a).   The following Lemma shows that the set of $1$'s in any set of consecutive entries of some row or column of $R$ depends only on 
the number of such entries, and is one of two successive integers depending on that number.

\begin{lemma}\label{matrix R properties}
The matrix $R$ has the following properties.

\noindent{\rm(a)} The sum of entries in any column of $R$ is $2^{e_1} - a_1$.

\noindent{\rm(b)} The sum of any $t$ consecutive entries in any row or column of $R$ is either $S_t$ or $S_t + 1$, where 
$S_t =  \lfloor (\frac {2^{e_1} - a_{1}}{a_1})t \rfloor$. 

\end{lemma}

\begin{proof} For (a), it suffices to prove the claim for column $1$ of $R$, since any other column of $R$ is just a circular 
shift of column $1$.  This column sum is $\sum_{i=1}^{a_1} (\lfloor (\frac {2^{e_1} - a_{1}}{a_1})i \rfloor -  \lfloor (\frac {2^{e_1} - a_{1}}{a_1})(i-1) \rfloor)$, 
which telescopes to $2^{e_1} - a_1$.

Consider (b).  From the circulant property of $R$ and the constant column sum property from part (a), 
it suffices to prove this claim for any sum of $t$ consecutive entries (mod $a_{1}$)  in column $1$, say a telescoping sum of the form
$$\sum_{i=r}^{r+t-1}R_{i1} =  \lfloor (\frac {2^{e_1} - a_{1}}{a_1})(r+t-1) \rfloor -  \lfloor (\frac {2^{e_1} - a_{1}}{a_1})(r-1) \rfloor.$$
Now letting $A = (\frac {2^{e_1} - a_{1}}{a_1})t$ and $B = (\frac {2^{e_1} - a_{1}}{a_1})(r-1)$, 
we see that this sum is $\lfloor A+B \rfloor - \lfloor B \rfloor.$  But any such difference is either $\lfloor A \rfloor$ or $\lfloor A \rfloor+1 $, 
so our sum is $S_t$ or $S_{t} +1$ as claimed.      \end{proof}

The basic idea in constructing $f$ is to fill out $Y_2^{(m)}$ with images  under 
$f$ column by column of $Y_2^{(m)}$ in increasing column index.  For any fixed column of $Y_2^{(m)}$, 
say the $j$'th column $Y_{2}^{j}$, 
we proceed down that column mapping points of $G(a_{1})$ in nondecreasing chain index.  The set $f(C_i)\cap Y_{2}^{j}$ 
will consist of either $1$ point or $2$ successive points of $Y_{2}^{j}$, depending on whether $R_{ij} = 0$ or $1$ 
respectively.    

The construction of $f$ follows.  While reading the construction below, the reader may wish to consult Figure \ref{layout of G(5)} illustrating the matrix $R$ and 
corresponding map $f: G(5)\rightarrow Y_{2}$.  
For $1\le r\le a_{1}$, we let $C^{(r)} = \bigcup_{i=1}^{r}C_{i}$.

\begin{center} \textbf{Construction of the map $f: G(a_{1})\rightarrow Y_{2}$}
\end{center}

\noindent{\textbf{1.}}The images of $f$ in $Y_{2}^{1}$ are given as follows.

\smallskip

\noindent{\textbf{(1a)}} Initial points of $C_{1}$ are mapped to $Y_{2}^{1}$ as follows.

\noindent If $R_{11} = 0 $, then $f(1,1) = (1,1)$ and $c_{1} \leftarrow 1$.  

\noindent If $R_{11} = 1$, then $f(1,1) = (1,1)$, $f(1,2) = (2,1)$, and $c_1 \leftarrow 2$.   

\smallskip

\noindent{\textbf{(1b)}} For $i=1$ to $a_{1} -1$

\noindent \textit{begin}

\noindent Assume inductively that  $f(C^{(i)})\cap Y_{2}^{1} = \{(x,1): 1\le x\le c_{i} \}$ for an integer $c_i$.  
Then initial points of $C_{i+1}$ are mapped to $Y_{2}^{1}$ as follows.

\noindent If $R_{11} = 0$, then $f(i+1,1) = (c_{i}+1,1)$, and $c_{i+1}\leftarrow c_{i} +1$. 
  
\noindent If $R_{11} = 1$, then $f(i+1,1) = (c_{i}+1,1)$, $f_2(i+1,2) = (c_{i}+2,1)$, and $c_{i+1}\leftarrow c_{i} +2$.   

\noindent If $i = a_{1} - 1$, then go to step 2.

\noindent \textit{end}   

\smallskip

\noindent{\textbf{2.}} For all $1\le t\le N$ (with $N$ sufficiently large), the images of $f$ in $Y_{2}^{t}$ are given inductively as follows.

\smallskip

\noindent For $j=1$ to $N$   

\noindent \textit{begin}

Suppose inductively that every point of $Y_{2}^{(j)}$ has been defined as an image under $f$.  
Suppose also that for each $1 \le i\le a_1$, 
we have $C_i\cap f^{-1}(Y_{2}^{(j)}) = \{(i,t): 1\le t \le n_{ij}\}$ for suitable integers $n_{ij}$, with $n_{0j} = n_{i0} = 0$.
\newline [Comment : So the first $n_{ij}$ points of $C_{i}$ have been mapped to $Y_{2}^{(j)}$.]
\smallskip

\noindent Now define $f(G(a_{1}))\cap Y_{2}^{j+1}$ as follows.

\noindent For $r=0$ to $a_{1} -1$

\noindent \textit{begin}

Suppose inductively that $f(C^{(r)})\cap Y_{2}^{j+1}$ is an initial segment of $Y_{2}^{j+1}$, say given by
\newline $f(C^{(r)})\cap Y_{2}^{j+1} = \{ (x,j+1): 1\le x \le c_{r,j+1} \}$, where $c_{0,j} = 0$ for all $j\geq 1$.
 
 \smallskip
 
 \noindent Define $f(C_{r+1})\cap Y_{2}^{j+1}$ as follows.

\smallskip

\noindent{\textbf{(2a)}} If $R_{r+1,j+1} = 0 $, then $f(r+1,1+n_{r+1,j}) = (c_{r,j+1}+1,j+1)$, and $c_{r+1,j+1}\leftarrow c_{r,j+1}+1$. 

\smallskip 

\noindent{\textbf{(2b)}} If $R_{r+1,j+1} = 1 $, and $j+1$ is even then 
$f(r+1,1+n_{r+1,j}) = (c_{r,j+1}+2,j+1)$, $f(r+1,2+n_{r+1,j}) = (c_{r,j+1}+1,j+1)$, and
$c_{r+1,j+1}\leftarrow c_{r,j+1}+2$. 

\smallskip

\noindent{\textbf{(2c)}} If $R_{r+1,j+1} = 1 $ and $j+1$ is odd, then $f(r+1,1+n_{r+1,j}) = (c_{r,j+1}+1,j+1)$, $f(r+1,2+n_{r+1,j}) = (c_{r,j+1}+2,j+1)$, and
$c_{r+1,j+1}\leftarrow c_{r,j+1}+2$.

\noindent \textit{end}

\noindent \textit{end}

\smallskip

Toward analyzing this construction, recall that $m = \lceil \frac {|G|}{2^{e_{1}}} \rceil$.  With the goal of showing (in the theorem which 
follows) that $f(G)\subseteq Y_{2}^{(m)}$, we analyze $ f^{-1}(Y_{2}^{(m)})$.  Let $C_{i}(t) = \{ (i,y): 1\le y\le t \}$ 
be the set of the first $t$ points of chain $C_{i}$.  Now by steps 2a and 2b, $f(C_{i})$ contributes either one point 
or two successive points to any column $Y_{2}^{j}$, depending on whether $R_{ij} = 0$ or $1$ respectively.  So $f(C_{i})$  
contributes exactly $j + \sum_{t=1}^{j}R_{it}$ points to $Y_{2}^{(j)}$.  Thus letting $N_{ij} = j + \sum_{t=1}^{j}R_{it}$, 
we have $ f^{-1}(Y_{2}^{(m)})\cap C_{i} = C_{i}(N_{im})$.  So let $G(a_{1},N_{im})$ be the subgraph of $G(a_{1})$ induced by $\bigcup_{i=1}^{a_{1}} C_{i}(N_{im})$.  
The next theorem gives various properties 
of $f$, including $Y_{2}^{(m-1)}\subset f(G)\subseteq f(G(a_{1},N_{im})) = Y_{2}^{(m)}$.

Finally, we define $f_{2}: G\rightarrow Y_{2}^{(m)}$ as the restriction of $f$ to the subgraph $G$ of $G(a_{1})$.  In Figure \ref{3by7} we illustrate 
part of $f_{2}(G)$ for $G = [3\times 7\times 4\times a_{4}]$ for some $a_{4} > 1$.  Each $2$-page $D_{2}^{i}$ 
of $G$ is isomorphic to $[3\times 7]$, and the images $f_{2}(D_{2}^{i})$, $1\le i\le 4$, are shown in detail with dividers separating 
the images $f_{2}(D_{2}^{i})$ and $f_{2}(D_{2}^{i+1})$ of successive $2$-pages .  Near these dividers, for each chain and each $2$-page 
 we have placed a box around the image of the chain's first point in that $2$-page.  
For example, under the letter $C$ are three boxed points, representing the images of the first 
points of each of the three chains in the third $2$-page $D_{2}^{3}$.  We have also labeled three points in the figure by their preimages in $G$.  For 
example $(2,4,4,...)$ indicates the preimage $(x_{1}, x_{2}, x_{3}, \cdots , x_{k})\in D_{2}^{4}\subset G$ with $x_{1} = 2$, $x_{2} = 4$, $x_{3} = 4$, and so on.   
There are four $2$-pages of $G$ in each $3$-page of $G$, and Figure \ref{3by7} includes 
the image $f_{2}(D_{3}^{1})$ of the first $3$-page, consisting of $\bigcup_{i=1}^{4}f_{2}(D_{2}^{i})$.  The beginnings of $f_{2}(D_{2}^{5})$, this being the first $2$-page 
of the $3$-page $D_{3}^{2}$, and of 
$f_{2}(D_{3}^{2}) = \bigcup_{i=5}^{8}f_{2}(D_{2}^{i})$ are also illustrated at the far right in the same figure.

\begin{theorem}\label{two dim properties}
The map $f: G(a_1)\rightarrow Y_2$ constructed above has the following properties.  Let $m = \lceil \frac {|G|}{2^{e_{1}}} \rceil$,
$S_{t} = \lfloor (\frac {2^{e_1} - a_{1}}{a_1})t \rfloor$
(as in Lemma \ref{matrix R properties}b), $N_{ij} = j + \sum_{t=1}^{j}R_{it}$, and let 
$G(a_{1},N_{im})$ be the subgraph of $G(a_{1})$ induced by $\bigcup_{i=1}^{a_{1}} C_{i}(N_{im})$.  

\smallskip

\noindent{\rm{(a)}} For each $i$ and $j$, $|f(C_{i})\cap Y_{2}^{j}| = 1$ or $2$, depending on whether $R_{ij} = 0$ or $1$ respectively. Further, 
if $R_{ij} = 1$, then $f(C_{i})\cap Y_{2}^{j}$ consists of two successive points of $Y_{2}^{j}$.  Also, $f$ is monotone in the sense 
that $f(i,j)_2 \le f(i,j+1)_2 \le f(i,j)_2 + 1$ for each $i$ and $j$.

\smallskip

\noindent{\rm{(b)}} Let $L_{r}(j) = f(C^{(r)})\cap Y_{2}^{j}$.  Then 
$L_{r}(j)$ is an initial segment, say $\{(d,j): 1\le d \le |L_{r}(j)|\}$, of $Y_{2}^{j}$, with $|L_{r}(j)| = r + \sum_{i=1}^{r} R_{ij}$. 

\smallskip

\noindent{\rm{(c)}} For $i,h,j\geq 1$ with $h\le j$, let $\pi(i,h\rightarrow j)$ be the number of points of $C_i$ mapped to 
columns $h$ through $j$ of $Y_{2}^{(m)}$.  Then $ \pi(i,1,j) = N_{ij}$.  
Further, for any $r,s\geq 1$ we have 
$\pi(i,r\rightarrow r+j) = j+1+S_{j+1}$ or $j+2+S_{j+1}$, and $|\pi(i,r\rightarrow r+j) -\pi(i,s\rightarrow s+j)|\le 1$.

\smallskip

\noindent{\rm{(d)}} For any $1\le r_{1} < r_{2}\le a_{1}$ we have $|N_{r_{1},j} - N_{r_{2},j}| \le 1$.  Also with $L_{r}(j)$ 
as in (b), for any $1\le j_{1} < j_{2}\le m$ we have $| |L_{r}(j_1)| -  |L_{r}(j_2)| |\le 1$, and $| |L_{r+1}(j_1)| -  |L_{r}(j_2)| |\le 2$.   

\smallskip

\noindent{\rm{(e)}} $f(G(a_{1},N_{im})) = Y_{2}^{(m)}$.

\smallskip

\noindent{\rm{(f)}} $G\subseteq G(a_{1},N_{im})$, and $Y_{2}^{(m-1)}\subset f(G)$.

\smallskip

\noindent{\rm{(g)}} For any $i$, $j$, and $r$ we have $|f(i,r)_{2} - f(j,r)_{2}| \le 1$.

\smallskip

\noindent{\rm{(h)}} For any $i$, $j$, and $r$ we have $|f(i,j)_{1} - f(i,r)_{1}| \le 2$.

\smallskip

\noindent{\rm{(i)}} Suppose that $|f(C_{r})\cap Y_{2}^{j}| = 2$.  Then 
$|L_{r}(j)| \geq |L_{r}(j+1)|$ and $N_{r,j}\geq N_{r+1,j}$.

\end{theorem}

We omit the involved but straightforward proof of this theorem here, and give it in Appendix 1.  The properties listed above 
 can be easily verified in the examples illustrated in Figures \ref{layout of G(5)} and \ref{3by7}.

The following corollary will be used later in proving our dilation bound and the containment $f_{2}(G)\subseteq Opt(G)$.  Its proof also appears 
in Appendix 1.

 \begin{corollary} \label{2-dim dilation}
 Let $v$ and $w$ be adjacent points of $G$.  The map $f_2: G\rightarrow Y_2^{(m)}$ has the following properties.
 
 \medskip
 
 \noindent{\rm{(a)}} $|f_{2}(v)_{1} - f_{2}(w)_{1}|\le 3$, and $|f_{2}(v)_{2} - f_{2}(w)_{2}|\le 1$.  
 
 \medskip
 
 \noindent{\rm{(b)}} $|f_{2}(v)_{1} - f_{2}(w)_{1}| + |f_{2}(v)_{2} - f_{2}(w)_{2}|\le 3$.
 
 \medskip
   
\noindent{\rm{(c)}} $f_{2}(G)\subseteq Opt(G)$.

\medskip

\noindent{\rm{(d)}} Let $T$ and $T'$ be segments of $p$ consecutive points on chains $C_i$ and $C_j$ respectively, 
$1\le i,j\le a_{1}$, where possibly $i = j$.  Let $c$ and $c'$ be the number of columns of $Y_2$ spanned by 
$f_{2}(T)$ and $f_{2}(T')$ respectively.  Then $|c - c'|\le 1$. 

\medskip

\noindent{\rm(e)} For $1\le r\le P_{2}$, let $r' = min\{ c: f_{2}(D_{2}^{(r)})\subseteq Y_{2}^{(c)} \}$.  Then $Y_{2}^{(r'-1)}\subset f_{2}(D_{2}^{(r)})$ and 
$|Y_{2}^{(r')} - f_{2}(D_{2}^{(r)})| < 2^{e_{1}}$.

 \end{corollary}

 \section{The idea of the general construction, with examples} \label{idea and example}
 
 In this section we give the idea behind our general construction, saving complete details and proofs of validity 
 for later sections.  We continue with the notation of Section \ref{notation}.
 
 The overall plan is to construct a sequence of maps $f_{i}: G\rightarrow [\langle Y_{i-1} \rangle \times P(u_{i})] = Y_{i}^{(u_{i})}$, $u_{i}= \lceil \frac{|G|}{|\langle Y_{i-1} \rangle|} \rceil = \lceil \frac{|G|}{2^{e_{i-1}}} \rceil$, $2\le i\le k$, 
 the first of which is $f_{2}$ from section \ref{the 2D mapping}.  Since $u_{i}\le 2^{\lceil log_{2}(|G|) \rceil - e_{i-1}}$, we will have 
$f_{i}(G)\subseteq Y_{i}^{(u_{i})}\subseteq \langle Y_{i-1} \rangle \times P(2^{\lceil log_{2}(|G|) \rceil - e_{i-1}})$.  The last graph is a spanning subgraph of $Opt(G)$, 
so $f_{i}(G)\subseteq Opt(G)$ for $2\le i\le k$. In particular, for $i=k$ 
we get  $f_{k}(G)\subseteq Opt'(G)$.  The maps $f_{i}$ will be successive approximations to $f_{k}$ in that for any $x\in V(G)$ 
and $2\le i <k$ we will have $f_{i}(x)_{1\rightarrow i-1} = f_{k}(x)_{1\rightarrow i-1}$.   
   
 The final map $f_{k}$ gives the basic geometry of our construction. We then apply a labeling $L$ of the points of $Opt'(G)$ with hypercube addresses from $Opt(G)$ 
to obtain the final embedding $H^{k}: G\rightarrow Opt(G)$, where $H^{k} = L\circ f_{k}$.  We construct the maps $f_{i}$ inductively, letting 
 $f_{i+1}$ be the composition $f_{i+1} = \sigma_{i}\circ I_{i}\circ f_{i}$, using maps $I_{i}$ and $\sigma_{i}$ described below.  Recall that $P_{i} = a_{i+1}a_{i+2}\cdots a_{k}$ is 
 the number of $i$-pages in $G$.

 So suppose $f_{i}: G\rightarrow Y_{i}^{(u_{i})}$ has been constructed, and we outline the construction of $f_{i+1}: G\rightarrow Y_{i+1}^{(u_{i+1})}$. 
 Now $I_{i}:Y_{i}^{(u_{i})}\rightarrow S_{i}^{(P_{i})}\subset Y_{i}$ is a one to one ``inflation" map which spreads out the image $f_{i}(G)$ ``evenly" in $S_{i}^{(P_{i})}$
 by successively ``skipping over" certain carefully chosen $(i-1)$-levels of $Y_{i}$ that are designated ``blank".  See Figure \ref{1thru12} for 
 an example where $i=2$, and where $f_{2}$ is the map of Figure \ref{3by7} and blank $1$-levels (columns) are shaded. 

 We let $S_{i}^{(P_{i})}(G)$ be the set of  
 points in $S_{i}^{(P_{i})}$ lying in nonblank $(i-1)$-levels (i.e. levels not designated ``blank") of $S_{i}^{(P_{i})}$.  We stipulate that  
 $(I_{i}\circ f_{i})(G)\subset S_{i}^{(P_{i})}(G)$.  The number of successive $i$-sections $S_{i}^{r}$ in the range of $I_{i}$ is the same as 
 the number of $i$-pages $D_{i}^{r}$ in $G$, each equalling $P_{i}$, and for each $r$ we associate $S_{i}^{r}$ with $D_{i}^{r}$ in a sense to be made clear below.

Define $I_{i}: Y_{i}^{(u_{i})}\rightarrow S_{i}^{(P_{i})}$ as follows, for now assuming that certain $(i-1)$-levels of $Y_{i}$ (all lying within $S_{i}^{(P_{i})}$ ) have been designated blank.  
 For $x = (x_{1}, x_{2},\ldots ,x_{i})\in Y_{i}^{(u_{i})}$, 
 we let $ I_{i}(x) =  (x_{1}, x_{2},\ldots , x_{i-1} ,x_{i}')$, where $x_{i}'$ is the common $i$-coordinate in the $x_{i}$'th nonblank $(i-1)$-level 
 of $Y_{i}$ (in order of increasing $i$-coordinate).  Now for any $S\subseteq Y_{i}^{(u_{i})}$, let $I_{i}(S) = \{I_{i}(s): s\in S \}$.  Observe that by its 
 definition, $I_{i}$ preserves $(i-1)$-levels; that is $I_{i}(Y_{i}^{t}) = Y_{i}^{t'}$ is an $(i-1)$-level of $Y_{i}$.  Also for $1\le s < t\le u_{i}$, where 
$I_{i}(Y_{i}^{s}) = Y_{i}^{s'}$ and $I_{i}(Y_{i}^{t}) = Y_{i}^{t'}$, we have $s' < t'$; that is, $I_{i}$ preserves the order (by increasing $i$-coordinate) of $(i-1)$-levels.  We can picture 
$I_{i}$ as an order preserving spreading out of the $u_{i}$ many $(i-1)$-levels of $Y_{i}^{(u_{i})}$ containing $f_{i}(G)$ among the $P_{i}2^{e_{i}-e_{i-1}}$ many 
$(i-1)$-levels of $S_{i}^{(P_{i})}$.  The image $I_{i}(Y_{i}^{(u_{i})})$ becomes the set of 
nonblank $(i-1)$-levels ($u_{i}$ of them) in $S_{i}^{(P_{i})}$.  The remaining $P_{i}2^{e_{i}-e_{i-1}} - u_{i}$ many $(i-1)$-levels of $S_{i}^{(P_{i})}$ are blank, 
and distributed among the nonblank $(i-1)$-levels so that certain balance properties outlined below are satisfied.  

Consider the example $G' = [3\times 7\times 4\times 3]$.  Start with $f_{2}: G'\rightarrow Y_{2}^{(u_{2}(G'))}$ given in the previous section, where $u_{2}(G') = \lceil \frac {252}{4} \rceil = 63$.  
Figure \ref{3by7} gives the initial part of $f_{2}(G')$, while Figures \ref{1thru12}a) and b), illustrate the map $I_{2}: Y_{2}^{(u_{2}(G'))}\rightarrow S_{2}^{(P_{2})}$, $P_{2} = 12$.  Each $2$-section $S_{2}^{j}$, $1\le j\le 12$, 
satisfies $S_{2}^{j}\cong \langle Y_{1} \rangle \times P(2^{e_{2}-e_{1}}) = P(4)\times P(8)$.  So $S_{2}^{(P_{2})}$
has $P_{2}2^{e_{2}-e_{1}} = 12\cdot 8 = 96$ columns (i.e. $(i-1$)-levels where $i=2$), of which the $63$ nonblank ones comprising $I_{2}(Y_{2}^{(63)})$  
contain $(I_{2}\circ f_{2})(G')$, while the remaining $33$ 
blank ones (shaded) contain no points of $(I_{2}\circ f_{2})(G')$.  Note in the figure how $I_{2}$ preserves the order in $S_{2}^{(12)}$   
of the $63$ columns in the domain $Y_{2}^{(63)}$ of $I_{2}$, and the $33$ blank 
columns (shaded) of  $S_{2}^{(12)}$ are distributed throughout $S_{2}^{(12)}$.

 Turning to arbitrary $G$, the quantity and distribution of blank $(i-1)$-levels within $S_{i}^{(P_{i})}$ will be such that for each $1\le r\le P_{i}$, the subgraph $S_{i}^{(r)}$ of $Y_{i}$ has 
 barely enough nonblank $(i-1)$-levels to host $(I_{i}\circ f_{i})(D_{i}^{(r)})$.  In effect, we want $S_{i}^{(r)}$ to have at least as much nonblank volume as $|D_{i}^{(r)}|$ for 
 each $r\geq 1$, but barely so in increments the size of an $(i-1)$-level in $Y_{i}$.  We formulate 
 this condition precisely as follows. Let $s_{i}(j)$ be the number of blank $(i-1)$-levels in section 
 $S_{i}^{j}$.  Observing that $|D_{i}^{j}| = a_{1}a_{2}\cdots a_{i}$ for any $j$, 
 that each $(i-1)$-level of $Y_{i}$ has $2^{e_{i-1}}$ points, and that each section $S_{i}^{j}$ 
 consists of $2^{e_{i}-e_{i-1}}$ many pairwise disjoint $(i-1)$-levels, we require that  
 
 \begin{equation}\label{overv just enough cols general}
\lceil \frac{ra_{1}a_{2}\cdots a_{i}}{2^{e_{i-1}}}  \rceil + \sum_{j=1}^{r} s_{i}(j)  = r2^{e_{i} - e_{i-1}}
\end{equation}
for each $1\le r\le P_{i}$.
  
 Next, the map $\sigma_{i}: (I_{i}\circ f_{i})(G)\rightarrow \langle Y_{i} \rangle \times P(u_{i+1}) = Y_{i+1}^{(u_{i+1})} $ ``stacks" 
 the sets $S_{i}^{j}\cap (I_{i}\circ f_{i})(G)$, $1\le j\le P_i$,
 over the single section $S_{i}^{1} \cong \langle Y_{i}    \rangle$ as follows.   Let $x=(x_{1}, x_{2},\ldots,x_{i})\in S_{i}^{j}\cap(I_{i}\circ f_{i})(G)$, $1\le j\le P_{i}$. 
Let  $\bar{x_{i}}\equiv x_{i}$ (mod $2^{e_{i} - e_{i-1}}$), $1\le \bar{x_{i}} \le 2^{e_{i} - e_{i-1}}$, be the congruence class of $x_{i}$ mod $2^{e_{i} - e_{i-1}}$.  
 Define the first 
 $i$ coordinates of $\sigma_{i}(x)$ by $\sigma_{i}(x)_{1\rightarrow i} = (x_{1}, x_{2},\ldots, x_{i-1}, \bar{x_{i}})$.  
 Observe that since $x\in S_{i}^{j}$, we have $x_{i} = (j-1)2^{e_{i} - e_{i-1}} + \bar{x_{i}}$, and that $\sigma_{i}(x)_{1\rightarrow i}\in S_{i}^{1}$.   
 To get the $i+1$'st (and last) coordinate, 
 let $c$ be the number of points $y=(y_{1}, y_{2},\ldots,y_{i})\in S_{i}^{t}\cap(I_{i}\circ f_{i})(G)$, $1\le t\le j$, satisfying $\sigma_{i}(y)_{1\rightarrow i} = \sigma_{i}(x)_{1\rightarrow i}$.
 Then define $\sigma_{i}(x) = (x_{1}, x_{2},\ldots,x_{i-1},\bar{x_{i}}, c)$. 
 
 Finally define $f_{i+1}$ as the composition $f_{i+1} = \sigma_{i} \circ I_{i} \circ f_{i}$.    
 
 For a fixed point $(x_{1}, x_{2},\ldots,x_{i-1},\bar{x_{i}})\in S_{i}^{1}$, we 
 view the set of images $(x_{1}, x_{2},\ldots,x_{i-1},\bar{x_{i}}, c)$ 
 under $\sigma_{i}$ as a stack addressed by $(x_{1}, x_{2},\ldots,x_{i-1},\bar{x_{i}})$ 
 extending into the $(i+1)$'st dimension. Thus each point $(x_{1}, x_{2},\ldots,x_{i-1},\bar{x_{i}})$ of $S_{i}^{1}$ 
 becomes the address of such a stack, and the image $\sigma_{i}(x) = (x_{1}, x_{2},\ldots,x_{i-1},\bar{x_{i}}, c)$ is the $c$'th point ``up" in this stack.     
 See Figure \ref{stack4sections}, to which we return later with a full explanation, for an initial look at $\sigma_{2}$.

 Since the domain of $\sigma_{i}$ is $(I_{i}\circ f_{i})(G)$, which is a set contained in the collection of nonblank $(i-1)$-levels of $S_{i}^{(P_{i})}$, it follows that  
 the points in blank $(i-1)$-levels of $S_{i}^{(P_{i})}$ make no contribution under the map $\sigma_{i}$ to the aforementioned stacks.  We can 
 picture the images $\sigma_{i}(x)$, $x\in S_{i}^{j}\cap (I_{i}\circ f_{i})(G)$, as ``falling through" blank $(i-1)$-levels $Y_{i}^{d}\subset S_{i}^{t}, t< j$, with $d\equiv x_{i}$ (mod $2^{e_{i}-e_{i-1}}$).  
 But if such a $Y_{i}^{d}$ is nonblank, and if $z = (x_{1}, x_{2}, \cdots , x_{i-1}, d)\in Y_{i}^{d}\cap (I_{i}\circ f_{i})(G)$, then $\sigma_{i}(z)_{1\rightarrow i} = \sigma_{i}(x)_{1\rightarrow i}$.  Thus 
 $z$ does contribute (under $\sigma_{i}$) to the same stack as $x$; that is, both $\sigma_{i}(z)$ and $\sigma_{i}(x)$ belong to the stack addressed by 
 $(x_{1}, x_{2},\ldots,x_{i-1},\bar{x_{i}})\in S_{i}^{1}$.  Further, $\sigma_{i}(z)_{i+1} < \sigma_{i}(x)_{i+1}$ by definition of $c$ and since $t<j$; that is, $\sigma_{i}(z)$ 
 appears ``below" $\sigma_{i}(x)$ in this stack.  So we see that each section $S_{i}^{j}$, $1\le j < P_{i}$, contributes (via $\sigma_{i}$) either $0$ or $1$ point
to the stack $(x_{1}, x_{2},\ldots,x_{i-1},\bar{x_{i}})\in S_{i}^{1}$.  It contributes $1$ point precisely when its unique $(i-1)$-level $Y_{i}^{d}\subset S_{i}^{j}$ 
satisfying $d\equiv \bar{x_{i}}$ (mod $2^{e_{i}-e_{i-1}}$) is not blank and the point $z = (x_{1}, x_{2}, \cdots , x_{i-1}, d)\in Y_{i}^{d}$ of $S_{i}^{j}$ lies in $(I_{i}\circ f_{i})(G)$.

 As an example, consider $G'' = [3\times 7\times 4]\subset G'$.  We construct $f_{3}(G'') = (\sigma_{2}\circ I_{2}\circ f_{2})(G'')$.  Note that  $e_{1} = 2$, $e_{2} = 5$,  $P_{2} = 4$, 
 $u_{2}(G'') = \lceil \frac{84}{4} \rceil = 21$, 
 and $u_{3}(G'') = \lceil \frac{84}{32} \rceil = 3$.  Again each $2$-section $S_{2}^{j}$, $1\le j\le P_{2} = 4$, is isomorphic to $P(2^{e_{1}}) \times P(2^{e_{2}-e_{1}}) = P(4)\times P(8)$, and 
 $S_{2}^{(P_{2})}$ has $P_{2}\cdot 2^{e_{2}-e_{1}} = 4\cdot 8 = 32$ columns.

 \noindent \underline{Step 1}: Start with $f_{2}: G''\rightarrow Y_{2}^{(u_{2}(G''))} = Y_{2}^{(21)}$ given in section \ref{the 2D mapping} and illustrated as $f_{2}(D_{3}^{1})$ in Figure \ref{3by7}. 
 
 \noindent \underline{Step 2}: Perform the map $I_{2}: Y_{2}^{(21)}\rightarrow S_{2}^{(P_{2})} = S_{2}^{(4)}$ by letting $I_{2}(a,b) = (a,b')$, where 
 $b'$ is the $b$'th nonblank column of $S_{2}^{(P_{2})}$, ordered by increasing second coordinate.  The result is shown in  Figure \ref{1thru12}a), where 
 $I_{2}(f_{2}(G''))$ is unshaded and blank columns are shaded.  The choice of blank columns will be discussed later.
 
 On comparing with Figure \ref{3by7}, we see that $I_{2}(f_{2}(G''))$ is indeed obtained from $f_{2}(G'')$ by skipping over the designated 
blank columns.  Each $2$-section $S_{2}^{i}\subset S_{2}^{(4)}$, $1\le i\le 4$, is the host for the $2$-page 
$D_{2}^{i}$ of $G''$, with the possible exception of the images of first 
points of chains in $D_{2}^{i}$ (these being boxed in Figure \ref{1thru12}) which may appear 
in the last nonblank column of $S_{2}^{i-1}$ (instead of in $S_{2}^{i}$).  
 
 \noindent \underline{Step 3}: Perform the stacking map $\sigma_{2}:(I_{2}\circ f_{2})(G'')\rightarrow Y_{3}^{(u_{3}(G''))} = Y_{3}^{(3)}$.  Thus for each 
 $(a,b)\in S_{2}^{j}\cap (I_{2}\circ f_{2})(G'')$, we let $\sigma_{2}((a,b)) = (a,\overline {b}, c)$, where 
$\overline {b} \equiv y$ (mod $8$), $1\le \overline {b}\le 8$, 
and $c = |\{(a,b')\in S_{2}^{t}\cap (I_{2}\circ f_{2})(G''): t\le j, ~b' \equiv b$ (mod $8)      \}|$.  
 As above, we view $(a,\overline {b})$ as the address in $S_{2}^{1}$ of the stack on which 
$(a,b)$ has been placed by $\sigma_{2}$, and $\sigma_{2}((a,b))$ is the $c$'th point ``up" on this stack. 

In Figure \ref{stack4sections} we illustrate $f_{3}(G'') = \sigma_{2}((I_{2}\circ f_{2})(G''))$; that is, how $\sigma_{2}$ stacks the  
four sets $S_{2}^{j}\cap (I_{2}\circ f_{2})(G'')$, $1\le j\le 4$,  
 over $S_{2}^{1} = \langle Y_{2} \rangle$ to yield $f_{3}(G'')$.  At top center we begin with $I_{2}(Y_{2}^{(21)})\subseteq S_{2}^{(P_{2})} = S_{2}^{(4)}$ where   
the four $2$-sections $S_{2}^{j}$, $1\le j\le 4$ comprising $S_{2}^{(4)}$, are placed vertically in succession for convenience.  We now perform the stacking, with the result shown at lower left in the figure.    
Here we regard the bottom $4\times 8$ layer (of the three layers in the result) as a copy of $S_{2}^{1}$, each of whose $32$ points is 
the address of a stack.  Consider the $4$ points lying in column $2$ of this $S_{2}^{1}$.  As shown in the figure, each of these points is the address of a stack 
of height $3$.  For brevity let us write $\sigma_{2}(S_{2}^{j})$ for $\sigma_{2}(S_{2}^{j}\cap (I_{2}\circ f_{2})(G''))$.  
The contributions to any one of these $4$ stacks 
 come from $\sigma_{2}(S_{2}^{1})$, $\sigma_{2}(S_{2}^{3})$, and $\sigma_{2}(S_{2}^{4})$ (indicated respectively by the labels $1,3,4$ in these stacks).  
To see how this happens, refer back to the top center of this figure and look at the points in column $2$ in the $4$ sections $S_{2}^{j}$, $1\le j\le 4$, 
being stacked.  Since column $2$ of the second of these, $S_{2}^{2}$, is blank, the points in it make no contribution to the stacks whose addresses lie in 
  column $2$ of $S_{2}^{1}$ at lower left.  But the points in column $2$ of $S_{2}^{1}$, $S_{2}^{3}$, and $S_{2}^{4}$ (at top center) lie in 
nonblank columns and are contained in $(I_{2}\circ f_{2})(G'')$.  Hence they do  contribute to the column $2$ stacks of $S_{2}^{1}$ 
at lower left.    
We can view 
such points in column $2$ of $S_{2}^{3}$ and $S_{2}^{4}$  in the top center as ``falling through" blank column $2$ in $S_{2}^{2}$ under the action of $\sigma_{2}$. 
Similarly every point 
in column $1$ of $S_{2}^{1}= \langle Y_{2} \rangle$ at lower left is the address of a stack of height $2$. The contributions to these stacks are from 
$\sigma_{2}(S_{2}^{2})$, $\sigma_{2}(S_{2}^{3})$, indicated by labels $2$ and $3$ in these stacks.  These contributions   
``fall through" blank column $1$ of $S_{2}^{1}$ (at top center) under the action of $\sigma_{2}$.  The lower right of the figure shows how 
individual image points are affected by this stacking.  For example, the images under $\sigma_{2}$ of  paths in  
$S_{2}^{2}\cap (I_{2}\circ f_{2})(G'')$ (in bold at top center) jump between levels of the final result at lower right.    
Note also that the maximum stack height is indeed $u_{3}(G'') = 3$, achieved at stacks addressed by points in columns $2,3,5,6$, and $8$ of $S_{2}^{1}$.  So $f_{3}(G'')\subseteq Y_{3}^{(3)}$.

In the rest of this section we show how to assign blank columns in $S_{2}^{(P_{2})}$ (used in constructing $I_{2}$ and then $f_{3}$), and 
generally how to assign blank $(i-1)$-levels in $S_{i}^{(P_{i})}$ (used in constructing $I_{i}$ and then $f_{i+1}$) for an arbitrary multidimensional grid $G$.  We will see that 
the ability to make this assignment is implied by the existence of a certain class of $(0,1)$ matrices, whose construction we give in the next section.  For now we describe conditions 
which our assignment will satisfy that are sufficient for making our embedding $f_{k}$ (and later $H^{k}$) have the required dilation and containment properties.  To start, consider 
the assignment of blank columns in $S_{2}^{(P_{2})}$.  
\smallskip
\newline \textbf{(a)}  The number $s_{2}(j)$ of blank columns in each $2$-section $S_{2}^{j}$, $1\le j\le P_{2}$, is chosen so that after $(I_{2}\circ f_{2})(G)$ ``skips over'' these columns, 
each subgraph $S_{2}^{(r)}$ (the union of the first $r$ many 2-sections), $1\le r\le P_{2}$, has barely 
enough nonblank columns to host the image $(I_{2}\circ f_{2})(D_{2}^{(r)})$ of the first $r$ many $2$-pages of $G$.  This is expressed in equation (\ref{overv just enough cols general}) in 
the case $i=2$, which says that $\lceil \frac{ra_{1}a_{2}}{2^{e_{1}}}  \rceil + \sum_{j=1}^{r} s_{2}(j) = r2^{e_{2} - e_{1}}$
for each $1\le r\le P_{2}$.  Thus each image $(I_{2}\circ f_{2})(D_{2}^{j})$ lies in 
its own section $S_{2}^{j}$, apart from images of first points of chains in $D_{2}^{j}$ as explained previously and shown in Figure \ref{1thru12}.     
In the example $G'' = [3\times 7\times 4]$, the sequence $s_{2}(j)$ can be found recursively 
from  equation (\ref{overv just enough cols general}) (with $i=2$) to be $s_{2}(1) = 2$, $s_{2}(2) = 3$, $s_{2}(3) = 3$, and $s_{2}(4) = 3$, and these 
numbers of blank columns are shown in the four sections respectively in Figure \ref{1thru12}a).
\smallskip

\noindent \textbf{(b)} The blank columns are distributed over the various $2$-sections $S_{2}^{j}$ so that for $xy\in E(G)$ the contribution to $dist_{Y_{k}}(f_{k}(x),f_{k}(y))$ from $|f_{k}(x)_{2} - f_{k}(y)_{2}|$ is small.  
As background, note that $dist_{Y_{k}}(f_{k}(x),f_{k}(y)) = \sum_{i=1}^{k}|f_{k}(x)_{i} - f_{k}(x)_{i}|$.  It turns out that $f_{k}$ will satisfy 
$|f_{k}(x)_{i} - f_{k}(x)_{i}| = |f_{i+1}(x)_{i} - f_{i+1}(x)_{i}| = |(I_{i}\circ f_{i})(x)_{i} - (I_{i}\circ f_{i})(y)_{i}|$ for each $1\le i\le k-1$, where the last difference is 
taken mod $2^{e_{i}-e_{i-1}}$.  The last equality for $i=2$ 
can be verified in the construction of $f_{3}(G'')$ and figures above. 

Focusing on the case $i=2$, suppose that $x$ and $y$ agree in their first two coordinates.  So $x$ and $y$ are corresponding points in their respective 
$2$-pages, say $x\in D_{2}^{s}$ and $x\in D_{2}^{t}$.  Now we wish to keep  $|(I_{2}\circ f_{2})(x)_{2} - (I_{2}\circ f_{2})(y)_{2}|$ small mod $2^{e_{2}-e_{1}}$.  
Thus we want $(I_{2}\circ f_{2})(x)$ to be skipping over roughly the same number of blank columns in $S_{2}^{s}$ as does $(I_{2}\circ f_{2})(y)$ in $S_{2}^{t}$. 
To accomplish this, 
given that $s$ and $t$ are arbitrary as are $x$ and $y$ as corresponding points, we will require a strong balance, over all $2$-sections $S_{2}^{j}$, in the frequency of blank columns in any initial segment 
of columns in $S_{2}^{j}$.  

A similar balance will be required in the frequency of blank ($i-1$)-levels in any initial segment of ($i-1$)-levels of any $i$-section $S_{i}^{j}$.  A precise 
formulation of this requirement with additional detail will be described below.

\smallskip

\noindent \textbf{(c)} The assignment of blank columns across all $2$-sections $S_{2}^{j}$, $1\le j\le P_{2}$, is uniformly distributed mod $2^{e_{2}-e_{1}}$.  This condition 
allows us to stack the sections $S_{2}^{j}$ on top of each other, over a single $2$-section $S_{2}^{1} = \langle Y_{2} \rangle$, so that the stacks addressed in   
$S_{2}^{1}$ have roughly equal stack heights.  In the example $G'' = [3\times 7\times 4]$  illustrated in Figure \ref{1thru12}a, note 
that the $32$ columns in $2$-sections $S_{2}^{j}$, $1\le j\le 4$, can be partitioned into congruence classes mod $8$ by column number, and each congruence class 
has exactly 1 or 2 of its columns designated blank.  The result is that stack heights differing by at most $1$.

\medskip
The goals described above in (a)-(c) can be formulated as combinatorial conditions to be satisfied by the designation of blank columns, and generally of blank ($i-1$)-levels.  
Starting with the designation of blank columns, it will be convenient to define a  $P_{2}\times 2^{e_{2}-e_{1}}$, $(0,1)$ matrix $F(2) = (f_{cd}(2))$. The rows of $F(2)$ correspond to the 
$2$-sections $S_{2}^{c}$, $1\le c\le P_{2}$, and the columns of $F(2)$ to the columns ($2^{e_{2}-e_{1}}$ of them) within each $2$-section.     
 We let $f_{cd}(2) = 1$ if column $d$ in section $S_{2}^{c}$ (which recall is column $Y_{2}^{(c-1)(2^{e_{2}-e_{1}})+d}$ in $Y_{2}$)
 is blank, and $f_{cd}(2) = 0$ if that column is nonblank.  Since $s_{2}(c)$ is 
the number of blank columns in $S_{2}^{c}$, the sum of entries in row $c$ of $F(2)$ is 
\begin{equation}\label{overv row sums}
\sum_{d=1}^{2^{e_{2}-e_{1}}}f_{cd}(2) = s_{2}(c).
\end{equation}
for $1\le c\le P_{2}$.

Toward formulating the goal expressed in (c), consider now the contribution, through the stacking map $\sigma_{2}$, from $S_{2}^{(r)}$ to 
any stack addressed by a point in $S_{2}^{1}$.  For any stack address $(x,y)\in S_{2}^{1}$, let 
$Stack_{3}((x,y),r)) = \{ \sigma_{2}(z):  \sigma_{2}(z)_{1\rightarrow 2} = (x,y), z\in S_{2}^{(r)}  \}$, which is the set 
of points in the stack addressed by  $(x,y)$ whose preimages under the map $\sigma_{2}$ come from $S_{2}^{(r)}$.  Now for any $\sigma_{2}(z)\in Stack_{3}((x,y),r))$, 
we have $z = (x,d)$, with $d\equiv y$ (mod $2^{e_{2}-e_{1}})$, and column $Y_{2}^{d}$ is nonblank.  Then for $r < P_{2}$ 
we see that the stack height $|Stack_{3}((x,y),r))|$ is the 
number of zeros of the matrix $F(2)$ lying in column $y$ and within rows $1$ through $r$ .  So to keep stack heights 
nearly equal over all stack addresses $(x,y)\in S_{2}^{1}$, we require that this number of zeros is nearly the same over all columns 
$y$ in $F(2)$.  For this, it suffices to have the number of $1$'s in rows $1$ through $r$ of any column nearly the same; that is, to have nearly 
equal initial column sums.  This becomes the condition    
\begin{equation}\label{overv initial column sums}
|\sum_{c=1}^{r}f_{cy}(2) - \sum_{c=1}^{r}f_{cy'}(2)| \le 1.
\end{equation}
for any $1\le y,y'\le 2^{e_{2}-e_{1}}$ and $1\le r\le P_{2}$.  It says that the blank columns are uniformly distributed mod $2^{e_{2}-e_{1}}$.

To formulate (b), for integers $a$ and $b$ let $||a-b||$ be the difference $a-b$ taken mod $2^{e_{2}-e_{1}}$.  
Let $xy\in E(G)$, say with $x\in D_{2}^{s}$ and $y\in D_{2}^{t}$, $s\ne t$.  Since $x$ and $y$ agree in their first two coordinates,  
 $x$ and $y$ are each the $p$'th points in $D_{2}^{s}$ and $D_{2}^{t}$ 
of their respective 
chains, for some $1\le p\le a_{2}$.  By equation (\ref{overv just enough cols general}) for $i=2$, $S_{2}^{(r)}$ has 
barely enough nonblank columns to contain $(I_{2}\circ f_{2})(D_{2}^{(r)})$ for any $1\le r\le P_{i}$.  Thus by the monotonicity 
property of $f_{2}$ in Theorem \ref{two dim properties}a, for any $j\geq 1$, the image under $I_{2}\circ f_{2}$ of the first point of 
any chain in $D_{2}^{j}$ must lie either in the last nonblank column of $S_{2}^{j-1}$ or the first nonblank column of $S_{2}^{j}$.  Now let $T$ (resp. ($T'$)) be the set of 
the first $p$ points in $D_{2}^{s}$ (resp. $D_{2}^{t}$) in the chain containing $x$ (resp. $y$).    
Also let 
$c$ (resp. $c'$) be the number of columns of $Y_{2}$ spanned by $f_{2}(T)$ (resp. $f_{2}(T')$), and thus the number 
of nonblank columns of $Y_{2}$ spanned by $(I_{2}\circ f_{2})(T)$ (resp. $(I_{2}\circ f_{2})(T')$).  By Corollary \ref{2-dim dilation}d we have $|c-c'|\le 1$.  
Thus the contribution to $||(I_{2}\circ f_{2})(x)_{2} - (I_{2}\circ f_{2})(y)_{2}||$ from the difference between the number of 
nonblank columns in $S_{2}^{s}$ (resp. $S_{2}^{t}$) preceding $(I_{2}\circ f_{2})(x)$ (resp. $(I_{2}\circ f_{2})(y)$) is small (in fact $\le 1$).  The same contribution due 
to the difference in starting columns of $(I_{2}\circ f_{2})(T)$ and $(I_{2}\circ f_{2})(T')$ is also small ($\le 1$) by the above.  Thus 
 $||(I_{2}\circ f_{2})(x)_{2} - (I_{2}\circ f_{2})(y)_{2}||$ depends primarily on the number                  
 $N_{1}$ (resp. $N_{2}$) of blank columns   
in $S_{2}^{s}$ (resp. $S_{2}^{t}$) preceding the column containing $(I_{2}\circ f_{2})(x)$ (resp. $(I_{2}\circ f_{2})(y)$).  Each column 
counted by $N_{1}$ (resp. $N_{2}$) pushes the image $(I_{2}\circ f_{2})(x)$ (resp. $(I_{2}\circ f_{2})(y)$) one more column to the right in $S_{2}^{s}$ (resp. $S_{2}^{t}$).  
So we want to keep 
$|N_{1}-N_{2}|$ small.    
Since each blank column corresponds to a $1$ in $F(2)$, we see that each of 
$N_{1}$ and $N_{2}$ is just an initial row sum in $F(2)$ (row $s$ for $N_{1}$ and row $t$ for $N_{2}$).  These considerations motivate the goal 
of keeping the difference between corresponding initial row sums in $F(2)$ small.  
It will suffice for our purposes to have 
\begin{equation}\label{overv initial row sums}
 |\sum_{j=1}^{b} f_{sj}(2) - \sum_{j=1}^{b} f_{tj}(2)|\le 2.
 \end{equation}
 for any, $1\le s,t\le a_{3}a_{4}\cdots a_{k}$, $1\le b\le 2^{e_{2}-e_{1}}$.
 
 For fixed $i$, we will see that the sequence $\{ s_{i}(j) \}$, $1\le j\le P_{i}$, recursively defined by (\ref{overv just enough cols general}) satisfies $|s_{i}(j_{1}) - s_{i}(j_{2})|\le 1$ 
 for all $1\le j_{1}, j_{2}\le P_{i}$.  Applying this to the case $i=2$, we can view the satisfying of conditions (a)-(c), as relying on the construction of a 
 $P_{2}\times 2^{e_{2}-e_{1}}$, $(0,1)$ matrix $F(2)$ 
 with prescribed row sums $s_{2}(c)$, $1\le c\le P_{i}$, these sums differing by at most $1$ (from the preceding sentence and (\ref{overv row sums})).  Further, $F(2)$ will have balanced initial column sums and balanced initial row 
 sums (from (\ref{overv initial column sums}) and 
 (\ref{overv initial row sums})).  
 
 Such an $F(2)$ for the embedding $f_{3}: G'' = [3\times 7\times 4] \rightarrow Y_{3}^{(3)}$ discussed above, where $F(2)$ has $P_{2} = 4$ rows, is illustrated by the $(0,1)$ matrix 
in the right of Table 1a).  The fractional matrices at left from which this and the other $(0,1)$ matrices in this table are derived will be explained later.  This $F(2)$ encodes which columns 
 are designated `blank' in performing the inflation step $I_{2}: Y_{2}^{(u_{2}(G''))} \rightarrow S_{2}^{(4)}$.  The stacking map $\sigma_{2}$ 
  is then applied to yield the final embedding $f_{3} = \sigma_{2}\circ I_{2}\circ f_{2}: [3\times 7\times 4] \rightarrow Y_{3}^{(3)}.$

The corresponding requirements for arbitrary dimension $i\geq 2$ are analogous.  We return to the construction of $f_{i+1}$ from $f_{i}$ as the composition 
 $f_{i+1} = \sigma_{i}\circ I_{i}\circ f_{i}$.  Recall that there are $P_{i} = a_{i+1}a_{i+2}\ldots a_{k}$ many $i$-sections $S_{i}^{j}$, $1\le j\le P_{i}$, and each such $i$-section has $2^{e_{i}-e_{i-1}}$ many 
 $(i-1)$-levels.  So let $F(i) = (f_{cd}(i))$ be the $P_{i}\times 2^{e_{i}-e_{i-1}}$, $(0,1)$ matrix (analogous to $F(2)$) defined by $ f_{cd}(i) = 1$ 
 if the $d$'th $(i-1)$-level of $i$-section $S_{i}^{c}$ (that is, $(i-1)$-level $Y_{i}^{(c-1)2^{e_{i}-e_{i-1}} + d}$ of $Y_{i}$ ) is blank, and $ f_{cd}(i) = 0$ otherwise.  So we require the analogues of the 
 relations (\ref{overv row sums} - \ref{overv initial row sums});
 
\begin{equation}\label{overv row sums general}
\sum_{j=1}^{2^{e_{i}-e_{i-1}}}f_{cj}(i) = s_{i}(c),
\end{equation}
for $1\le c\le P_{i}$,

\begin{equation}\label{overv initial column sums general}
|\sum_{c=1}^{r}f_{cy}(i) - \sum_{c=1}^{r}f_{cy'}(i)| \le 1,
\end{equation}
for $1\le y,y'\le 2^{e_{i}-e_{i-1}}$, $1\le r\le P_{i}$, 

\begin{equation}\label{overv initial row sums general}
 |\sum_{d=1}^{r} f_{sd}(i) - \sum_{d=1}^{r} f_{td}(i)|\le 2,
 \end{equation}
 for $1\le s,t\le P_{i}$, $1\le r\le 2^{e_{i}-e_{i-1}}$.

 Relation (\ref{overv row sums general}) says that there are $s_{i}(c)$ many blank $(i-1)$-levels in $S_{i}^{c}$.  Relation 
 (\ref{overv initial column sums general}) gives balanced initial column sums in $F(i)$.  This implies that blank 
 $(i-1)$-levels are uniformly distributed mod $2^{e_{i}-e_{i-1}}$.  This ensures 
 balanced stack heights,  under the map $\sigma_{i}$, for stacks addressed by $S_{i}^{1} \cong \langle Y_{i}   \rangle$.   
  Finally 
 (\ref{overv initial row sums general}) will imply (after some work) that $||(I_{i}\circ f_{i})(x)_{i} - (I_{i}\circ f_{i})(y)_{i}||$ is small for any 
 $xy\in E(G)$, with $x\in D_{i}^{s}$ and $x\in D_{i}^{t}$, $1\le s,t\le P_{i}$.  This will make the contribution to $dist_{Y_{k}}(f_{k}(x),f_{k}(x))$ from 
 $|f_{k}(x)_{i} - f_{k}(y)_{i}|$ small for corresponding points ($x$ and $y$) in distinct $i$-pages.  Applying this requirement for all $i$ will keep $dilation(f_{k})$ small.

 We are thus reduced to the construction of a $(0,1)$ matrix $F(i)$ for each $2\le i\le k-1$, with prescribed row 
 sums $s_{i}(j)$ (with values recursively computed using (\ref{overv just enough cols general})), $1\le j\le P_{i}$, differing by at most $1$, and 
 balanced initial column and row sums as required in (\ref{overv initial column sums 
 general}) and (\ref{overv initial row sums general}).  We construct such matrices in the next section.

Returning to $G' = [3\times 7\times 4\times 3]$ we illustrate 
 the construction of $f_{4}:G'\rightarrow Y_{4}^{(u_{4}(G'))} = Y_{4}^{(2)}$. 
 We have $u_{2}(G') = \lceil  \frac{|G'|}{2^{e_{1}}} \rceil = \lceil  \frac{252}{4} \rceil = 63$, $u_{3}(G') = \lceil  \frac{|G'|}{2^{e_{2}}} \rceil  = \lceil \frac{252}{32} \rceil = 8$,
and $u_{4}(G') = \lceil \frac{|G'|}{2^{e_{3}}} \rceil = \lceil \frac{252}{128} \rceil = 2$.  
 
The maps $f_{2}, f_{3}$, and $f_{4}$ are constructed in succession.  The map $f_{2}:G'\rightarrow Y_{2}^{(63)}$ is given in the previous section.  To build 
 $f_{3} = \sigma_{2}\circ I_{2}\circ f_{2}$ the next step is to define the inflation map $I_{2}$.  For this, note that $P_{2} = 12$, so there 
 will be $12$ many $2$-sections $S_{2}^{j}$, $1\le j\le 12$, each isomorphic to $P(4)\times P(8)$.  Thus $I_{2}$ has the form 
 $I_{2}: Y_{2}^{(63)}\rightarrow S_{2}^{(12)}$.  The sequence $ \{ s_{2}(j) \}$, $1\le j\le 12$,  
is calculated inductively using (\ref{overv just enough cols general}) for $i=2$, with  
 $e_{1} = 2$ and $e_{2} = 5$, to give  $ 2,3,3,3,2,3,3,3,2,3,3,3 $. A balanced distribution of blank columns 
among the sections $S_{2}^{j}$, $1\le j\le 12$ satisfying (\ref{overv just enough cols general}) - (\ref{overv initial row sums}) 
is given by the $12\times 8$, $(0,1)$ matrix $F(2)$ shown at right in Table $1$b).  So $I_{2}$ will make make $f_{2}(G')$ skip 
over the blank columns in $S_{2}^{(12)}$ as in the previous example. 
The blank columns (shaded) among the first 
$4$ sections $S_{2}^{1}$, $S_{2}^{2}$, $S_{2}^{3}$, $S_{2}^{4}$, as encoded by the first $4$ 
rows of the matrix in the right of Table $1$b), are illustrated in Figure \ref{1thru12}a.  The blank columns among the 
remaining $8$ sections $S_{2}^{j}$, $5\le j\le 12$, (encoded by the last $8$ rows of the same matrix) are illustrated in Figure  \ref{1thru12}b), where these $8$ sections are surrounded 
by a box.  We put small boxes 
around images of first points of chains within each $2$-page.

$$
\begin{array}{c}
\\ \\
\left[ \begin{array}{cccc}
\mathbf{1/4} & \mathbf{1/4} & \ldots & \mathbf{1/4}\\
\mathbf{3/8} & \mathbf{3/8} & \ldots & \mathbf{3/8}\\
\mathbf{3/8} & \mathbf{3/8} & \ldots & \mathbf{3/8}\\
\mathbf{3/8} & \mathbf{3/8} & \ldots & \mathbf{3/8}\\
\end{array}\right] \\

\end{array}
\begin{array}{c}
row\\ sum\\
2 \\ 3\\ 3\\  3\\ \end{array}
\begin{array}{c}
\\  \\ \\
Knuth-like  \\ rounding\\ \rightarrow \\  \\ \end{array}
\begin{array}{c}
\\ \\
\left[ \begin{array}{cccccccc}
1 & 0 & 0 & 0 & 1 & 0 & 0 & 0\\
0 & 1 & 0 & 1 & 0 & 0 & 1 & 0\\
0 & 0 & 1 & 0 & 0 & 1 & 0 & 1\\
1 & 0 & 0 & 1 & 0 & 0 & 1 & 0\\
\end{array}\right]\\

\end{array}
$$
\begin{center} a) Matrix $F(2)$ encoding blank columns for the embedding $f_{3}: [3\times 7\times 4] \rightarrow Y_{3}$ \end{center}
$$
\begin{array}{c}
\left[
\begin{array}{cccc}
\mathbf{1/4} & \mathbf{1/4} & \ldots  & \mathbf{1/4}\\
\mathbf{3/8} & \mathbf{3/8} & \ldots  & \mathbf{3/8}\\
\mathbf{3/8} & \mathbf{3/8} & \ldots  & \mathbf{3/8}\\
\mathbf{3/8} & \mathbf{3/8} & \ldots  & \mathbf{3/8}\\ \hline
\mathbf{1/4} & \mathbf{1/4} & \ldots  & \mathbf{1/4}\\
\mathbf{3/8} & \mathbf{3/8} & \ldots  & \mathbf{3/8}\\
\mathbf{3/8} & \mathbf{3/8} & \ldots  & \mathbf{3/8}\\
\mathbf{3/8} & \mathbf{3/8} & \ldots  & \mathbf{3/8}\\ \hline
\mathbf{1/4} & \mathbf{1/4} & \ldots  & \mathbf{1/4}\\
\mathbf{3/8} & \mathbf{3/8} & \ldots & \mathbf{3/8}\\
\mathbf{3/8} & \mathbf{3/8} & \ldots & \mathbf{3/8}\\
\mathbf{3/8} & \mathbf{3/8} & \ldots & \mathbf{3/8}\\
\end{array}
\right] \\

\end{array}
\begin{array}{c}
2\\ 3\\ 3\\ 3\\ 2\\ 3\\ 3\\ 3\\ 2\\ 3\\ 3\\ 3\\ 
\end{array}
\begin{array}{c}
\\ \\ \\ \\
Knuth-like  \\ rounding\\ \rightarrow \\ \\ \\ \\ \\ \end{array}
\begin{array}{c}
\left[ \begin{array}{cccccccc}
1 & 0 & 0 & 0 & 1 & 0 & 0 & 0\\
0 & 1 & 0 & 1 & 0 & 0 & 1 & 0\\
0 & 0 & 1 & 0 & 0 & 1 & 0 & 1\\
1 & 0 & 0 & 1 & 0 & 0 & 1 & 0\\ \hline
0 & 1 & 0 & 0 & 0 & 1 & 0 & 0\\
0 & 0 & 1 & 0 & 1 & 0 & 0 & 1\\
1 & 0 & 0 & 1 & 0 & 0 & 1 & 0\\
0 & 1 & 0 & 0 & 1 & 0 & 0 & 1\\ \hline
0 & 0 & 1 & 0 & 0 & 1 & 0 & 0\\
0 & 1 & 0 & 1 & 0 & 0 & 0 & 1\\
1 & 0 & 0 & 0 & 1 & 0 & 1 & 0\\
0 & 1 & 1 & 0 & 0 & 1 & 0 & 0\\
\end{array}\right]\\

\end{array}
$$

%$$
%\begin{array}{c}
%\left[ \begin{tabular}{cccc}
%\mathbf{1/4} & \mathbf{1/4} & \ldots & \mathbf{1/4}\\
%\mathbf{3/8} & \mathbf{3/8} & \ldots & \mathbf{3/8}\\
%\mathbf{3/8} & \mathbf{3/8} & \ldots & \mathbf{3/8}\\
%\mathbf{3/8} & \mathbf{3/8} & \ldots & \mathbf{3/8}\\
%\end{tabular}\right] \\
%\begin{array}{cccc}
%5/8 & 5/8 & \ldots & 5/8\\
%\end{array}
%\end{array}
%$$
\begin{center} b) Matrix $F(2)$ encoding blank columns for the embedding $f_{3}: [3\times 7\times 4\times 3 ] \rightarrow Y_{3}$ \end{center}

$$
\begin{array}{c}
\left[ \begin{array}{cccc}
\mathbf{1/4} & \mathbf{1/4} & \mathbf{1/4} & \mathbf{1/4}\\
\mathbf{1/4} & \mathbf{1/4} & \mathbf{1/4} & \mathbf{1/4}\\
\mathbf{1/2} & \mathbf{1/2} & \mathbf{1/2} & \mathbf{1/2}\\
\end{array}\right] \\

\end{array}
\begin{array}{c}
1 \\ 1\\ 2\\  \end{array}
\begin{array}{c}
Knuth-like  \\ rounding\\ \rightarrow \\  \end{array}
\begin{array}{c}
\left[ \begin{array}{cccc}
1 & 0 & 0 & 0\\
0 & 0 & 1 & 0\\
0 & 1 & 0 & 1\\
\end{array}\right]\\

\end{array}
$$
\begin{center} c) Matrix $F(3)$ encoding blank $2$-levels for the embedding $f_{4}: [3\times 7\times 4\times 3 ] \rightarrow Y_{4}$ \end{center}

\begin{center} Table 1: Matrices encoding blank levels \end{center}

Next we apply the stacking map 
$\sigma_{2}: (I_{2}\circ f_{2})(G') \rightarrow  Y_{3}^{(u_{3}(G'))} = Y_{3}^{(8)}$ as defined above, thereby yielding 
the map $f_{3} = \sigma_{2}\circ I_{2}\circ f_{2}: G'\rightarrow Y_{3}^{(8)}$.  Note that $\sigma_{2}$ stacks the 12 many $2$-sections of Figure \ref{1thru12} onto $S_{2}^{1}$, 
with the result shown in Figure \ref{7and12stacking}.  The left subfigure shows the stacking of the first $7$ sections, while 
the right subfigure all $12$ sections.

Again for each stack address $(x,y)\in S_{2}^{1} \cong  \langle Y_{2} \rangle$ we indicate by label $j$, given to various cross sections of this stack, 
which image sets $\sigma_{2}(S_{2}^{j})$, $1\le j\le 12$, 
contribute points to this stack. 
 For example 
one can check in the right subfigure that for any stack address $(x,2)\in S_{2}^{1}$, $1\le x\le 4$, (so  $(x,2)$ lies in column $2$ of $S_{2}^{1}$) 
the image sets $\sigma_{2}(S_{2}^{j})$ 
contributing to $Stack_{3}((x,2),12)$ satisfy  
$j = 1,3,4,6,7,9,11 $ in order of increasing stack height.  To see why, refer to Figure \ref{1thru12}.  There you see that points in column $2$ 
of $S_{2}^{3}$ and $S_{2}^{4}$ ``fall through" blank column $2$ of $S_{2}^{2}$, points in column $2$ of $S_{2}^{6}$ and $S_{2}^{7}$ ``fall through" 
blank columns 2 in $S_{2}^{5}$ and $S_{2}^{2}$, points in column $2$ of $S_{2}^{9}$ ``fall through" 
blank columns 2 in $S_{2}^{8}$, $S_{2}^{5}$, and $S_{2}^{2}$, and so on.      
The resulting stack height is $|Stack_{3}((x,2),12)| = 7$.   By contrast, for any stack address $(x,5)\in  S_{2}^{1}$, $1\le x\le 4$, 
the sections contributing to such a stack are $j =  2,3,4,5,7,9,10,12 $ in order of increasing stack height, and 
the resulting stack height is $|Stack_{3}((x,5),12)| = 8$.  In fact the maximum stack height over all 32 stacks addressed by points of $S_{2}^{1}$ is $8$.

A more detailed look at the members of various stacks addressed by points of $S_{2}^{1}$ is given in Figure \ref{columns 4 and 5}.  Here we identify the points $x\in G'$ such that $(\sigma_{2} \circ I_{2}\circ f_{2})(x)$ belongs 
to stack addresses in columns $4$ and $5$ of $S_{2}^{1}$.  For example, consider the members 
of $Stack_{3}((3,4),7)$ (bolded in the figure), consisting by definition of those images $(\sigma_{2} \circ I_{2}\circ f_{2})(x)$ 
satisfying $(\sigma_{2} \circ I_{2}\circ f_{2})(x)_{1\rightarrow 2} = (3,4)\in S_{2}^{1}$, and $(I_{2}\circ f_{2})(x)\in S_{2}^{j}$ for some $1\le j\le 7$.  Here we must 
have $(I_{2}\circ f_{2})(x) = (3,d)$, where $d\equiv 4$(mod $8$) by the definition of $\sigma_{2}$ and $8(j-1)+1\le d\le 8j$ since $(I_{2}\circ f_{2})(x)\in S_{2}^{j}$.  We find such an $x$ when $j=1$, 
with $(I_{2}\circ f_{2})(x) = (3,4)\in S_{2}^{1}$, and Figure \ref{1thru12} shows that $x = (2,4,1,1)$.  Here $x_{1\rightarrow 3} = (2,4,1)$ since $x$ is the $4$'th  point on chain $2$ of $D_{2}^{1}$, 
while $x_{4} = 1$ since $x\in D_{3}^{1}$.  
There is no such $x$ when 
$j=2$ since column $4$ of $S_{2}^{2}$ is blank.  We find such an $x$ when $j=3$, thus satisfying $(I_{2}\circ f_{2})(x) = (3,20)\in S_{2}^{3}$, 
and the figure shows that $x = (2,4,3,1)$ since $x$ is point (2,4) in $D_{2}^{3}$.  The remaining two such points $x$ are found similarly; $x = (2,3,1,2)\in D_{2}^{5}$ 
with $(I_{2}\circ f_{2})(x) = (3,36)\in S_{2}^{5}$, and $x = (2,4,2,2)\in D_{2}^{6}$  
with $(I_{2}\circ f_{2})(x) = (3,44)\in S_{2}^{6}$.  For the last $x$ note that the first point of  $D_{2}^{6}$ in the same chain as $x$ is $(2,1,2,2)$ and $(I_{2}\circ f_{2})((2,1,2,2))\in S_{2}^{5}$.

To construct $f_{4} = \sigma_{3}\circ I_{3}\circ f_{3}$, we begin with the parameters needed to define $I_{3}$.  
Observe that $P_{3} = 3$ and $e_{3} = 7$, so $e_{3}-e_{2} = 2$ 
and $S_{3}^{j} = P(4) \times P(8) \times P(4)$ for each $1\le j\le 3$.  So $I_{3}$ has the form $I_{3}: Y_{3}^{(u_{3}(G'))} = Y_{3}^{(8)}\rightarrow S_{3}^{(P_{3})} = S_{3}^{(3)}$.  
Since $|D_{3}^{j}| = 84$, we 
get by (\ref{overv just enough cols general}) the sequence $s_{3}(1) =1$, $s_{3}(2) =1$, and $s_{3}(3) =2$.  A balanced distribution of blank $2$-levels (among 
the sections $S_{3}^{j}$) satisfying  (\ref{overv just enough cols general}) together with (\ref{overv row sums general}) - (\ref{overv initial row sums general}) is given 
by the $3\times 4$, $(0,1)$ matrix $F(3)$ given in Table 1c. 
We then apply the map $I_{3}: Y_{3}^{(8)}\rightarrow S_{3}^{(3)}$, which distributes the eight $2$-levels in $Y_{3}^{(8)}$ among the twelve $2$-levels of 
$S_{3}^{(3)}$, leaving four of the latter twelve levels blank.

The result is represented in the left side of Figure \ref{stacking in 4D} as the set of sections $S_{3}^{j}$, 
$1\le j\le 3$, each having four $2$-levels, with blank $2$-levels shaded.  For example, $S_{3}^{1}$ has $2$-level 1 as blank (as specified by 
row 1 of $F(3)$), $S_{3}^{2}$ has $2$-level 3 as blank (as specified in row 2), and $S_{3}^{3}$ has $2$-levels $2$ and $4$ as blank 
(as specified in row 3).

Now we apply the 
stacking map $\sigma_{3}: S_{3}^{(3)}\cap (I_{3}\circ f_{3})(G')\rightarrow Y_{4}^{(u_{4}(G'))} = Y_{4}^{(2)}$ (whose image recall we abbreviate as $\sigma_{3}(S_{3}^{(3)})$), 
which stacks each of the sets $S_{3}^{j}\cap (I_{3}\circ f_{3})(G')$, $1\le j\le 3$, in succession onto 
$S_{3}^{1}\cong \langle Y_{3}   \rangle$ as defined previously.  
One can check that the maximum  stack height $|Stack_{4}(z,3)|$ over all $128$ stack addresses $z\in \langle Y_{3}  \rangle$ is $2$, again either by checking 
that each column of $F(3)$ has at most (in fact exactly) $2$ zeros, or by seeing 
in Figure \ref{stacking in 4D} that each of the $128$ stack addresses in $ \langle Y_{3}  \rangle$ receives at most $2$ points 
under the map $\sigma_{3}$.  In the remainder of Figure \ref{stacking in 4D} we illustrate the stacking map $\sigma_{3}$ in stages.  
First $S_{3}^{2}$ is stacked on top of $S_{3}^{1} \cong \langle Y_{3} \rangle$, with the result that $64$ of the stack addresses $z$ 
(those lying in levels $2$ and $4$ of $ \langle Y_{3} \rangle$) so far have $4$-dimensional stack height 2 (i.e. $|Stack_{4}(z, 2)| = 2$ for these $z$), while 
$64$ of the stack addresses $z$ (those lying in levels $1$ and $3$ of $ \langle Y_{3} \rangle$) 
so far have $4$-dimensional stack height 1 (i.e. $|Stack_{4}(z, 2)| = 1$ for these $z$).  Next, $S_{3}^{3}$ is stacked on top of the previous stacking, so
that now every stack address has $4$-dimensional stack height $2$  (i.e. $|Stack_{4}(z, 3)| = 2$ for all $z\in \langle Y_{3} \rangle$).

Consider the final result $\sigma_{3}(S_{3}^{(3)})$, yielding the map $f_{4}$, illustrated in the upper right of Figure \ref{stacking in 4D}.  There, 
focus on the stack addresses $(x,y,z)\in  \langle Y_{3} \rangle$ with 
$z=2$, $1\le x\le 4$, $1\le y\le 8$, these lying in the second $2$-level of  $ \langle Y_{3} \rangle$ (since $z=2$).   
The sets $\sigma_{3}(S_{3}^{j})$, 
$1\le j\le 3$, contributing points to any such stack $Stack_{4}((x,y,2), 3)$ satisfy $j=1$ or $2$, where in Figure \ref{stacking in 4D} the contribution 
of $\sigma_{3}(S_{3}^{1})$ is represented by $1b$, and of $\sigma_{3}(S_{3}^{2})$ by $2b$.

 In Figure \ref{detail in 4D} we look  in detail at the individual stacks $Stack_{4}((3,1,2),3)$, $Stack_{4}((3,4,2),3)$, and 
 $Stack_{4}((1,4,2),3)$.  For 
example, the members of $Stack_{4}((3,1,2),3)$ are the two 
images $(\sigma_{3} \circ I_{3} \circ f_{3})(\alpha)$, indicated by their preimages $\alpha \in G'$ which are 
$\alpha = (3,2,2,1)$ and $\alpha = (2,1,4,2)$ in increasing order of stack height.

 \section{Tools for the general construction}
 
In this section we develop two tools used in our general construction: 
\newline (1) the designation of blank $(i-1)$-levels in $i$-sections $S_{i}^{j}$, $1\le j\le P_{i}$, and
\newline (2) the construction of an ordering (by consecutive integer labels) of the vertices of any hypercube such that for any reasonably long interval of successive label values, any two vertices 
whose labels lie in that interval are at fairly small hypercube distance.

 \subsection{The Construction of Blank Levels}
 
In this subsection we describe the sequence $\{ s_{i}(j) \}$, $1\le i\le k$, $1\le j\le P_{i}$, where 
 $ s_{i}(j)$ is the number of $(i-1)$-levels of $S_{i}^{j}$ that are designated blank 
 under the map $I_{i}\circ f_{i}$.  We show that this sequence satisfies equation (\ref{overv just enough cols general}).  That is,  
 for each $1\le r\le P_{i}$ we show that $S_{i}^{(r)}$ has just enough nonblank $(i-1)$-levels to contain $(I_{i}\circ f_{i})(D_{i}^{(r)})$.  Finally, we construct for each 
 $j$ the actual set of $ s_{i}(j)$ many $(i-1)$-levels in $S_{i}^{j}$ that are designated blank, and show that the 
 required properties (\ref{overv row sums general})-(\ref{overv initial row sums general}) are satisfied. This construction is based on a theorem of Knuth on simultaneous 
 roundings of sequences.

 Given $G = [a_{1}\times a_{2}\times \cdots \times a_{k}  ]$ and $1\le i\le k-1$, define the sequence  $\{ s_{i}(j) \}$, with $1\le j\le P_{i}$, by
 
\begin{equation}\label{blank formula}
s_{i}(j) = 2^{e_{i} - e_{i-1}} - \lceil \frac{a_{1}a_{2}\cdots a_{i}}{2^{e_{i-1}}} \rceil + \lfloor  j\phi_{i} \rfloor - \lfloor  (j-1)\phi_{i} \rfloor 
\end{equation}
 where $\phi_{i} = \lceil \frac{a_{1}a_{2}\cdots a_{i}}{2^{e_{i-1}}} \rceil -  \frac{a_{1}a_{2}\cdots a_{i}}{2^{e_{i-1}}}$.

 \begin{lemma} \label{blank arithmetic}
 
 Let  $P_{i} = a_{i+1}a_{i+2}\cdots a_{k}$, $1\le r\le P_{i}$, and $1\le i\le k$.  
 
\noindent \rm{(a)} The sequence $\{ s_{i}(j)$\},  $1\le j\le P_{i}$, defined 
above satisfies $\lceil  \frac{ra_{1}a_{2}... a_{i}} {2^{e_{i-1}}}  \rceil + \sum_{j=1}^{r}s_{i}(j)  = r 2^{e_{i}-e_{i-1}}$.  
In particular, taking $s_{i}(j)$ to be the number of blank $(i-1)$-levels of $S_{i}^{j}$, then the number of nonblank $(i-1)$-levels in 
$S_{i}^{(r)}$ is $\lceil  \frac{ra_{1}a_{2}... a_{i}} {2^{e_{i-1}}}  \rceil = \lceil  \frac{|D_{i}^{(r)}|} {2^{e_{i-1}}}  \rceil $.

\noindent \rm{(b)} $s_{i}(j) = p$ or $p+1$, where $p = 2^{e_{i} - e_{i-1}} - \lceil \frac{a_{1}a_{2}\cdots a_{i}}{2^{e_{i-1}}} \rceil$.  Also 
$\frac{s_{i}(j)}{2^{e_{i} - e_{i-1}}}\le \frac{1}{2}$.

 \end{lemma}     
 
 \begin{proof}For (a), observe that the sum $\sum_{j=1}^{r}s_{i}(j)$ telescopes.  Note also that $r\lceil  \frac{a_{1}a_{2}... a_{i}} {2^{e_{i-1}}}  \rceil$ is an integer.  Thus 
$\sum_{j=1}^{r}s_{i}(j)  =  r2^{e_{i} - e_{i-1}} -  r\lceil  \frac{a_{1}a_{2}... a_{i}} {2^{e_{i-1}}}  \rceil + \lfloor  r \lceil  \frac{a_{1}a_{2}... a_{i}} {2^{e_{i-1}}}  \rceil -  r  \frac{a_{1}a_{2}... a_{i}} {2^{e_{i-1}}}  \rfloor 
= r2^{e_{i} - e_{i-1}} + \lfloor -r  \frac{a_{1}a_{2}... a_{i}} {2^{e_{i-1}}}  \rfloor  =  r2^{e_{i} - e_{i-1}} -  \lceil r  \frac{a_{1}a_{2}... a_{i}} {2^{e_{i-1}}}  \rceil $.  Part (a) follows.

Consider (b).  Since $\phi_{i}<1$ we have $0\le \lfloor  j\phi_{i} \rfloor - \lfloor  (j-1)\phi_{i} \rfloor \le 1$, giving the first statement.    
Since 
$a_{1}a_{2}\cdots a_{i}\geq 2^{e_{i}-1}+1$, we get $s_{i}(j)\le 2^{e_{i} - e_{i-1}} - \lceil \frac{2^{e_{i}-1}+1}{2^{e_{i-1}}} \rceil +1\le 2^{e_{i} - e_{i-1}-1}$.  
The second statement follows.       \end{proof}

We now specify, for each $2\le i\le k-1$ and $1\le j\le P_{i}$, which $ s_{i}(j) $ out of the $2^{e_{i}-e_{i-1}}$ many $(i-1)$-levels in $S_{i}^{j}$ will be designated blank. 
Keep in mind that this designation must satisfy the balance properties (\ref{overv row sums general})-(\ref{overv initial row sums general})  we required in the overview.  For this, we need the following theorem of Knuth \cite{KN}.

Let $x_{1}, \ldots, x_{n}$ be a sequence of reals, and $\gamma$ a permutation of $\{ 1,2,\ldots,n \}$.  Let $S_{k} = x_{1} + \cdots + x_{k}$ and 
$\Sigma_{k} = x_{\gamma(1)} + \cdots +  x_{\gamma(k)}$ be the partial sums for these two independent orderings of the $x_{i}$'s.  Consider a \textit{rounding} of 
the $x_{i}$'s; that is, a designation of integers $\bar{x}_{i}$ satisfying $\lfloor  x_{i}  \rfloor \le  \bar{x}_{i} \le    \lceil x_{i} \rceil$ for $1\le i\le n$.  Now 
let the corresponding partial sums be $\bar{S}_{k} =  \bar{x}_{1} + \cdots + \bar{x}_{k}$  and  
 $\overline{\Sigma}_{k} = \bar{x}_{\gamma(1)} + \cdots + \bar{x}_{\gamma(k)}$.  We say that a rounding of the $x_{i}$'s is \textit{consistent} with 
 the original sequence $\{ x_{i} \}$ (resp. with the permuted sequence under $\gamma$) if $\lfloor S_{k}  \rfloor \le \bar{S}_{k} \le \lceil S_{k}   \rceil $ 
 (resp. $\lfloor \Sigma_{k}  \rfloor \le \overline{\Sigma}_{k} \le \lceil \Sigma_{k}   \rceil $) for each $1\le k\le n$.  We also say that the rounding  
 is a \textit{two-way} rounding if it is simultaneously consistent with both the original sequence and the permuted sequence under $\gamma$.

 \begin{theorem} \label{Knuth} \cite{KN} For any finite sequence $x_{1},\ldots, x_{n}$ of reals and any permutation $\gamma$ of $\{ 1,2,\ldots,n\}$, there is 
 a two-way rounding of the $x_{i}$.       
 
 \end{theorem}
 
The existence of two-way roundings was shown earlier by Spencer \cite{SP} by probabilistic methods, as a corollary to more general 
results on the discrepancy of set systems \cite{LSV}.  Knuth's network flow based proof of Theorem \ref{Knuth}, omitted here, is constructive and yields improved error bounds.  The two-way 
rounding produced is not necessarily unique.

As a consequence of Knuth's theorem, we obtain the following theorem (from \cite{D1} and \cite{D2}) on roundings of matrices which are consistent with respect to all 
initial row and column sums.  We give its short proof for completeness.  This extends the rounding lemma of Baranyai (\cite{Bar}) giving such consistency with respect 
to all row sums, to all column sums, and to the sum of all matrix entries.  Additional results on roundings of matrices, including extensions of previous work, applications to digital 
halftoning, and improved running time and error bounds in implementation can be found in the work 
of Doerr (\cite{D3}, \cite{D1}, \cite{D2} and others), Asano (\cite{As1} and \cite{As2}), and Wright \cite{SW} among others.

 \begin{theorem} \label{Doerr's theorem}
 Let $T = (t_{ij})$ be an $m\times n$ matrix with $0\le t_{ij}\le 1$ for all $i$ and $j$.  Then there exists an $m\times n$, $(0,1)$
 ``rounding" matrix $F = (f_{ij})$ of $T$; that is, where $f_{ij} = \lfloor t_{ij}  \rfloor$ or $\lceil t_{ij}  \rceil$, satisfying the following properties.
 
 \medskip
 
 \noindent{\rm{(a)}} For each $b$ and $i$, $1\le b\le n$, $1\le i\le m$, we have $|\sum_{j=1}^{b}(t_{ij}-f_{ij})| < 1$.
 
 \medskip
 
 \noindent{\rm{(b)}} For each $b$ and $j$, $1\le b\le m$, $1\le j\le n$, we have $|\sum_{i=1}^{b}(t_{ij}-f_{ij})| < 1$.
 
 \medskip
 
 \noindent{\rm{(c)}} $|\sum_{i=1}^{m} \sum_{j=1}^{n} (t_{ij}-f_{ij})| < 1$.

 \end{theorem} 
 
 \begin{proof} For the most part we paraphrase proofs in \cite{D1} and \cite{D2}.  First consider parts (a) and (b).  We construct an $(m+1)\times(n+1)$ matrix $Y$ from $T$ which has 
 integral row and column sums by appending to each row and each column a last entry in $[0,1)$ just large 
 enough to make that row or column have an integer sum.  Specifically, let $y_{ij} = t_{ij}$ for $1\le i\le m$ and 
 $1\le j\le n$, and then $y_{m+1,j} = \lceil \sum_{i=1}^{m}t_{ij}  \rceil - \sum_{i=1}^{m}t_{ij}$ for $1\le j\le n$,  
 and $y_{i,n+1} = \lceil \sum_{j=1}^{n}t_{ij}  \rceil - \sum_{j=1}^{n}t_{ij}$ for $1\le i\le m$, and $y_{m+1,n+1} = \sum_{i=1}^{m} \sum_{j=1}^{n} t_{ij}$.  Then $Y$ has integral row and column sums.
 
 Consider now the following two orderings of the entries of $Y$.  First we order these entries by ``row-major" order; that is, first the 
 entries of row $1$, then those of row $2$, etc., until row $m+1$, and within any row $i$ place $y_{ij}$ ahead of $y_{ir}$ iff $j<r$.  Similarly consider the ``column-major" order where 
 we place $y_{ij}$ before $y_{rs}$ iff either $j<s$, or $j=s$ and $i<r$. 
 
Applying Knuth's theorem, there is a (not necessarily unique) two-way rounding matrix $\overline{Y} = \bar{y}_{ij}$ relative to these two orders.  Since every row and 
 column sum of $Y$ is already an integer, we get $\sum_{i=1}^{p}\sum_{j=1}^{n+1}(y_{ij} - \bar{y}_{ij}) = 0$ for any $1\le p\le m+1$, and
 $\sum_{j=1}^{q}\sum_{i=1}^{m+1}(y_{ij} - \bar{y}_{ij}) = 0$ for any $1\le q\le n+1$.  Thus taking suitable partial sums in the row major order we get initial 
 row sum estimates  $|\sum_{j=1}^{b}(y_{ij} - \bar{y}_{ij})| < 1$ for any fixed $i$ and $1\le b\le n+1$.  Similarly, suitable partial sums in column major 
 order yield  $|\sum_{i=1}^{b}(y_{ij} - \bar{y}_{ij})| < 1$ for any fixed $j$ and $1\le b\le m+1$.  Finally, let $F$ be the upper left $m\times n$ submatrix of $\overline{Y}$; 
 that is, $F = (f_{ij})$ where $f_{ij} = \bar{y}_{ij}$ for $1\le i\le m$ and $1\le j\le n$.     
 Then recalling that $t_{ij} = y_{ij}$ for such $i$ and $j$, we get parts (a) and (b) of the theorem. 
 
 For (c), let $R = \sum_{i=1}^{m+1} \sum_{j=1}^{n+1} (y_{ij}-\bar{y}_{ij})$, $S = \sum_{i=1}^{m+1}(y_{i,n+1}-\bar{y}_{i,n+1})$, and $S' = \sum_{j=1}^{n+1}(y_{m+1,j}-\bar{y}_{m+1,j})$, 
 noting that these three quantities are all $0$ since each row or column sum in $Y$ is integral.  Then 
 $|\sum_{i=1}^{m} \sum_{j=1}^{n} (t_{ij}-f_{ij})| = |R - S - S' + (y_{m+1,n+1} - \bar{y}_{m+1,n+1})| \le | y_{m+1,n+1} - \bar{y}_{m+1,n+1}| < 1$.     
 \end{proof}
 
 \medskip
 
 The setting in which we apply Theorem \ref{Doerr's theorem} is as follows.  Let $X = \{s_{1}, s_{2},\ldots, s_{m}  \}$ be an 
 integer sequence such that for all $1\le i\le m$ we have $s_{i} = k$ or $k+1$ for some fixed integer $k$ independent of $i$.  Also let $n$ be a positive integer satisfying $k+1\le n$.  Define 
 the $m\times n$ matrix $T^{X} = (t_{ij})$ as follows.  For each fixed row index $i$, $1\le i\le m$, let $t_{ij} = \frac{s_i}{n}$ for 
 all $j$, $1\le j\le n$.  Thus all entries in a given row of  $T^{X}$ have the same constant value.  
 For some rows this constant is  $\frac{k}{n}$ while for the rest it is $\frac{k+1}{n}$, and row $i$ 
 of $T^{X}$ has row sum $s_{i}$.  Now let $F^{X} = (f_{ij})$ be an $m\times n$, $(0,1)$ rounding matrix of $T^{X}$ as guaranteed to exist 
 by Theorem \ref{Doerr's theorem}; that is, with $T^{X}$ and $F^{X}$ playing the roles of $T$ and $F$ respectively in that theorem.  For suitably chosen 
 integers $m,n$ and integer sequence $X$ determined by the construction in the next 
 section, this $F^{X}$ will be, for fixed $i$, the matrix $F(i)$ introduced in the previous section that encodes which $(i-1)$-levels of $Y_{i}$ will be designated blank.  
 This $F^{X}$ has certain balance properties described in the following theorem.
 
 \begin{theorem}\label{balance properties}
 
 Let $k$ and $n$ be positive integers with $1\le k+1\le n$, and $X = (s_{1}, s_{2},\cdots, s_{m} )$ a sequence with $s_{i} = k$ or $k+1$ for all $i$.  
 Let $T^{X} = (t_{ij})$ and $F^{X} = (f_{ij})$, a rounding of $T^{X}$, be the $m\times n$ matrices as defined above.  Then $F^{X}$ has the 
 following properties.
 
 \medskip
 
\noindent{\rm{(a)}} $\sum_{j=1}^{n} f_{ij} = s_{i}$ for $1\le i\le m$.

\medskip

\noindent{\rm{(b)}} $|\sum_{j=1}^{b} f_{jr} - \sum_{j=1}^{b} f_{js}|\le 1$ for $1\le r,s\le n$, $1\le b\le m$.

\medskip

\noindent{\rm{(c)}} $|\sum_{j=1}^{b} f_{rj} - \sum_{j=1}^{b} f_{sj}|\le 2$ for $1\le r,s\le m$, $1\le b\le n$.

 \end{theorem}
 
 \begin{proof} For (a), note that row $i$ of $T^{X}$ has the integer $s_{i}$ as its sum.  Then (a) follows from Theorem  \ref{Doerr's theorem}a on taking 
 $b=n$.
 
 For (b), observe that for any $b, 1\le b\le m$, the sequences formed by the first $b$ entries in any two columns $r$ and $s$ of $T^{X}$ are identical.  So 
  $\sum_{j=1}^{b} t_{jr} = \sum_{j=1}^{b} t_{js}$ for any $r$ and $s$.  Thus by Theorem \ref{Doerr's theorem}b, $\sum_{j=1}^{b} f_{jr}$ 
 and $\sum_{j=1}^{b} f_{js}$ are roundings of the same quantity, and hence can differ by at most $1$, proving (b).
 
 For (c), recall that any two rows $r$ and $s$ of $T^{X}$ are either identical, or one has constant entries $\frac{k}{n}$ and the other constant entries $\frac{k+1}{n}$.  So the 
 maximum difference between any two initial row sums of $T^{X}$ occurs in the second possibility, and in that case the difference we get is 
 $|\sum_{j=1}^{b} t_{rj} - \sum_{j=1}^{b} t_{sj}|\le \frac{b}{n} \le 1$.  By Theorem \ref{Doerr's theorem}c the initial row sums $\sum_{j=1}^{b} f_{rj}$ and $\sum_{j=1}^{b} f_{sj}$ 
 of $F^{X}$ are roundings of the corresponding initial row sums in $T^{X}$.  Since the latter sums differ by at most $1$, part (c) follows, where the upper bound of $2$ can be achieved 
 only if there is an integer strictly between the corresponding initial row sums in $T^{X}$.          
 \end{proof}
 
For any two integers $r$ and $s$, $1\le r,s\le m$, we will need to bound the column difference between 
the $p$'th zero from the left in row $r$ of $F^{X}$ and the $q$'th zero from the left in row $s$ of $F^{X}$ as a function of $|p-q|$, independent of $r$ and $s$.  For this purpose, let $N_{i}(d)$ be the column 
index of the $d$'th zero from the left in row $i$ of $F^{X}$; that is, $N_{i}(d) = min\{b: b = d +  \sum_{j=1}^{b} f_{ij}   \}$.  We omit mention of $X$ in the notation $N_{i}(d)$, since $X$ will be clear by context. 
We will call the entry $f_{ih}$ {\it forward} if $\sum_{j=1}^{h}f_{ij} = \lceil  \sum_{j=1}^{h}t_{ij}   \rceil$.  Otherwise we call $f_{ih}$ {\it backward}, so in that case 
$\sum_{j=1}^{h}f_{ij} = \lceil  \sum_{j=1}^{h}t_{ij}   \rceil - 1 $.   
Recall that every entry of $F^{X}$ is either forward or backward by Theorem  \ref{Doerr's theorem}a. 

It will be convenient to interpret the entries $f_{ij}$ of the $m\times n$ matrix $F^{X}$ and the function $N_{i}(*)$ using "wraparound".  For example we let  
$f_{i,n+3} = f_{i+1,3}$.  Similarly if row $i$ of $F^{X}$ has $q$ many $0$'s then we let $N_{i}(q+3) = N_{i+1}(3)$, while if row $i+1$ has $q'$ many $0$'s then 
let $N_{i}(q+q'+3) = N_{i+2}(3)$, and so on.  Also for $t\le 0$, let $N_{i+1}(t) = N_{i}(q+t)$.

 \begin{corollary}\label{zeros column index}
 
 Let $X = \{ s_{1}, s_{2}, \ldots, s_{m} \}$, $T^{X} = (t_{ij})$, $F^{X} = (f_{uv})$, and $k$ be as in the preceding theorem, and assume now that $k+1\le \frac{n}{2}$.  Let $c = \frac{k+1}{n}$.
 Then the $m\times n$, $(0,1)$-matrix $F^{X}$ has the following properties.  
 
  \medskip
  
\noindent{\rm{(a)}} If $f_{ih}$ is forward, then $\sum_{j=h+1}^{h+2e}f_{ij}\le e$.  In particular, if $f_{i,N_{i}(d)}$ is forward, \newline then $N_{i}(d+e) - N_{i}(d) \le 2e$. 
  
\medskip
  
 \noindent{\rm{(b)}} If $f_{ih}$ is backward, then $\sum_{j=h+1}^{h+2e}f_{ij}\le e+1$.  In particular, if $f_{i,N_{i}(d)}$ is backward, \newline then $N_{i}(d+e) - N_{i}(d) \le 2e+2$. 
  
 \medskip
 
 \noindent{\rm{(c)}} For integers $r,s,d,e$ with $1\le r,s\le m$, we have $N_{r}(d+e) - N_{s}(d)\le 2e + 4$.  
 
 \medskip
 
 \noindent{\rm{(d)}} Let $2\le i\le k-1$.  Set $m = P_{i} = a_{i+1}a_{i+2}\cdots a_{k}$, $n = 2^{e_{i}-e_{i-1}}$, $k =$ min$\{ s_{i}(j): 1\le j\le P_{i}\}$, 
 $s_{j} = s_{i}(j)$, $1\le j\le m = P_{i}$, and $f_{uv} = f_{uv}(i)$. 
 With these settings, $F^{X}$ satisfies (\ref{overv row sums general})-(\ref{overv initial row sums general}) (which include 
 the conditions (\ref{overv row sums})-(\ref{overv initial row sums}) as the special case $i=2$).  

 \end{corollary} 
 
 \begin{proof} For part (a), begin by observing that for any positive integer $p$ we have $\sum _{j=N_{i}(d)+1}^{N_{i}(d)+p}t_{ij}\le pc\le \frac{p}{2}$.  Now take 
 $p=2e$ to be an even integer.  Then since $f_{ih}$ is forward, 
 the number of $1$'s among the entries $f_{ij}$, $h+1\le j\le h+2e$, is at most $e$, so $\sum_{j=h+1}^{h+2e}f_{ij}\le e$ as claimed.  For the second statement, since the number of these entries is $2e$
 and $\sum_{j=h+1}^{h+2e}f_{ij}\le e$, it follows that    
 the number of $0$'s among these must entries is at least $e$, and hence $N_{i}(d+e) - N_{i}(d)\le 2e$.  
 
For part (b), we proceed as in part (a), taking $p=2e$ to be an even integer.  The difference is that since $f_{i,N_{i}(d)}$ is backward,  the number 
 of $1$'s among the entries $f_{ij}$, $h+1\le j\le h+2e$ can be at most $1 + e$ (since that number now includes the $1$ corresponding 
 to  $\lceil  \sum_{j=1}^{N_{i}(d)}t_{ij}   \rceil$).  It follows that $\sum_{j=h+1}^{h+2e}f_{ij}\le e+1$.  For the second statement, since $\sum_{j=h+1}^{h+2e}f_{ij}\le e+1$ 
 it follows that  the number of $0$'s among these entries is at least $e-1$.  
 Now replacing $2e$ by $2e+2$ in the above reasoning, we find that the number of $0$'s among 
 the $2e+2$ entries $f_{ij}$, $h+1\le j\le h+2e+2$, is at least $e$.  It follows that $N_{i}(d+e) - N_{i}(d) \le 2e+2$.
 
 Consider now part (c).  By Theorem \ref{balance properties}c we have $N_{r}(d-2)\le N_{s}(d)$; that is, there are at least $d-2$ many $0$'s among the 
 entries $f_{rj}$, $1\le j\le N_{s}(d)$.  Suppose first that $N_{r}(d-1) > N_{s}(d)$, so that there are exactly $d-2$ many $0$'s among the 
 entries $f_{rj}$, $1\le j\le N_{s}(d)$.  It follows that $\sum_{j=1}^{N_{s}(d)}f_{rj} - \sum_{j=1}^{N_{s}(d)}f_{sj} = 2$.  Thus $f_{r,N_{s}(d)}$ is forward.  
 So by part (a) we have $\sum_{N_{s}(d)+1}^{N_{s}(d)+2e+4}f_{rj}\le e+2$.  So there are at least $e+2$ many $0$'s among the entries 
 $f_{rj}$, $N_{s}(d)+1\le j\le N_{s}(d)+2e+4$.  Combining these $0$'s with the $d-2$ many $0$'s  among the entries  $f_{rj}$, $1\le j\le N_{s}(d)$, obtain 
 $N_{r}(d+e)\le N_{s}(d) + 2e + 4$, as required.  Now assume $N_{r}(d-1) \le N_{s}(d)$.  Then we have 
 $N_{r}(d+e) - N_{s}(d)\le N_{r}(d+e) - N_{r}(d-1)\le 2(e+1) + 2 = 2e+4$, where we have used the more generous bound from part (b) in bounding 
$N_{r}(d+e) - N_{r}(d-1)$.  This completes (c).

 Finally consider part (d).  By Lemma \ref{blank arithmetic}b we have $s_{i}(j) = k$ or $k+1$ for each term $s_{i}(j)$, $1\le j\le P_{i}$, of our sequence $X$.  Hence 
 we may apply Theorem \ref{balance properties}a,b,c with the given settings for $m,n,X$, and for the entries $f_{uv}$ of $F^{X}$, to obtain 
 (\ref{overv row sums general}), (\ref{overv initial column sums general}), and (\ref{overv initial row sums general}) 
 respectively.       \end{proof} 
 
We now apply Theorems \ref{Doerr's theorem} and \ref{balance properties} to construct the matrices $F(i)$ of the previous section.  Recall that $F(i) = (f_{cd}(i))$ is 
the $P_{i}\times 2^{e_{i}-e_{i-1}}$, (0,1)-matrix such that for $1\le d\le 2^{e_{i}-e_{i-1}}$ and $1\le c\le P_{i}$,  the $(i-1)$-level $Y_{i}^{(c-1)2^{e_{i}-e_{i-1}} + d }$ 
of $i$-section $S_{i}^{c}$ (i.e. the $d$'th $(i-1)$-level of $S_{i}^{c}$) is blank precisely when $f_{cd}(i) = 1$.

\medskip

\begin{center} \textbf{Construction of the matrix $F(i)$, $1\le i\le k-1$}
\end{center}

\noindent{\textbf{1.}} Given $G$ and fixed $i$, compute the sequence $X = \{s_{i}(j) \}, 1\le j\le P_{i}$, using formula (\ref{blank formula}).

\smallskip

\noindent{\textbf{2.}} Form the $P_{i}\times 2^{e_{i}-e_{i-1}}$ matrix $T^{X}$ and construct rounding $F^{X}$ of $T^{X}$ as in Theorem \ref{balance properties}; that is, letting 
$m = P_{i}$, $n = 2^{e_{i}-e_{i-1}}$, and $s_{j} = s_{i}(j)$ in the definition of $T^{X}$ (preceding Theorem \ref{balance properties}).  We note that the hypotheses of 
Theorem \ref{balance properties} hold with this $X$ by Lemma \ref{blank arithmetic}b.

\smallskip

\noindent{\textbf{3.}} Let $F(i) = F^{X}$.

\medskip

To illustrate, recall $G'' = [3\times 7\times 4]$ from the previous section.  Toward constructing $F(2)$, apply formula (\ref{blank formula}) to obtain the sequence  
$X = \{s_{2}(1)=2, s_{2}(2)=3, s_{2}(3)=3, s_{2}(4)=3 \}$. So the $P_{2}\times 2^{e_{2}-e_{1}} = 4\times 8$ matrix  
 $T^{X}$ has constant row entries $\frac{s_{2}(r)}{8} = \frac{1}{4}$ or  constant row entries $\frac{s_{2}(r)}{8} = \frac{3}{8}$ 
 in any given row $r$, $1\le r\le 4$, as shown at left in Table 1a.  The rounding of $T^{X}$ by Knuth's network flow method or Doerr's approach (Theorems \ref{Knuth} and \ref{Doerr's theorem})  
 yields, as one possibility, the $(0,1)$ matrix $F(2): = F^{X}$ at right in Table 1a.  Next recall $G' = [3\times 7\times 4\times 3]$.  To construct $F(2)$, 
 apply formula (\ref{blank formula}) to obtain the sequence $X = \{ s_{2}(j) \}$, $1\le j\le 12$, given by $X = (2,3,3,3,2,3,3,3,2,3,3,3)$.  So the 
 $P_{2}\times 2^{e_{2}-e_{1}} = 12\times 8$ matrix $T^{X}$ has constant row entries  $\frac{1}{4}$ or $\frac{3}{8}$ in any given row, and is shown at left in Table 1b.  
 Again applying Theorems \ref{Knuth} and \ref{Doerr's theorem} 
 gives a possible rounding of $T^{X}$ given by $F(2) = F^{X}$ at right in Table 1b.  Toward constructing $F(3)$, still for $G'$, we apply formula (\ref{blank formula}) to obtain the sequence 
 $X = \{s_{3}(1)=1, s_{3}(2)=1, s_{3}(3)=2 \}$.  The $P_{3}\times 2^{e_{3}-e_{2}} = 3\times 4$ matrix $T^{X}$ therefore has constant row entries $\frac{1}{4}$ 
 or $\frac{1}{2}$, as shown at left in Table 1c, while a possible rounding $F(3) = F^{X}$ by these theorems is shown at right.

\subsection{An Integer Labeling of Hypercubes}\label{cats}
In this subsection we construct an ordering $L_{t}$ of the vertices of $Q_{t}$ such that for any interval of at most $O(log(t))$ consecutive points under this ordering, 
any  two points $x,y$ in that interval satisfy $dist_{Q_{t}}(x,y)\le 3$.    
 This labeling is based on the existence of certain spanning 
subgraphs of hypercubes of suitable dimension, as follows.  Define a \textit{cyclic caterpillar} as a connected graph $H$ such that removal of all leaves 
of $H$ results in a cycle graph $C_{e}$ for some $e\geq 3$.  A cyclic caterpillar is \textit{r-regular} if each vertex of its cycle 
subgraph $C_{e}$ has exactly $r$ neighboring leaves not on the cycle.  Denote such an $r$-regular cyclic caterpillar by $Cat(e,r)$.

We are interested in finding spanning subgraphs $Cat(e,r)$ of hypercubes.  Clearly if $Cat(e,r)$ spans $Q_{n}$, then $e = \frac{2^{n}}{r+1}$, so 
$r+1$ must be a power of $2$.  Papers on this subject include \cite{CK}, \cite{DHLM}, and \cite{HLW}, while \cite{B} examines the general question of finding short dominating cycles or paths in the hypercube.  
We use following result.

\begin{theorem}\label{spanning cyclic caterpillars} [Corollary 5.8 in \cite{CK}] There exists a spanning cyclic regular caterpillar $Cat(e,2r+1)$ of $Q_n$ 
provided that $r+1 = 2^{i}$ and $n = 2^{i+1}+2i$ for some integer $i\geq 0$. 

\end{theorem}

From this we easily obtain the following.

\begin{corollary}\label{containment forever} Suppose $r+1 = 2^{i}$ and $t\geq 2^{i+1} + 2i$ for an integer $i\geq 0$.  Then $Q_{t}$ 
contains a spanning cyclic $(2r+1)$-regular caterpillar $Cat(e,2r+1)$ for suitable $e$.
    
\end{corollary}  

\begin{proof} Recall that $Q_{t+1}$ is the cartesian product $Q_{t+1} = Q_{t}\times K_{2}$.  Hence if $Q_{t}$ contains a spanning subgraph $Cat(e,2r+1)$, then 
$Q_{t+1}$ contains a spanning subgraph $Cat(2e,2r+1)$.  The corollary follows by induction on $t$. 
\end{proof}

We can now construct our desired labeling of hypercubes of sufficiently large dimension. 

\begin{corollary}\label{labeling} 
\noindent{\rm{(a)}} Let $r,i,t$ be positive integers satisfying $r+1 = 2^{i}$ and $t\geq 2^{i+1} + 2i$.  Then there exists 
a one to one integer labeling 
$L_{t}: V(Q_{t})\rightarrow \{1,2,\ldots,2^{t} \}$  such that for any $x,y\in V(Q_{t})$ we have  
$|L_{t}(x) - L_{t}(y)|\le 2r+3 \Rightarrow dist_{Q_{t}}(x,y)\le 3$, where the indicated difference is taken modulo $2^t$.

\noindent{\rm{(b)}} Let  $t\geq 22$.  Then there exists 
a one to one integer labeling 
$L_{t}: V(Q_{t})\rightarrow \{1,2,\ldots,2^{t} \}$  such that for any $x,y\in V(Q_{t})$ we have  
$|L_{t}(x) - L_{t}(y)|\le 17 \Rightarrow dist_{Q_{t}}(x,y)\le 3$, where the indicated difference is taken modulo $2^t$.

\end{corollary}

\begin{proof} For (a), consider the spanning subgraph $Cat(e,2r+1)$ of $Q_{t}$, $t\geq 2^{i+1} + 2i$, from Corollary \ref{containment forever}.  For 
$1\le i\le e$, let $x_i$ be the vertices of the cycle $C_{e}$ in $Cat(e,2r+1)$ indexed consecutively around this cycle.  Also let 
$x_{i,1}, x_{i,2},\ldots, x_{i,2r+1}$ be the leaf neighbors of $x_i$.  Now define $L_{t}$ by letting $L_{t}(x_i) =(2r+2)i$ 
for $1\le i\le e$, and $L_{t}(x_{i,j}) = (2r+2)(i-1) + j$ for fixed $i$ and $1\le j\le 2r+1$.  For the claim, the critical case to check is when 
$x$ (or $y$) = $x_{i,2r+1}$ for any $1\le i\le e$.  The $2r+3$ points which follow this $x$ in this ordering are (in order) 
$x_{i}$, $\{ x_{i+1,j}: 1\le j\le 2r+1 \}$, and $x_{i+1}$, all at distance $3$ from $x$.  Further $2r+3$ is best possible since the $(2r+4)$'th point 
following $x$ is $x_{i+2,1}$, and $dist_{Q_{t}}(x,x_{i+2,1}) = 4$.

For (b), simply apply part (a) with $r=7$, yielding $i = 3$, and $t\geq 22$.
\end{proof} 

We will apply the labeling $L_{t}$ when $t = e_{i} - e_{i-1}$, for each $1\le i\le k$.  So to apply Corollary \ref{labeling}b we need the condition  
$t = e_{i} - e_{i-1}\geq 22$.  We therefore assume $a_{i}\geq 2^{22}$, $1\le i\le k$, from now on, which ensures that this condition holds.

  \section{The general construction} \label{high dimensional}
  
  In this section we inductively construct a series of embeddings  $f_{i}: G\rightarrow [\langle Y_{i-1} \rangle \times P(u_{i})] = Y_{i}^{(u_{i})}\subset Y_{i}$, $3\le i\le k$, where 
  $u_{i}= \lceil \frac{|G|}{|\langle Y_{i-1} \rangle|} \rceil = \lceil \frac{|G|}{2^{e_{i-1}}} \rceil$.  At the end we relabel the points of $\langle Y_{k-1} \rangle \times P(u_{k})$ with 
  hypercube addresses coming from $Opt(G)$, using the inverse of the labeling of Corollary \ref{labeling}.  The composition of this relabeling with the map $f_{k}$ yields 
  the final embedding $H^{k}: G\rightarrow Opt(G)$.    We follow the general plan outlined in section \ref{idea and example}.

Recall $S_{i}(G)$, the set of points of $S_{i}^{(P_{i})}$ lying in nonblank columns of $S_{i}^{(P_{i})}$.  We let $S_{i}^{r}(G) = S_{i}(G)\cap S_{i}^{r}$, the set of points lying in nonblank columns of $S_{i}^{r}$, and 
$S_{i}^{(r)}(G) = \bigcup_{j=1}^{r}S_{i}^{j}(G)$.  Recall we stipulated that $(I_{i}\circ f_{i})(G)\subseteq S_{i}(G)$.  For $r\le P_{i}$ we let 
$S_{i}^{r}(G)' = S_{i}^{r}(G)\cap (I_{i}\circ f_{i})(G)$ and $S_{i}^{(r)}(G)' = S_{i}^{(r)}(G)\cap (I_{i}\circ f_{i})(G)$.  Thus $S_{i}^{(P_{i})}(G)'$ is the domain of the stacking map $\sigma_{i}$, 
and $S_{i}^{r}(G)'$ (resp. $S_{i}^{(r)}(G)' $) is the subset of that domain lying in section $S_{i}^{r}$ (resp. $S_{i}^{(r)}$).  We will see later that $S_{i}^{(P_{i}-1)}(G) = S_{i}^{(P_{i}-1)}(G)'$.

\subsection{The three dimensional embedding}

In this subsection we construct the map $f_{3}: G\rightarrow [\langle Y_{2} \rangle \times P(u_{3})]\subset Y_{3}$ for a $3$-dimensional grid $G = [a_{1}\times a_{2} \times a_{3}]$.  Since 
$f_{3}(G)$ is a $3$-dimensional grid, we can visualize its construction along lines of Section \ref{idea and example} using Figures \ref{1thru12}, \ref{stack4sections}, and \ref{7and12stacking}, to which we refer the reader 
toward following the construction which follows.  This $3$-dimensional case will hopefully help in understanding the generalization 
to higher dimensions in the next subsection.

\medskip

\begin{center} \textbf{Construction of the map $f_{3}: [a_{1}\times a_{2}\times a_{3} \times \cdots \times a_{k} ]\rightarrow Y_{3}$}
\end{center}

\noindent{\textbf{1.}} Set $G =  [a_{1}\times a_{2}\times a_{3} \times \cdots \times a_{k} ]$.  Construct the embedding $f_{2}: G\rightarrow Y_{2}^{(u_{2})}$ 
of section \ref{the 2D mapping}.    
\smallskip

\noindent{\textbf{2.}} [\textbf{Designation of Blank Columns}] 
 
\noindent \textbf{a)} Set $P_{2} = \Pi_{i=3}^{k}a_{i}$.  

\noindent{\textbf{b)}} Construct the matrix $F(2) = (f_{ij}(2))$ by the procedure following Corollary \ref{zeros column index} for the case $i=2$. 
(Comment: The matrix entries $f_{ij}(2)$ now satisfy (\ref{overv row sums})-(\ref{overv initial row sums}) 
by Corollary \ref{zeros column index}d.)   
\newline \textbf{c)} Now designate `blank' columns in the subgraph $S_{2}^{(P_{2})}$ of $Y_{2}$ as follows. 
\begin{center}  Column $j$ of $S_{2}^{i}$ is `blank' if $f_{ij}(2) =1$, and is `nonblank' if $f_{ij}(2) = 0.$
\end{center}
(Comment: Recall that column $j$ of $S_{2}^{i}$ is column $(i-1)2^{e_{2}-e_{1}}+j$ of $Y_{2}$.)
\newline \textbf{d)} For $1\le r\le P_{2}$, let $S_{2}^{r}(G)$ (resp. $S_{2}^{(r)}(G)$) be the set points lying in nonblank columns of $S_{2}^{r}$ (resp. $S_{2}^{(r)}$). 
\smallskip

\noindent{\textbf{3.}}  [\textbf{The Inflation Step}]   

\noindent \textbf{a)} ``Inflate" the image $f_{2}(G)$ by the map $I_{2}: Y_{2}^{(u_{2}(G))}\rightarrow S_{2}^{(P_{2})}(G)$ given as follows.  For any 
$(a,b)\in Y_{2}^{(u_{2}(G))}$, let 
$$I_{2}(a,b) = (a,b'),$$ 
where $Y_{2}^{b'}$ is the $b$'th nonblank column in $S_{2}^{(P_{2})}$ in increasing order of $b$.
\newline  \textbf{b)} Let $S_{2}^{r}(G)' = S_{2}^{r}(G)\cap (I_{2}\circ f_{2})(G)$ and $S_{2}^{(r)}(G)' = S_{2}^{(r)}(G)\cap (I_{2}\circ f_{2})(G)$.

\smallskip 

\noindent {\textbf{4.}}  [\textbf{The Stacking Step}]  

\noindent``Stack" the sets $S_{2}^{r}(G)'$, $1\le r\le P_{2}$, successively ``over"  $S_{2}^{1}\cong \langle Y_{2} \rangle$ by the 
stacking map $\sigma_{2}: S_{2}^{(P_{2})}(G)'\rightarrow S_{2}^{1}\times P(u_{3}) = Y_{3}^{(u_{3})}$ defined as follows.  Take $y = (a,b)\in  S_{2}^{(P_{2})}(G)'$, say with $y\in S_{2}^{r}(G)'$.

\noindent  \textbf{a)} Let $n_{y} = |\{ z\in S_{2}^{(r)}(G)': z_{1}= a$ and $z_{2} \equiv b$ (mod $2^{e_{2}-e_{1}}$)$\}|$. 

\noindent  \textbf{b)} Let $\bar{b}$ satisfy $b\equiv \bar{b}$ (mod $2^{e_{2}-e_{1}}$), 
with $1\le \bar{b}\le 2^{e_{2}-e_{1}}.$  

\noindent \textbf{c)} Define $\sigma_{2}(a,b) = (a,\bar{b},n_{y}).$

\noindent (Comment: Take any $(a,b)\in S_{2}(G)'$.  Then $1\le \sigma_{2}(a,b)_{2} = \bar{b}\le 2^{e_{2}-e_{1}}$, and $1\le \sigma_{2}(a,b)_{1} = a\le 2^{e_{1}}$.  Thus 
$\sigma_{2}(a,b)_{1\rightarrow 2}\in S_{2}^{1}$ so we can view 
$\sigma_{2}$ as ``stacking" the sets $S_{2}^{r}(G)', 1\le r\le P_{2}$, in succession by increasing $r$ ``over" $S_{2}^{1} \cong \langle Y_{2} \rangle $.  
So we have $(\sigma_{2}\circ I_{2}\circ f_{2})(G)\subseteq  \langle Y_{2} \rangle \times P(u)\subset Y_{3}$,  
where $u$ is the maximum of $n_{y}$ over all $y$ belonging to the last set $S_{2}^{P_{2}}(G)'$.  We will prove later that $u = u_{3}(G)$.)  

\smallskip
\noindent{\textbf{5.}} [\textbf{The Composition Step}] 
\newline Finally, define $f_{3}: G\rightarrow Y_{3}$ as the composition $f_{3} = \sigma_{2}\circ I_{2}\circ f_{2}$.

\subsection{Embeddings of grids of higher dimension }\label{higher construction}

Again set $G = [ a_{1}\times \cdots \times a_{k}]$.  In this subsection we inductively construct embeddings  
$f_{i}: G \rightarrow Y_{i}^{(u_{i})}$ for $3\le j\le k$.  Assume then that 
we have constructed the required maps $f_{2},\ldots, f_{i}$, $3\le i < k$, and we construct $f_{i+1}: G \rightarrow Y_{i+1}^{(u_{i+1})}$.

Starting from $f_{i}(G)\subset Y_{i}^{(u_{i})}$,  we will use direct analogues $I_i$ and $\sigma_i$ of the inflation map $I_2$ and 
the stacking map $\sigma_2$ used in constructing $f_3$ from $f_2$.  In particular, $I_i$ will inflate $f_{i}(G)$ by introducing blank 
$(i-1)$-levels $Y_{i}^{j}$ of $Y_{i}$ using a matrix $F(i) = F^{X}$  (constructed by the procedure following Corollary \ref{zeros column index})  that encodes which 
$(i-1)$-levels of $Y_{i}$ will be introduced as blank.  Then a stacking map 
$\sigma_i: S_{i}^{(P_{i})}(G)'\rightarrow S_{i}^{1}\times P(u_{i+1}) = Y_{i+1}^{(u_{i+1})}$ stacks the sets $S_{i}^{r}(G)'$, $1\le r\le P_{i}$, in succession by increasing $r$ over 
$S_{i}^{1}\cong \langle Y_{i} \rangle$ to yield the final image $f_{i+1}(G)$.  Thus we may write $f_{i+1}$ as the composition $f_{i+1} = \sigma_{i} \circ I_{i} \circ f_{i}$, 
and inductively $f_{i+1} = \sigma_{i} \circ I_{i} \circ \sigma_{i-1} \circ I_{i-1} \circ \cdots \circ \sigma_{2} \circ I_{2} \circ f_{2}$.

\begin{center} \textbf{Construction of the map $f_{i+1}: [a_{1}\times a_{2}\times a_{3} \times \cdots \times a_{k} ]\rightarrow Y_{i+1}^{(u_{i+1})}, ~2\le i< k$}
\end{center} 

\noindent{\textbf{1.}} Set $G = [a_{1}\times a_{2}\times a_{3} \times \cdots \times a_{k} ]$. Assume inductively that the map $f_{i}: [a_{1}\times a_{2}\times a_{3} \times \cdots \times a_{k} ]\rightarrow Y_{i}^{(u_{i})}$ has been 
constructed.  We now operate on $f_{i}(G)$ to obtain our image $f_{i+1}(G)\subset Y_{i+1}^{(u_{i+1})}$. 

\smallskip

\noindent{\textbf{2.}} [\textbf{Designation of Blank $(i-1)$-Levels}] 
\newline \textbf{a)} Set $P_{i} = \Pi_{t=i+1}^{k}a_{t}$.
\newline \textbf{b)}  Construct the $P_{i}\times 2^{e_{i}-e_{i-1}}$, $(0,1)$-matrix $F(i) = (f_{cd}(i))$ by the procedure following Corollary \ref{zeros column index}. 
\newline [Comment: Thus by Corollary \ref{zeros column index}d, the matrix entries $f_{cd}(i)$ satisfy relations (\ref{overv row sums general})-(\ref{overv initial row sums general}).]

\noindent \textbf{c)} For $1\le c\le P_{i}$, define `` level $j$ of $S_{i}^{c}$ '', or the `` $j$'th level of $S_{i}^{c}$ '', to be the $(i-1)$-level $Y_{i}^{(c-1)2^{e_{t}-e_{t-1}}+j}\subset Y_{i}$.  Designate 
level $j$ of $S_{i}^{c}$ as being either `blank'  or `nonblank' as follows. 
\begin{center}  Level $j$ of $S_{i}^{c}$ is `blank' if $f_{cj}(i) =1$, and is `nonblank' if $f_{cj}(i) = 0$.
  
\end{center} 

\noindent \textbf{d)} Linearly order all the nonblank $(i-1)$-levels in $S_{i}^{(P_{i})}$ by increasing $i$'th coordinate; that is, for any two nonblank levels $(i-1)$-levels $Y_{i}^{t}$ 
and $Y_{i}^{t'}$ we have $Y_{i}^{t} < Y_{i}^{t'}$ if and only if $t < t'$.  

\noindent \textbf{e)} Let $S_{i}^{c}(G)$ (resp. $S_{i}^{(c)}(G)$) be the set points of $S_{i}^{(c)}$ lying in nonblank $(i-1)$-levels of $S_{i}^{c}$ (resp. $S_{i}^{(c)}$).

\smallskip

\noindent{\textbf{3.}}  [\textbf{The Inflation Step}]  
\newline \textbf{a)} ``Inflate" the image $f_{i}(G)$ by the map $I_{i}: Y_{i}^{(u_{i}(G))}\rightarrow S_{i}^{(P_{i})}(G)$ as follows.  
Take any $z = (z_{1},z_{2},\ldots,z_{i})\in Y_{i}^{(u_{i}(G))}$. Then let 
$$I_{i}(z) = (z_{1},z_{2},\ldots,z_{i-1}, z_{i}'),$$ 
where $Y_{i}^{z_{i}'}$ is the $z_{i}$'th nonblank $(i-1)$-level in $S_{i}^{(P_{i})}$ (in the ordering of nonblank $(i-1)$-levels in $S_{i}^{(P_{i})}$ 
from step 2(d)).

\noindent [Comment: Let $b$ and $r$ be the integers such that $Y_{i}^{z_{i}'}$ is the $b$'th nonblank $(i-1)$-level of $S_{i}^{r}(G)$.  
Recall that $N_{r}(b)$ is the column index 
in matrix $F(i)$ of the $b$'th zero in row $r$ of $F(i)$.  Then by step $2$(c), we have $z_{i}' = (r-1)2^{e_{i}-e_{i-1}} + N_{r}(b)$,
$1\le N_{r}(b)\le 2^{e_{i}-e_{i-1}}$, so 
 $$I_{i}(z) = (z_{1},z_{2},\ldots,z_{i-1}, (r-1)2^{e_{t}-e_{t-1}} + N_{r}(b)).]$$ 

\noindent \textbf{b)} Let $S_{i}^{r}(G)' = S_{i}^{r}(G)\cap (I_{i}\circ f_{i})(G)$ and $S_{i}^{(r)}(G)' = S_{i}^{(r)}(G)\cap (I_{i}\circ f_{i})(G)$.

\smallskip

\noindent{\textbf{4.}}  [\textbf{The Stacking Step}]  
\newline ``Stack" the sets $S_{i}^{r}(G)'$, $1\le r\le P_{i}$, from step $3$ (The Inflation Step) on top of each other ``over"  $S_{i}^{1}$ in succession as $r$ increases. We do this by the  
stacking map $\sigma_{i}: S_{i}^{(P_{i})}(G)'\rightarrow S_{i}^{1}\times P(u_{i+1})\subseteq Y_{i+1}^{(u_{i+1})}$ defined as follows.  Let $y = (y_{1},y_{2},\ldots,y_{i-1}, y_{i})\in S_{i}^{(P_{i})}(G)'$, 
say with $y\in S_{i}^{r}(G)'$, $1\le r\le P_{i}$.    
\newline \textbf{a)} Let $n_{y} = |\{ z\in S_{t}^{(r)}(G)': z_{1\rightarrow i-1}= y_{1\rightarrow i-1}$ and $z_{i} \equiv y_{i}$ (mod $2^{e_{i}-e_{i-1}}$)$\}|$.  
\newline \textbf{b)} Let $\bar{y}_{i}$ be such that 
$\bar{y}_{i} \equiv y_{i}$ (mod $2^{e_{i}-e_{i-1}})$, with $1\le \bar{y}_{i} \le 2^{e_{i}-e_{i-1}}.$  So $\bar{y}_{i} = N_{r}(b)$ as in the comment to step $3$ (The Inflation Step).  
\newline  \textbf{c)} Define $\sigma_{i}(y) = (y_{1}, y_{2}, \ldots, y_{i-1}, \bar{y}_{i},n_{y}) = (y_{1}, y_{2}, \ldots , y_{i-1}, N_{r}(b) ,n_{y}).$
\newline [Comment: Since $\sigma_{i}(y)_{1\rightarrow i} = (y_{1}, y_{2}, \ldots, y_{i-1}, \bar{y}_{i})\in S_{i}^{1}$, we see that  $\sigma_{i}(S_{i}^{(P_{i})}(G)')\subset S_{i}^{1}\times P(m)$, 
where $m$ is the maximum of $n_{y}$ over all $y$ belonging to the last set $S_{i}^{P_{i}}(G)'$.  We will see later that $m = u_{i+1}(G)$. ]

\smallskip

\noindent{\textbf{5.}} [\textbf{The Composition Step}]  
\newline Finally define $f_{i+1}: G\rightarrow Y_{i+1}^{(u_{i+1})}$ as the composition $f_{i+1} = \sigma_{i}\circ I_{i}\circ f_{i}$.  In particular, for any $z = f_{i}(v)\in f_{i}(G)$ we have 
$f_{i+1}(v) =  \sigma_{i}(I_{i}(z))$. 

\bigskip

Consider now the construction of $f_{i}$ from $f_{i-1}$ by the above construction.  The stacking step 4 suggests that we can regard $f_{i}(G) \subseteq S_{i-1}^{1}\times P(m)$ 
(as in the comment to step 4) as a collection 
of stacks addressed by the points of $S_{i-1}^{1}\cong \langle Y_{i-1} \rangle$. The height of each such stack extends into the 
$i$'th dimension of $Y_{i}$.  Specifically, for any $x\in \langle Y_{i-1} \rangle $, we let 
$Stack_{i}(x,r) = \{  \sigma_{i-1}(y): (\sigma_{i-1}(y))_{1\rightarrow i-1}= x$ and $y\in S_{i-1}^{(r)}(G)'  \}$, the stack addressed by $x$, for a given integer $r\le P_{i-1}$.  So $Stack_{i}(x,r)$ consists of 
 images $\sigma_{i-1}(y)$ which project onto $x$ in their first $i-1$ coordinates, and such that $y$ comes
from the first $r$ many sets $S_{i-1}^{j}(G)'$, $1\le j\le r$.  These sets $S_{i-1}^{j}(G)'$, $1\le j\le r$, are stacked on top of $S_{i-1}^{1}\cong \langle Y_{i-1} \rangle$ in order of increasing $j$.  
We view the ``height" of $\sigma_{i-1}(y)$ in 
$Stack_{i}(x,r)$ as its $i$'th coordinate $\sigma_{i-1}(y)_{i} = n_{y}$, and the height of $Stack_{i}(x,r)$ as $|Stack_{i}(x,r)|$.  
Now define $[r]_{i} = $max$\{|Stack_{i}(x,r)|: x\in \langle Y_{i-1} \rangle  \}$, the maximum  
height of any of these stacks addressed by points of $\langle Y_{i-1} \rangle$.  Thus by definition $\sigma_{i-1}(S_{i-1}^{(r)}(G)')$ is contained in the first $[r]_{i}$ many 
$(i-1)$-levels of $Y_{i}$; that is, $\sigma_{i-1}(S_{i-1}^{(r)}(G)')\subset Y_{i}^{([r]_{i})}$. 
As a convenience, for a stack address 
$x\in S_{i-1}^{1}\cong \langle Y_{i-1} \rangle$, we will refer to $x$ either as a member of $S_{i-1}^{1}$ or of $\langle Y_{i-1} \rangle$, in most cases of $\langle Y_{i-1} \rangle$.

The parameter $[r]_{i}$ plays a role in our containment result.  We will see in Lemma \ref{whatever}f that $I_{i-1}(f_{i-1}(D_{i-1}^{(r)}))\subseteq S_{i-1}^{(r)}(G)$, as 
suggested in section \ref{idea and example}.  Thus $[r]_{i}\geq   \lceil \frac{|D_{i-1}^{(r)}|}{ |\langle Y_{i-1} \rangle| }  \rceil  = \lceil \frac{|D_{i-1}^{(r)}|}{2^{e_{i-1}}}  \rceil $. 
We will also see ( Lemma \ref{whatever}e1 ) that in fact equality holds in the first inequality for each relevant $r$ and $i$; that is, $[r]_{i}$ is as small as it could possibly be.  
Thus taking $r = P_{i-1}$ and 
recalling that $D_{i-1}^{(P_{i-1})} = G$, we obtain $f_{i}(G) \subseteq Y_{i}^{(u_{i}(G))}\subseteq Opt(G)$ for each $i\geq 2$, our containment result.

The stacking step 4 of the construction will yield the following monotonicity properties of $Stack_{i}(x,r)$.  First, if $w',w''\in Stack_{i}(x,r)$ with $\sigma_{i-1}^{-1}(w')\in S_{i-1}^{c}$ 
and $\sigma_{i-1}^{-1}(w'')\in S_{i-1}^{d}$, $1\le c,d\le r$, then 
$w'_{i} > w''_{i}$ implies that $c > d$. That is, the originating $(i-1)$-section number under $\sigma_{i-1}^{-1}$ is strictly increasing as we move up any fixed stack.  For the 
second property, let  $Stack_{i}'(x,r) = f_{i}(D_{i-1}^{(r)})\cap Stack_{i}(x,r)$.  We will see that $Stack_{i}'(x,r)$ is an initial substack of $Stack_{i}(x,r)$.  Specifically, 
if $w',w''\in Stack_{i}'(x,r)$ with $f_{i}^{-1}(w')\in D_{i-1}^{c}$ and $f_{i}^{-1}(w'')\in D_{i-1}^{d}$, then 
$w'_{i} > w''_{i}$ implies that $c \geq d$.  Thus the originating $(i-1)$-page number under $f_{i}^{-1}$ is nondecreasing as we move up any fixed stack.  The first monotonicity property (which we call 
{\it stack monotonicity}) is immediate 
from our construction, and will  
be noted for the record in the Lemma which follows (part g).  The second monotonicity property (which we call 
{\it page monotonicity}) will be proved later in Theorem \ref{missing points}(a).

\begin{lemma}\label{whatever} 

\noindent{\rm(a)} For $x\in V(G)$. For $2\le j\le k$ the maps $f_{j}$ satisfy the following.

(a1) For $1\le i\le k-1$ and $i+1\le j\le k$ we have $f_{j}(x)_{1\rightarrow i} = f_{i+1}(x)_{1\rightarrow i}$.   

(a2) In particular, suppose 
$2\le i\le k-1$ and take $z = f_{i}(x)$.  Let $I_{i}(z) = (z_{1}, z_{2},\ldots, z_{i}')$
 
as in step $3$ (the Inflation Step) of the above construction, and express 

$z_{i}'$ as $z_{i}' = (r-1)2^{e_{i}-e_{i-1}} + N_{r}(b)$ as in the comment to step $3$. 

Then $f_{k}(x)_{i} = N_{r}(b)$, and $1\le f_{k}(x)_{i}\le 2^{e_{i}-e_{i-1}}$.

\smallskip

\noindent{\rm(b)} For any $x,y\in V(G)$, if $|f_{i}(x)_{i} - f_{i}(y)_{i}|\le e$ then $|f_{k}(x)_{i} - f_{k}(y)_{i}|\le 2e + 2$, where the last difference is 
interpreted mod $2^{e_{i}-e_{i-1}}$.

\smallskip

\noindent{\rm(c)}  Consider the map $\bar{\sigma}_{i}:y\rightarrow \sigma_{i}(y)_{1\rightarrow i}$, with $y\in S_{i}^{(P_{i})}(G)'$, obtained from $\sigma_{i}$ by projecting onto the 
first $i$ coordinates.  Then $\bar{\sigma}_{i}$  is one to one when restricted to any one $i$-section; that is, to the set $\{ y\in S_{i}^{r}(G)' \}$ for a given $1\le r\le P_{i}$. 
Hence $\sigma_{i}$ is one to one, and $f_{i}$ is one to one for all $2\le i\le k$.

\smallskip

\noindent{\rm(d)} Let  $1\le r< P_{i}$ and $1\le t< P_{i-1}$ be integers.  Then

(d1) For $i\geq 2$, $S_{i}^{(r)}(G)' = S_{i}^{(r)}(G)$

(d2) For $i\geq 3$, $|Stack_{i}(x,t)| = [t]_{i}$ or $[t]_{i} -1$ for any stack address $x\in \langle Y_{i-1} \rangle$.  

(d3) For $i\geq 3$ and any two stack addresses $x,y\in \langle Y_{i-1} \rangle$, if $x_{i-1} = y_{i-1}$, then $|Stack_{i}(x,t)| = |Stack_{i}(y,t)|$.

 \smallskip
 
 \noindent{\rm(e)} Let $l(i,r) = \lceil \frac{|D_{i-1}^{(r)}|}{2^{e_{i-1}}}  \rceil =\lceil \frac{ra_{i-1}a_{i-2}\cdots a_{1}}{2^{e_{i-1}}}  \rceil$, and 
 $l'(i,r) = \lceil \frac{|D_{i}^{(r)}|}{2^{e_{i-1}}}  \rceil =\lceil \frac{ra_{i}a_{i-1}\cdots a_{1}}{2^{e_{i-1}}}  \rceil$.  Then for $i\geq 3$ 

(e1) for $1\le r\le P_{i-1}$ we have $[r]_{i} = l(i,r)$, and $|Y_{i}^{(l(i,r))}| - |f_{i}(D_{i-1}^{(r)})|< 2^{e_{i-1}}$, and
     
(e2) for $1\le r\le P_{i-1}$ we have $f_{i}(D_{i-1}^{(r)})\subseteq Y_{i}^{(l(i,r))}$, and for $1\le r\le P_{i}$, $f_{i}(D_{i}^{(r)})\subseteq Y_{i}^{(l'(i,r))}$.

(e3) $f_{i}(G)\subseteq Y_{i}^{(u_{i})}\subseteq Opt(G)$.  Moreover, we have $f_{k}(G)\subseteq Opt'(G)$ and $H^{k}(G)\subseteq Opt(G)$.

\smallskip

 \noindent{\rm(f)} For $i\geq 2$ we have $I_{i}(f_{i}(D_{i}^{(r)}))\subseteq S_{i}^{(r)}(G)$ for $1\le r\le P_{i}$, and  
 $S_{i}^{(r)}(G)\subset I_{i}(f_{i}(D_{i}^{(r+1)}))$ for $1\le r\le P_{i}-1$.  
  
 \smallskip
 
 \noindent{\rm(g)} (Stack monotonicity) Suppose $w',w''\in Stack_{i}(x,r)$, $1\le r\le P_{i-1}$, with $\sigma_{i-1}^{-1}(w')\in S_{i-1}^{c}$ and $\sigma_{i-1}^{-1}(w'')\in S_{i-1}^{d}$, $1\le c,d\le r$. Then 
$w''_{i} > w'_{i} \Leftrightarrow d > c$.

\end{lemma}

\begin{proof}Consider first part (a1).  We proceed by induction on $j$.  The base case $j = i+1$ is trivial.  So suppose inductively that 
$f_{j}(x)_{1\rightarrow i} = f_{i+1}(x)_{1\rightarrow i}$ for some $j\geq i+1$.  By the definition of $I_{j}$ and $\sigma_{j}$ from the inflation and stacking steps 
respectively, we have $I_{j}(z)_{1\rightarrow j-1} = z_{1\rightarrow j-1}$ and $\sigma_{j}(y)_{1\rightarrow j-1} = y_{1\rightarrow j-1}$ for $z$ and $y$
in the domain of $I_{j}$ and $\sigma_{j}$ respectively.  
Thus since $j\geq i+1$ 
we have $I_{j}(f_{j}(x))_{1\rightarrow i} = f_{j}(x)_{1\rightarrow i}$.  Therefore  $f_{j+1}(x)_{1\rightarrow i} = \sigma_{j}(I_{j}(f_{j}(x)  )  )_{1\rightarrow i} = I_{j}(f_{j}(x)  )_{1\rightarrow i} 
= f_{j}(x)_{1\rightarrow i}= f_{i+1}(x)_{1\rightarrow i}$.  
This completes the inductive step, so (a1) is proved.  

Next consider (a2).  For $i\geq 2$ we have $ f_{i+1}(x)_{i} = N_{r}(b)$   
by step $4$c (The Stacking Step), and $1\le N_{r}(b)\le 2^{e_{i}-e_{i-1}}$ by the comment to step $3$ and the definition of $N_{r}(b)$, proving (a2).

Consider (b).  We know that $I_{i}(f_{i}(x))$ is in some nonblank $(i-1)$-level, say the $b$'th one, of some $i$-section, say $S_{i}^{r}$, of $S_{i}^{(P_{i})}$.  So we can write 
$I_{i}(f_{i}(x))_{i} = (r-1)2^{e_{i}-e_{i-1}} + N_{r}(b)$.  Since $|f_{i}(x)_{i} - f_{i}(y)_{i}|\le e$ and $I_{i}$ preserves the order of $(i-1)$-levels, it follows that $|I_{i}(f_{i}(x))_{i} - I_{i}(f_{i}(y))_{i}| \le e$.  
Assuming without loss that $I_{i}(f_{i}(y))_{i} > I_{i}(f_{i}(x))_{i}$ (otherwise interchange the roles of $x$ and $y$), we can write 
$I_{i}(f_{i}(y))_{i} = (r-1)2^{e_{i}-e_{i-1}} + N_{r}(b+t)$ for some $0\le t\le e$, where we interpret the function $N_{r}(*)$ in row major order as in the discussion just preceding 
Corollary \ref{zeros column index}. So we have $|f_{k}(x)_{i} - f_{k}(y)_{i}|\le |N_{r}(b+e) - N_{r}(b)|\le 2e+2$, where the first inequality follows from part (a2), and the second by 
using the more generous of the bounds (a) and (b) in 
Corollary \ref{zeros column index}, and interpreting the differences mod $2^{e_{i}-e_{i-1}}$.

Consider (c).  For the first statement, let $y,z\in S_{i}^{r}(G)'$ for some $1\le r\le P_{i}$.  We must show that $(\sigma_{i}(y))_{1\rightarrow i}\ne (\sigma_{i}(z))_{1\rightarrow i}$.  
By the comment to step $3$ (The Inflation Step) we may write 
$y= (y_{1},y_{2},\ldots,y_{i-1}, (r-1)2^{e_{i}-e_{i-1}} + N_{r}(d))$ and $z = (z_{1},z_{2},\ldots,z_{i-1}, (r-1)2^{e_{i}-e_{i-1}} + N_{r}(c))$ 
for suitable integers $d$ and $c$ and $1\le N_{r}(d), N_{r}(c)\le 2^{e_{i}-e_{i-1}}$. 
By step 4 (The Stacking Step), we have $(\sigma_{i}(z))_{1\rightarrow i} = (z_{1}, z_{2},\ldots, z_{i-1},N_{r}(d) )$, 
and $(\sigma_{i}(y))_{1\rightarrow i} = (y_{1}, y_{2},\ldots, y_{i-1},N_{r}(c) )$.  If $y$ and $z$ disagree at one of their first $i-1$ coordinates, then 
obviously $(\sigma_{i}(y))_{1\rightarrow i}\ne (\sigma_{i}(z))_{1\rightarrow i}$ by these formulas. 
So we can suppose that $y_{1\rightarrow i-1} = z_{1\rightarrow i-1} $.  
But since $y$ and $z$ are distinct, the formulas above for $y$ and $z$ force them to disagree in their $i$'th coordinates, so $N_{r}(d) \ne N_{r}(c)$.  Thus $y_{i}\not\equiv z_{i}$ (mod $2^{e_{i}-e_{i-1}}$).   
So by definition of $\sigma_{i}$ we get $(\sigma_{i}(y))_{1\rightarrow i}\ne (\sigma_{i}(z))_{1\rightarrow i}$, as desired. 
For one-to-oneness of $\sigma_{i}$ itself, it only remains to show that if $y,z$ come from distinct sections, then $\sigma_{i}(y)\ne \sigma_{i}(z)$.  So let 
$y\in S_{i}^{s}(G)'$ and $z\in S_{i}^{t}(G)'$, say with $s<t$.  If $(\sigma_{i}(y))_{1\rightarrow i}\ne (\sigma_{i}(z))_{1\rightarrow i}$, then the claim follows obviously, 
so assume $(\sigma_{i}(y))_{1\rightarrow i} = (\sigma_{i}(z))_{1\rightarrow i}$.  Then $\sigma_{i}(y)$ and $\sigma_{i}(z)$ both belong to the stack addressed by 
$(\sigma_{i}(y))_{1\rightarrow i}$, namely, $Stack_{i+1}((\sigma_{i}(y))_{1\rightarrow i},t)$.  But then since $s < t$, $\sigma_{i}(y)$ is lower in this stack 
than $\sigma_{i}(z)$ by definition of $\sigma_{i}$.     
Then $(\sigma_{i}(y))_{i+1}<(\sigma_{i}(z))_{i+1}$, proving one-to-oneness of $\sigma_{i}$.  As for the claim about the $f_{i}$, observe first 
that $f_{2}$ is one-to-one.  So from $f_{i+1} = \sigma_{i}\circ I_{i}\circ f_{i}$ and the one-to-oneness of $\sigma_{i}$ and $I_{i}$, it follows by 
induction on $i$ that $f_{i}$ is one-to-one for all $2\le i\le k$.

For (d), we start by proving (d1) for $i=2$.  By Theorem \ref{two dim properties}e,f, $Y_{2}^{(m-1)}\subset f_{2}(G)\subseteq Y_{2}^{m}$, 
where $m = \lceil \frac {|G|}{2^{e_{1}}} \rceil$ (noting that $m = u_{2}(G)$).  That is, all but at most the last column $Y_{2}^{m}$ of $Y_{2}^{(m)}$ is contained in 
$f_{2}(G)$.  Applying $I_{2}$ to both sides of the last containment and noting that $(I_{2}\circ f_{2})(G) = S_{2}^{(P_{2})}(G)'$, 
we see that every nonblank column of $S_{2}^{(P_{2})}(G)$ lies in $S_{2}^{(P_{2})}(G)'$, except possibly the last one $I_{2}(Y_{2}^{m})$, 
call it $Y_{2}^{d}$.  Since every $2$-section must contain at least one nonblank column we have $Y_{2}^{d}\subset S_{2}^{P_{2}}$ (the last $2$-section), 
 so $S_{2}^{(P_{2}-1)}(G) = S_{2}^{(P_{2}-1)}(G)'$.  Taking subsets, we get $S_{2}^{(r)}(G) = S_{2}^{(r)}(G)'$ for $1\le r < P_{2}$.
 
 Next we show that (d1) for $i$ implies (d2) and (d3) for $i+1$.  Take $x = (x_{1}, x_{2}, \ldots, x_{i})\in  \langle Y_{i} \rangle$ to be a stack address.  By 
 definition of matrix $F(i)$, its entry $f_{t,x_{i}}(i)$ satisfies $f_{t,x_{i}}(i) = 0$ if and only if the $(i-1)$-level $Y= Y_{i}^{ (t-1)2^{e_{i}-e_{i-1}}+x_{i}}$ (the $x_{i}$'th $(i-1)$-level of $i$-section $S_{i}^{t}$) is nonblank.  
 Assume first that $f_{t,x_{i}}(i) = 0$.  Thus $Y$ is nonblank.  Since $t<P_{i}$ (by the hypothesis for (d2) with $i+1$ replacing $i$), 
 we have $Y\subset S_{i}^{(P_{i}-1)}(G) = S_{i}^{(P_{i}-1)}(G)'$, the last equality 
 holding by the conclusion of (d1) for $i$.  Thus every point of $Y$ lies in the domain of $\sigma_{i}$.  So $Stack_{i+1}(x,t)$ receives the point 
  $\sigma_{i}(x_{1}, x_{2}, \ldots, (t-1)2^{e_{i}-e_{i-1}}+x_{i})$ under the stacking map $\sigma_{i}$.  If $f_{t,x_{i}}(i) = 1$, then $Stack_{i+1}(x,t)$ receives no point $\sigma_{i}(y)$ 
  with $y\in S_{i}^{t}$.  It follows that $|Stack_{i+1}(x,t)| = t - \sum_{j=1}^{t}f_{j,x_{i}}(i)$.
  Since by Lemma \ref{balance properties}b the sum on the right must be one of two successive integers depending on $t$, it follows that $|Stack_{i+1}(x,t)|$ is one of two successive integers 
  depending on $t$ but independent of $x$.  Thus by definition of $[t]_{i+1}$, we get $|Stack_{i+1}(x,t)| = [t]_{i+1}$ or $[t]_{i+1} -1$, proving (d2) for $i+1$.     
  Consider now (d3) for $i+1$, and let 
  $y\in  \langle Y_{i} \rangle$ be a stack address with $y_{i} = x_{i}$.  Then clearly $f_{j,y_{i}}(i) = f_{j,x_{i}}(i)$ for $1\le j\le P_{i}$ by 
  the definition of matrix $F(i)$.  Thus  $\sum_{j=1}^{t}f_{j,x_{i}}(i) = \sum_{j=1}^{t}f_{j,x_{i}}(i)$, and hence   $|Stack_{i+1}(x,t)| = |Stack_{i+1}(y,t)|$, as required.

Finally we show that (d2) for $i+1$ in place of $i$ implies (d1) for $i+1$ in place of $i$.  By assumption we have $|Stack_{i+1}(x,t)| = [t]_{i+1}$ or $[t]_{i+1}-1$ for $t\le P_{i}-1$ and any stack address 
 $x\in \langle Y_{i} \rangle$.  Observe that $[P_{i}]_{i+1}\le [P_{i}-1]_{i+1}+1$ since the projection $\sigma_{i}$ onto the first $i$ coordinates is one-to-one when restricted to any 
 single section, in this case $S_{i}^{P_{i}}$, by part (c).  Further, $|Stack_{i+1}(x,P_{i}-1)| \geq [P_{i}-1]_{i+1}-1$ for all stack addresses $x$ by assumption.  It follows that 
 $Y_{i+1}^{([P_{i}-1]_{i+1}-1)}\subseteq f_{i+1}(G)\subseteq Y_{i+1}^{([P_{i}]_{i+1})}\subseteq Y_{i+1}^{([P_{i}-1]_{i+1}+1)}$.  So at most two $i$-levels of $Y_{i+1}^{([P_{i}]_{i+1})}$ 
 are possibly not entirely contained in $f_{i+1}(G)$, these being the top two $Y_{i+1}^{[P_{i}]_{i+1}}$ and $Y_{i+1}^{[P_{i}]_{i+1}-1}$, and if this possibility occurs then $[P_{i}]_{i+1} = [P_{i}-1]_{i+1}+1$ 
 (so also $[P_{i}]_{i+1}-1 = [P_{i}-1]_{i+1}$).        
 But still $Y_{i+1}^{j}\cap f_{i+1}(G) \ne \emptyset$ if and only if $1\le j\le [P_{i}]_{i+1}$ by definition of $[P_{i}]_{i+1}$.  Thus every $i$-level $I_{i+1}(Y_{i+1}^{j})$, $1\le j\le [P_{i}]_{i+1}$, is 
 nonblank, and by the previous sentence all but at most two of these nonblank $i$-levels belong to $S_{i+1}^{(P_{i+1})}(G)'$, the two possible exceptions 
 being $I_{i+1}(Y_{i+1}^{[P_{i}]_{i+1}})$ and $I_{i+1}(Y_{i+1}^{[P_{i}]_{i+1}-1})$.         
 Since any $(i+1)$-section $S_{i+1}^{j}$, $1\le j\le P_{i+1}$, contains at least two nonblank $i$-levels (since $a_{i}>4$ for all $i$)  it follows that 
the top two nonblank $i$-levels of $S_{i+1}^{(P_{i+1})}$ belong to the top $(i+1)$-section $S_{i+1}^{P_{i+1}}$.  So the remaining 
$(i+1)$-sections $S_{i+1}^{j}$, $j<P_{i+1}$, satisfy $S_{i+1}^{j}(G) = S_{i+1}^{j}(G)'$.  In particular we have   
$S_{i+1}^{(P_{i+1}-1)}(G) = S_{i+1}^{(P_{i+1}-1)}(G)'$, and taking subsets we obtain (d1) for $i+1$.

Consider now (e1).  The case $i=2$ follows immediately from Theorem \ref{two dim properties}e,f, and we assume $i\geq 3$. 

For the lower bound $[r]_{i}\geq \lceil \frac{|D_{i-1}^{(r)}|}{2^{e_{i-1}}}\rceil = l(i,r)$ , suppose first that $r<P_{i-1}$.  
By definition we have $[r]_{i}\geq \frac{|\sigma_{i}(S_{i-1}^{(r)}(G)')|}{|\langle Y_{i-1}  \rangle|} = \frac{|\sigma_{i}(S_{i-1}^{(r)})(G)')|}{2^{e_{i-1}}}$.  So since 
$\sigma_{i}$ is one to one, it suffices to show that $|S_{i-1}^{(r)}(G)'|\geq \lceil \frac{|D_{i-1}^{(r)}|}{2^{e_{i-2}}}  \rceil 2^{e_{i-2}}$.  For then we get 
$[r]_{i} \geq \lceil \frac{|D_{i-1}^{(r)}|}{2^{e_{i-2}}}  \rceil 2^{e_{i-2}-e_{i-1}} \geq \frac{|D_{i-1}^{(r)}|}{2^{e_{i-1}}}$, and since $[r]_{i}$ is an integer 
$[r]_{i}\geq \lceil \frac{|D_{i-1}^{(r)}|}{2^{e_{i-1}}}\rceil = l(i,r)$. By part (d1) and Lemma \ref{blank arithmetic} we have $S_{i-1}^{(r)}(G) = S_{i-1}^{(r)}(G)'$, and 
$S_{i-1}^{(r)}(G)'$ consists of $\lceil \frac{|D_{i-1}^{(r)}|}{2^{e_{i-2}}}  \rceil$ many $(i-2)$-levels in $Y_{i-1}$.  Since each such $(i-2)$-level has $2^{e_{i-2}}$ 
points, we get $|S_{i-1}^{(r)}(G)'|\geq \lceil \frac{|D_{i-1}^{(r)}|}{2^{e_{i-2}}}  \rceil 2^{e_{i-2}}$ as required.  Now suppose that $r = P_{i-1}$.  We have 
$|S_{i-1}^{(P_{i-1})}(G)'| = |G| = |P_{i-1}a_{i-1}a_{i-2}\cdots a_{1}|$.  So again since $\sigma_{i}$ is one to one we get 
$[P_{i-1}]_{i}\geq \frac{|\sigma_{i}(S_{i-1}^{(P_{i-1})}(G)')|}{|\langle Y_{i-1}  \rangle|} \geq \frac{ |P_{i-1}a_{i-1}a_{i-2}\cdots a_{1}|}{2^{e_{i-1}}} = \frac{|D_{i-1}^{(P_{i-1})}|}{2^{e_{i-1}}}$, as required.

It remains to show that $[r]_{i}\le l(i,r)$.   
 Recall that for any stack address $x\in \langle Y_{i-1} \rangle$, $|Stack_{i}(x,r)|$ is at most the number of $0$'s in column $x_{i}$  
among the first $r$ rows in matrix $F(i-1)$.  By Theorem \ref{balance properties}b this number of $0$'s is either the same for all $x$ or is one of 
two successive integers, call them $\alpha_{r}$ or $\alpha_{r}-1$, depending only on $r$ (and $i$, which we fix in this argument).  Since $[r]_{i}$ is the maximum of $|Stack_{i}(x,r)|$ 
over all $x\in \langle Y_{i-1} \rangle$, it suffices to prove that 
$\alpha_{r}\le l(i,r)$.  By our construction, the total number of $0$'s in the first $r$ rows of $F(i-1)$ is the number of of nonblank $(i-2)$-levels in $S_{i-1}^{(r)}$.  
So by Lemma \ref{blank arithmetic}a this number of $0$'s is $\lceil \frac{ra_{i-1}a_{i-2}\cdots a_{1}}{2^{e_{i-2}}}  \rceil$.     
Let $\frac{p}{2^{e_{i-2}}} = \lceil \frac{ra_{i-1}a_{i-2}\cdots a_{1}}{2^{e_{i-2}}}  \rceil - \frac{ra_{i-1}a_{i-2}\cdots a_{1}}{2^{e_{i-2}}}$.  Since the number of 
columns of $F(i-1)$ is $2^{e_{i-1}-e_{i-2}}$, we have 
$\alpha_{r} = \lceil \frac{\lceil \frac{ra_{i-1}a_{i-2}\cdots a_{1}}{2^{e_{i-2}}}  \rceil}{2^{e_{i-1}-e_{i-2}}} \rceil = \lceil \frac{ra_{i-1}a_{i-2}\cdots a_{1}}{2^{e_{i-1}}} + \frac{p}{2^{e_{i-1}}}   \rceil = 
\lceil \frac{ra_{i-1}a_{i-2}\cdots a_{1}}{2^{e_{i-1}}} \rceil$, as required.

For the second claim
of (e1), again note that 
each $(i-1)$-level of $Y_{i}$ has size $2^{e_{i-1}}$.  Thus $|Y_{i}^{(l(i,r))}| - |f_{i}(D_{i-1}^{(r)})|= 2^{e_{i-1}}\lceil \frac{|D_{i-1}^{(r)}|}{2^{e_{i-1}}}  \rceil - |D_{i-1}^{(r)}| < 2^{e_{i-1}}$, 
completing the proof of (e1).

For (e2)  we prove both $f_{i}(D_{i-1}^{(r)})\subseteq Y_{i}^{(l(i,r))}$ and $f_{i}(D_{i}^{(r)})\subseteq Y_{i}^{(l'(i,r))}$ by induction.  If $i=2$, then both statements hold by 
Theorem \ref{two dim properties}d,e,f.  So let $i > 2$ be given, and assume inductively that both of these containments hold for $i-1$ in place of $i$.  By this assumption we 
have $f_{i-1}(D_{i-1}^{(r)})\subseteq Y_{i-1}^{(l'(i-1,r))}$.  Applying the inflation map $I_{i-1}$ to both sides, we obtain $(I_{i-1}\circ f_{i-1})(D_{i-1}^{(r)}) \subseteq S_{i-1}^{(r)}(G)$, 
since by Lemma \ref{blank arithmetic} we see that $S_{i-1}^{(r)}$ has the same number $l'(i-1,r)$ of nonblank $(i-2)$-levels as the number of $(i-2)$-levels in $Y_{i-1}^{(l'(i-1,r))}$, 
and $I_{i-1}$ preserves the order of the latter set of $(i-2)$-levels. 
Thus $(I_{i-1}\circ f_{i-1})(D_{i-1}^{(r)})\subseteq S_{i-1}^{(r)}(G)'$ by definition of $S_{i-1}^{(r)}(G)'$.  
Now 
applying $\sigma_{i-1}$ to the last containment, and recalling that $f_{i} = \sigma_{i-1} \circ I_{i-1} \circ f_{i-1}$, we get 
$f_{i}(D_{i-1}^{(r)}) = (\sigma_{i-1}\circ I_{i-1})(f_{i-1}(D_{i-1}^{(r)}))\subseteq \sigma_{i-1}(S_{i-1}^{(r)}(G)')\subseteq Y_{i}^{([r]_{i})}$, where the last containment holds 
by definition of $[r]_{i}$.  By (e1) we have $[r]_{i} = l(i,r)$, thereby proving $f_{i}(D_{i-1}^{(r)})\subseteq Y_{i}^{(l(i,r))}$.  To show $f_{i}(D_{i}^{(r)})\subseteq Y_{i}^{(l'(i,r))}$,  
apply $f_{i}(D_{i-1}^{(r)})\subseteq Y_{i}^{(l(i,r))}$ using $a_{i}r$ in place of $r$ and observing that $D_{i}^{(r)} = D_{i-1}^{(a_{i}r)}$.            
This completes the inductive step, and hence the proof of (e2).

Consider now (e3).  Recall that $P_{i-1} = \prod_{t=i}^{k}a_{t}$, so that $V(G) = V(D_{i-1}^{(P_{i-1})})$.  Noting that 
 $l(i,P_{i-1}) =  \lceil \frac{|D_{i-1}^{(P_{i-1})}|}{2^{e_{i-1}}}   \rceil =  \lceil \frac{|V(G)|}{2^{e_{i-1}}}  \rceil = u_{i}(G)$, by part (e2) we have 
 $f_{i}(G)\subseteq Y_{i}^{(l(i,P_{i-1}))} = Y_{i}^{(u_{i}(G))}$, proving the first containment.  For the rest, note that 
$u_{i}(G)\le 2^{\lceil log_{2}(|G|) \rceil - e_{i-1}}$, so $f_{i}(G)\subseteq Y_{i}^{(u_{i}(G))}\subseteq \langle Y_{i-1} \rangle \times P(2^{\lceil log_{2}(|G|) \rceil - e_{i-1}})$.  The 
last graph is a spanning subgraph of $Opt(G)$ for each $i$, yielding the second containment, and 
for $i=k$ yielding the final two containments after performing the relabeling with hypercube addresses by the map $H^{k}$.   
 
 For (f), start with $f_{i}(D_{i}^{(r)})\subseteq Y_{i}^{(l'(i,r))}$ from (e2).  
 Then we get $I_{i}(f_{i}(D_{i}^{(r)}))\subseteq I_{i}(Y_{i}^{(l'(i,r))}) = S_{i}^{(r)}(G)$. 
The equality holds since the right side consists of $l'(i,r)$ many (nonblank) $(i-1)$-levels, while $I_{i}$ preserves 
the order of $(i-1)$-levels by increasing $i$-coordinate value.

 Next we show that $S_{i}^{(r)}(G)\subset I_{i}(f_{i}(D_{i}^{(r+1)}))$. 
 For the case $i=2$ we apply Theorem \ref{two dim properties}e,f and Corollary \ref{2-dim dilation}e to $G = D_{2}^{(r+1)}$, together with Lemma \ref{blank arithmetic} to obtain 
 $S_{2}^{(r)}(G) = I_{2}(Y_{2}^{(l'(2,r))})\subseteq I_{2}(Y_{2}^{(l'(2,r+1) - 1)})\subseteq I_{2}(f_{2}(D_{2}^{(r+1)}))$.

 Proceeding by induction on $i$, let $i\geq 3$ and assume the statement true for $i-1$ and $1\le r\le P_{i-1}-1$, and we prove it for $i$ and $1\le r\le P_{i}-1$.  It suffices to show that 
 $Y_{i}^{(l'(i,r))}\subset f_{i}(D_{i}^{(r+1)})$, since then we can just apply $I_{i}$ to each side and use $I_{i}(Y_{i}^{(l'(i,r))}) = S_{i}^{(r)}(G)$ (as observed above) to get the result 
 directly.  
 
By (e1) we have 
 $[a_{i}r+2]_{i} = l(i,a_{i}r+2) \geq l'(i,r) +1$, using $2a_{i-1}a_{i-2}\ldots a_{1} > 2^{e_{i-1}}$.  By (d), for any $x\in \langle Y_{i-1}  \rangle$ we have 
 $|Stack_{i}(x, a_{i}r+2)| = [a_{i}r+2]_{i}$ or $[a_{i}r+2]_{i} -1$, so $|Stack_{i}(x, a_{i}r+2)| \geq  l'(i,r)$.  It follows that 
 $Y_{i}^{(l'(i,r))}\subseteq \sigma_{i-1}(S_{i-1}^{(a_{i}r+2)}(G)')\subseteq \sigma_{i-1}(I_{i-1}(f_{i-1}(D_{i-1}^{(a_{i}r+3)})))$, where the last containment is 
 by the inductive hypothesis.  Now since $a_{i}\geq 2^{22} > 3$, we have $a_{i}r+3 < a_{i}(r+1)$, so $D_{i-1}^{(a_{i}r+3)}\subset D_{i-1}^{(a_{i}(r+1))} = D_{i}^{(r+1)}$.  So 
 we have $Y_{i}^{(l'(i,r))}(G)\subset (\sigma_{i-1}\circ I_{i-1}\circ f_{i-1})(D_{i}^{(r+1)}) = f_{i}(D_{i}^{(r+1)})$, as desired.  
 
  For (g) observe that $\sigma_{i-1}$ stacks the sets $S_{i-1}^{r}(G)$ onto $S_{i-1}^{1}$ in succession by increasing $r$; that is, $S_{i-1}^{c}(G)$ is stacked before $S_{i-1}^{d}(G)$  
  precisely when $c < d$.  Part (g) follows.       	 \end{proof}

We can now give the final embedding $H^{k}$ of $G$ into $Opt'(G)$.  It is based on the fact that $f_{k}(G)\subseteq Opt'(G)$ from Lemma \ref{whatever}(e3).  So by definition we 
have  $1\le f_{k}(x)_{j}\le 2^{e_{j}-e_{j-1}}$ for $1\le j\le k$ and all $x\in V(G)$.    
Given these facts, we obtain $H^{k}$ 
 from $f_{k}$ as follows.  For each $1\le j\le k$ we interpret $f_{k}(x)_{j}$  
as a hypercube point in $Q_{e_{j}-e_{j-1}}$ using the inverse image of the labeling $L_{e_{j}-e_{j-1}}$ of Corollary \ref{labeling}.  We then concatenate 
these hypercube points (now strings over $\{0,1\}$) left to right in order of increasing $j$ to obtain $H^{k}(x)$. 
The details are as follows.

\begin{center} \textbf{Construction of the map $H^{k}: G = [a_{1}\times a_{2}\times a_{3} \times \cdots \times a_{k} ]\rightarrow Opt(G)$}
\end{center}

\noindent{\textbf{1.}} Initialization and the case $k=2$.

a) Start with the map $f_{2}: G\rightarrow Y_{2}^{u_{2}}\subseteq P(2^{e_{1}})\times P(2^{e_{2}-e_{1}})$.

b) Define $H^{2}: G\rightarrow Opt(G)$ by $H^{2}(x) = (L_{e_{1}}^{-1}(f_{2}(x)_{1}),L_{e_{2}-e_{1}}^{-1}(f_{2}(x)_{2}))$. 

\noindent{\textbf{2.}} For $k\ge 3$, construct the maps $f_{3}, f_{4}, \ldots, f_{k}$ inductively as follows.

For $i = 3$ to $k-1$, construct $f_{i+1}:G\rightarrow Y_{i+1}$ from $f_{i}$ using the procedure given at the 

beginning of this subsection.

\noindent {\textbf{3.}} Having obtained the map $f_{k}:G\rightarrow Opt'(G)$ from the preceding step, define 
the map $H^{k}: G\rightarrow Opt(G)$ by 
\begin{center}$H^{k}(x) = (H^{k}(x)_{1}, H^{k}(x)_{2}, \ldots,H^{k}(x)_{k})$, where $H^{k}(x)_{j} = L_{e_{j}-e_{j-1}}^{-1}(f_{k}(x)_{j})$ for $1\le j\le k$,
\end{center}
taking $e_{0} = 0$, and where $L_{e_{j}-e_{j-1}}$ is the labeling from Corollary \ref{labeling}.

 \section{The dilation bound}

From Lemma \ref{whatever}e3 we have the containment result $H^{k}(G)\subseteq Opt(G)$.  The goal in this section is to complete the proof of our main 
result by showing that $dilation(H^{k})\le 3k$ when every $a_{i}$ is larger than some fixed constant.

We recall some notation. For $x\in \langle Y_{i-1} \rangle$, recall that $Stack_{i}(x,r) = \{ z = \sigma_{i-1}(y): z_{1\rightarrow i-1}= x$ and $y\in S_{i-1}^{(r)}(G)  \}$, that  
$Stack_{i}'(x,r) = f_{i}(D_{i-1}^{(r)})\cap Stack_{i}(x,r)$, and that  $[r]_{i}$ = max$\{|Stack_{i}(x,r)|: x\in \langle Y_{i-1} \rangle  \}$.    
So by definition we have $Stack_{i}'(x,r)\subseteq Stack_{i}(x,r)$.

To set the context for the next theorem, note that by Lemma \ref{whatever}e1,e2, $f_{i}(D_{i-1}^{(r)})\subseteq Y_{i}^{([r]_{i})}$, so that $|Stack_{i}'(x,r)|\le [r]_{i}$   
In the next theorem we will see that $Stack_{i}'(x,r)$ is always an initial substack of $Stack_{i}(x,r)$; equivalently, that $Stack_{i}'(x,P_{i-1})$ is page monotone.    
Also we will see that
the ``page stack" heights 
 $|Stack_{i}'(x,r)|$, $x\in \langle Y_{i-1} \rangle $, fall within a narrow range; for a given $r$ (and fixed $i$) any two such heights differ by at most $2$ independent of $x$.

For a subset $S\subseteq Y_{i}^{([r]_{i})} $, let $v_{i,r}(S) = |S \cap f_{i}(D_{i-1}^{(r)}) |$.

 \begin{theorem} \label{missing points}
 Let $i\ge 3$ and $1\le r\le P_{i-1}$.  

\noindent{\rm{(a)}} (page monotonicity) Take $x\in \langle Y_{i-1} \rangle$. Let $z',z''\in Stack_{i}(x,P_{i-1})$ with $f_{i}^{-1}(z')\in D_{i-1}^{s}$ and $f_{i}^{-1}(z')\in D_{i-1}^{t}$.  If $z'_{i} > z''_{i}$, then $s\geq t$.

\noindent{\rm{(b)}}   For every stack address 
$x\in \langle Y_{i-1} \rangle $, we have $[r]_{i} - 2\le |Stack_{i}'(x,r)|\le [r]_{i}$. Moreover, 
 all points of $Y_{i}^{([r]_{i})} - f_{i}(D_{i-1}^{(r)})$ lie in the union $Y_{i}^{[r]_{i}}\cup Y_{i}^{[r]_{i}-1}$ of the top two $(i-1)$-levels  of $Y_{i}^{([r]_{i})}$.

\noindent{\rm{(c)}}  $v_{i,r}(Y_{i}^{[r]_{i} -1}) + v_{i,r}(Y_{i}^{[r]_{i}}) > 2^{e_{i-1}}$.
 
 \end{theorem} 
 
 \begin{proof} Consider part (a).  It suffices to show that for any  $x\in \langle Y_{i-1} \rangle$ and $1\le r\le P_{i-1}$, we have that $Stack_{i}(x,r)$ is page monotone.  We prove this 
 by induction on $r$, for any fixed $i\geq 3$ and $x\in \langle Y_{i-1} \rangle$.  For the base case $r = 1$, the claim is trivial since $Stack_{i}(x,1)$ contains a single entry by one to oneness of $\bar \sigma_{i-1}$ 
 (Lemma \ref{whatever}c) on any one section (in this case, on $S_{i-1}^{1}$).  So suppose inductively that $Stack_{i}(x,r)$ is  page monotone for some $1\le r < P_{i-1}$.  We use Lemma \ref{whatever}f, with $i-1$ in place of $i$.  
Applying $\sigma_{i-1}$ to the second containment $S_{i-1}^{(r)}(G)\subset I_{i-1}(f_{i-1}(D_{i-1}^{(r+1)}))$ 
stated there and noting that $f_{i} = \sigma_{i-1}\circ I_{i-1}\circ f_{i-1}$, 
we see that every entry $z$ of $Stack_{i}(x,r)$ satisfies $f_{i}^{-1}(z)\in D_{i-1}^{(r+1)}$.  If $Stack_{i}(x,r+1) = Stack_{i}(x,r)$, then trivially $Stack_{i}(x,r+1)$ 
is page monotone by induction.  So assume $Stack_{i}(x,r+1) \ne Stack_{i}(x,r)$, and let $y$ be the unique element (by one to oneness of $\bar \sigma_{i-1}$ on any section) of $Stack_{i}(x,r+1) - Stack_{i}(x,r)$.  
By the first containment in Lemma \ref{whatever}f (again with $i-1$ replacing $i$) we have, on applying $\sigma_{i-1}$ to each side again, $f_{i}^{-1}(y)\in D_{i-1}^{j}$ for some $j\geq r+1$.
Now $Stack_{i}(x,r)$ is page monotone by induction and as just noted $f_{i}^{-1}(z)\in D_{i-1}^{(r+1)}$ for all $z\in Stack_{i}(x,r)$.  Since $Stack_{i}(x,r+1)$ is obtained by placing 
$y$ at the top of $Stack_{i}(x,r)$ and $y\in f_{i}(D_{i-1}^{j})$ for some $j\geq r+1$, it follows that $Stack_{i}(x,r+1)$ is page monotone, completing the inductive step.

 For part (b), the upper bound $|Stack_{i}'(x,r)|\le [r]_{i}$ follows from $Stack_{i}'(x,r)\subseteq Stack_{i}(x,r)$ and the 
 definition of $[r]_{i}$.   
 
 Consider the lower bound on $|Stack_{i}'(x,r)|$ in part (b).  The statement holds vacuously for $r=1$ and all $i\geq 3$, since $|Stack_{i}'(x,1)| = 0$ or $1$ depending on $x$, and $[1]_{i} = 1$, by one to oneness 
 of the projection map $\bar{\sigma}_{i-1}$ on any section from Lemma \ref{whatever}c.  So suppose $r\geq 2$. 

Fixing $i$, for any stack address $x\in \langle Y_{i-1}  \rangle$, let $s(x,r) = |Stack_{i}(x,r)| $ 
and $s'(x,r) = |Stack_{i}'(x,r)| $.  By Lemma \ref{whatever}f we have $S_{i-1}^{(r-1)}(G)\subseteq (I_{i-1}\circ f_{i-1})(D_{i-1}^{(r)})$.
Now applying $\sigma_{i-1}$ to both sides of this containment and using $Stack_{i}(x,r-1)\subseteq Stack_{i}(x,r)$ and $f_{i} = \sigma_{i-1}\circ I_{i-1}\circ f_{i-1}$, we have 
$Stack_{i}(x,r-1) = [\sigma_{i-1}(S_{i-1}^{(r-1)}(G))\cap Stack_{i}(x,r-1)]\subseteq [f_{i}(D_{i-1}^{(r)})\cap Stack_{i}(x,r)] = Stack_{i}'(x,r)$.  Thus 
$s'(x,r)\geq s(x,r-1)$. 
Since $r-1\le P_{i-1}-1$, by Lemma \ref{whatever}d we have 
$s(x,r-1) = [r-1]_{i}$ or $[r-1]_{i}-1$.  Therefore $s'(x,r) \geq [r-1]_{i}-1$.  By Lemma \ref{whatever}c 
we have $[r-1]_{i} \le [r]_{i} \le [r-1]_{i}+1$.  Thus $s'(x,r) \geq  [r-1]_{i}-1\geq  [r]_{i} - 2$, proving the first sentence of (b).  

The second sentence of (b) follows from the first sentence, together with 
the page monotonicity property of part (a).

Next consider (c).  By Lemma  \ref{whatever}e1 we have $|Y_{i}^{([r]_{i})}| - |f_{i}(D_{i-1}^{(r)})|< 2^{e_{i-1}}$.  Let 
 $A = |Y_{i}^{[r]_{i}}\cup Y_{i}^{[r]_{i}-1}|$ and $B = v_{i,r}(Y_{i}^{[r]_{i} -1}) + v_{i,r}(Y_{i}^{[r]_{i}})$.  By (b) we have $Y_{i}^{([r]_{i}-2)}\subseteq f_{i}(D_{i-1}^{(r)})$, so 
 $|Y_{i}^{([r]_{i})}| - |f_{i}(D_{i-1}^{(r)})| = A - B$.  
So 
$2^{e_{i-1}} > |Y_{i}^{([r]_{i})}| - |f_{i}(D_{i-1}^{(r)})| = A - B = 2\cdot 2^{e_{i-1}} - B$, so $B > 2^{e_{i-1}}$, as required.   
\end{proof}
 
 We introduce notation for identifying particular 
 $(i-1)$-subpages of a given $i$-page in $G$.  For $1\le j\le a_{i}$ let $D_{i}^{r}(j) = D_{i-1}^{(r-1)a_{i}+j}$, and we regard $D_{i}^{r}(j)$ as the 
 $j$'th  $(i-1)$-subpage of $D_{i}^{r}$ under the ordering of $(i-1)$-subpages of $D_{i}^{r}$ induced by $\prec_{i-1}$.  Similarly 
 let $S_{i}^{r}(j)$, $j\geq 1$, be the $j$'th nonblank $(i-1)$-level of $S_{i}^{r}(G)$, ordered by increasing $i$-coordinate.

 Now suppose $z = (I_{i}\circ f_{i})(x)$ for some $x\in D_{i}^{r}$.  We let $\nu_{i}(z)$ be the integer such that $z\in S_{i}^{r}(\nu_{i}(z))$ for suitable $r$; that is, $z$ belongs to the $\nu_{i}(z)$'th  nonblank 
 $(i-1)$-level, ordered by increasing $i$-coordinate, of the $i$-section $S_{i}^{r}$ which contains $z$.  To illustrate, recall 
 the examples at the right of Figure \ref{7and12stacking} and the left of Figure \ref{stacking in 4D}.  In Figure \ref{7and12stacking} we have 
 the first $8$ many $2$-levels of $Y_{3}$ containing the image 
 $f_{3}(G)$, where $G = [3\times 7\times 4\times 3]$.  In Figure \ref{stacking in 4D} at left we have inserted $4$ blank $2$-levels among these (as specified by the matrix in Table 1c), and grouped the resulting 
 $12$ many $2$-levels into the three $3$-sections $S_{3}^{j}$, $1\le j\le 3$, preserving the order of the nonblank levels.  So for example the $6$'th $2$-level in Figure \ref{7and12stacking} 
 (ordered by $3$'rd coordinate, or height in the figure) becomes, 
 after the map $I_{2}$ inserts these $4$ blank $2$-levels, the $3$'rd nonblank 
 $2$-level $S_{3}^{2}(3)$ of section $S_{3}^{2}$ (again ordered by $3$'rd coordinate, or height within $S_{3}^{2}$) in Figure \ref{stacking in 4D}.  So any point $z = (I_{3}\circ f_{3})(x)$, 
 where $f_{3}(x)$ was in the $6$'th $2$-level of Figure \ref{7and12stacking}, satisfies $\nu_{3}(z) = 3$ since $z$ belongs to the third nonblank 
 $2$-level of the $3$-section (in this case $S_{3}^{2}$) containing $z$ (as shown in Figure \ref{stacking in 4D}).  Similarly take any point $f_{3}(x)$ lying in the $7$'th level of Figure \ref{7and12stacking}.  Then 
 the corresponding point $z = (I_{3}\circ f_{3})(x)$ after inflation satisfies $\nu_{3}(z) = 1$, since $z$ belongs to the first nonblank $2$-level  $S_{3}^{3}(1)$ of the $3$-section 
 $S_{3}^{3}$ containing $z$ (as shown in Figure \ref{stacking in 4D}).    
 
Let $m_{r}$ be the number of nonblank $(i-1)$-levels in $S_{i}^{r}$ for any $1\le r\le P_{i}$.  We interpret the $(i-1)$-levels $S_{i}^{r}(j)$ for integers $j$ outside the range $[1,m_{r}]$, and 
also interpret the differences $\nu_{i}(z') - \nu_{i}(z'')$ using ``wraparound" as follows. For example,  
if $j\le 0$, then $ S_{i}^{r}(j)$  is understood as $S_{i}^{r-1}(m_{r-1}+ j)$.  Similarly, if $j > m_{r}$, then 
$S_{i}^{r}(j)$ is understood as $S_{i}^{r+1}(j-m_{r+1})$.  Also let $z'\in S_{i}^{r}(\nu_{i}(z'))$ and $z''\in S_{i}^{s}(\nu_{i}(z''))$, where 
$1\le \nu_{i}(z')\le m_{r}$ and $1\le \nu_{i}(z'')\le m_{s}$.  Then define $|| \nu_{i}(z'') - \nu_{i}(z') ||$ as the minimum of 
$\{ | \nu_{i}(z'') - \nu_{i}(z') |, m_{r}-\nu_{i}(z')+\nu_{i}(z''), m_{s}-\nu_{i}(z'')+\nu_{i}(z')    \}$.

 \begin{corollary} \label{containment}  For $2\le i\le k$ we have the following.
 
 \medskip
 
 \noindent{\rm(a)} For each $i\geq 2$ and $j\geq 1$, we have $(I_{i}\circ f_{i})(D_{i}^{j})\subseteq S_{i}^{j-1}\cup S_{i}^{j}$.  Hence for $1\le r\le P_{i-1}$ we have the following.
 
(1) For any stack address $x\in \langle Y_{i-1} \rangle$ and $1\le r\le P_{i-1}$, $|Stack_{i}(x,r)\cap f_{i}(D_{i-1}^{r})|\le 2$.
 
(2) Suppose $|Stack_{i}(x,r)\cap f_{i}(D_{i-1}^{r})| = 2$, and let $Stack_{i}(x,r)\cap f_{i}(D_{i-1}^{r}) = \{ z', z'' \}$.  Then $z'$ and $z''$ are at successive heights in 
$Stack_{i}(x,r)$; that is, $|z'_{i} - z''_{i}|\le 1$.

(3) Let $z\in D_{i-1}^{r}$, with $f_{i}(z)\in Stack_{i}(x,r)$ for some $x\in \langle Y_{i-1} \rangle$.  Then $[r]_{i}-2\le f_{i}(z)_{i}\le [r]_{i}$.     
 
 \medskip

 \noindent{\rm(b)} For $x,y\in G$, suppose $x\in D_{i}^{r}(q)$ and $y\in D_{i}^{s}(q)$ for some $1\le q\le a_{i}$ and 
$1\le r,s\le P_{i}.$  Let $z' = (I_{i}\circ f_{i})(x)$ and  $z'' = (I_{i}\circ f_{i})(y)$. 
Then $||\nu_{i}(z') - \nu_{i}(z'')||\le 3.$

\medskip 

\noindent{\rm(c)} For $x,y\in G$, let $I_{i}(f_{i}(x))\in S_{i}^{s}(G)$ and $I_{i}(f_{i}(y))\in S_{i}^{t}(G)$, $1\le s\le t\le P_{i}$.  

(1) If $s=t$, then $|f_{i+1}(x)_{i+1} - f_{i+1}(y)_{i+1}|\le 1$.

(2) If $|s - t| = 1$, then $|f_{i+1}(x)_{i+1} - f_{i+1}(y)_{i+1}|\le 2$.

\medskip

\noindent{\rm(d)} Suppose $x\in D_{i}^{s}$, $y\in D_{i}^{t}$, $1\le s\le t\le P_{i}$.

(1) If $s=t$, then $|f_{i+1}(x)_{i+1} - f_{i+1}(y)_{i+1}|\le 2$.

(2) If $|s - t| = 1$, then $|f_{i+1}(x)_{i+1} - f_{i+1}(y)_{i+1}|\le 3$.

 \end{corollary}

 \begin{proof} Consider (a).  Since $(I_{i}\circ f_{i})(D_{i}^{(j)})\subseteq S_{i}^{(j)}$ by Lemma \ref{whatever}f, 
 it suffices to show that $(I_{i}\circ f_{i})(D_{i}^{j})\cap S_{i}^{(j-2)} = \emptyset$.  Recall that $D_{i}^{(j-1)} = D_{i-1}^{(a_{i}(j-1))}$.  By 
Lemma \ref{whatever}e1 we have $[a_{i}(j-1)]_{i} = \lceil \frac{|D_{i-1}^{(a_{i}(j-1))}|}{2^{e_{i-1}}}  \rceil   = \lceil \frac{(j-1)a_{i}a_{i-1}\cdots a_{1}}{2^{e_{i-1}}} \rceil$.  
 So by Lemma \ref{blank arithmetic} we see that $[a_{i}(j-1)]_{i}$ is the number of nonblank $(i-1)$-levels in $S_{i}^{(j-1)}$.
 Since $I_{i}$ preserves the order (by increasing $i$ coordinate) of $(i-1)$-levels, it follows that $I_{i}(Y_{i}^{([a_{i}(j-1)]_{i})}) = S_{i}^{(j-1)}(G)$.
 By Theorem  \ref{missing points}b, only the top two $(i-1)$-levels of $Y_{i}^{([a_{i}(j-1)]_{i})}$ may contain points not in 
 $f_{i}(D_{i-1}^{(a_{i}(j-1))}) = f_{i}(D_{i}^{(j-1)})$. Again since $I_{i}$ preserves the order of $(i-1)$-levels, it follows that only the top 
 two nonblank $(i-1)$-levels of $I_{i}(Y_{i}^{([a_{i}(j-1)]_{i})}) = S_{i}^{(j-1)}$, call 
 them $Y_{i}^{c}$ and $Y_{i}^{d}$, may contain points not in $(I_{i}\circ f_{i})(D_{i}^{(j-1)})$.          
Thus $(I_{i}\circ f_{i})(D_{i}^{j})\cap S_{i}^{(j-1)}\subset (Y_{i}^{c}\cup Y_{i}^{d})$. But since $a_{i}\geq 2^{22} > 4$ for all $i$, these top two nonblank $(i-1)$-levels 
$Y_{i}^{c}$, $Y_{i}^{d}$ of $S_{i}^{(j-1)}$ lie in the last $i$-section  $S_{i}^{j-1}$ of $S_{i}^{(j-1)}$.  So $(I_{i}\circ f_{i})(D_{i}^{j})\cap S_{i}^{(j-2)} = \emptyset$ 
as desired. 

We now consider the consequences (a1)-(a3) of (a), starting with (a1).  By (a) we have $(I_{i-1}\circ f_{i-1})(D_{i-1}^{r})\subseteq S_{i-1}^{r-1}\cup S_{i}^{r}$.  
Thus $f_{i}(D_{i-1}^{r})\subseteq \sigma_{i-1}(S_{i-1}^{r-1}(G))\cup \sigma_{i-1}(S_{i-1}^{r}(G))$. 
Also the projection 
map $\bar \sigma_{i-1}$ is one to one when restricted to any single $(i-1)$-section, in particular to $S_{i-1}^{r-1}(G)$ and to $S_{i-1}^{r}(G)$.  Thus $f_{i}(D_{i-1}^{r})$ can contribute at most 
two points to any single stack $Stack_{i}(x,r)$; namely, up to one point from each of $S_{i-1}^{r-1}(G)$ and $S_{i-1}^{r}(G)$, proving (a1).

Part (a2) follows immediately from page monotonicity of $Stack_{i}'(x,r)$ (Theorem \ref{missing points}(a)).

Consider (a3).  For the upper bound note that $(I_{i-1}\circ f_{i-1})(D_{i-1}^{r})\subset S_{i-1}^{(r)}$ by part (a).  Since $z\in D_{i-1}^{r}$, we have 
$f_{i}(z)_{i}\le$ max$\{ \sigma_{i-1}(u)_{i}: u\in S_{i-1}^{(r)}(G) \} = [r]_{i}$ by definition of $[r]_{i}$.  Consider now the lower bound in (a3). 
If $f_{i}(z)$ is the highest point in $Stack_{i}'(x,r)$, then (a3) follows immediately from Theorem \ref{missing points}b.  Otherwise, by part (a2) we have that $f_{i}(z)$ is the second highest 
point in $Stack_{i}'(x,r)$.  Let the highest such point point be $f_{i}(y)$, $y\in D_{i-1}^{r}$, so that $f_{i}(y)_{i} - f_{i}(z)_{i} =1$ again by part (a2).  As in the proof of (a1), we must have 
$(I_{i-1}\circ f_{i-1})(z)\in S_{i-1}^{r-1}$ and $(I_{i-1}\circ f_{i-1})(y)\in S_{i-1}^{r}$.  
Thus by stack monotonicity (Lemma \ref{whatever}g) $f_{i}(z)$ is the highest point in $Stack_{i}(x,r-1)$, so $f_{i}(z)_{i} = |Stack_{i}(x,r-1)|$.  Since $r-1 < P_{i-1}$, by Lemma \ref{whatever}d2 
we have $f_{i}(z)_{i} = |Stack_{i}(x,r-1)| \geq [r-1]_{i} - 1$.  It remains to check that  $[r-1]_{i} - 1\geq [r]_{i} - 2$, which follows directly from the formula  
$[r]_{i} = \lceil \frac{ra_{i-1}a_{i-2}\cdots a_{1}}{2^{e_{i-1}}}  \rceil$ proved in Lemma \ref{whatever}e1.

For (b), recall that $D_{i}^{r}(q) = D_{i-1}^{(r-1)a_{i}+q}$, with the same formula for $D_{i}^{s}(q)$ where $r$ is replaced by $s$.  Let $d_{i} = a_{i}a_{i-1}\dots a_{1}$.  
By part (a) we have $(I_{i-1}\circ f_{i-1})(D_{i-1}^{(r-1)a_{i}+q})\subseteq S_{i-1}^{(r-1)a_{i}+q-1}\cup S_{i-1}^{(r-1)a_{i}+q}$.  Hence using Lemma \ref{whatever}e1 
and the upper bound in part (a3), we have $f_{i}(x)_{i}\le [ (r-1)a_{i}+q ]_{i} = \lceil \frac{   ((r-1)a_{i}+q)d_{i-1} }{2^{e_{i-1}}} \rceil $.  Now by 
Lemma \ref{blank arithmetic}a and our construction 
we know that $S_{i}^{(r-1)}$ has $\lceil \frac{(r-1)d_{i}}{2^{e_{i-1}}} \rceil = \lceil \frac{(r-1)a_{i}d_{i-1}}{2^{e_{i-1}}} \rceil $ nonblank $(i-1)$-levels. Thus the first 
$\lceil \frac{(r-1)a_{i}d_{i-1}}{2^{e_{i-1}}} \rceil$ nonblank blank $(i-1)$-levels in the image $I_{i}(Y_{i}^{([ (r-1)a_{i}+q ]_{i}  )  }   )$ belong to $S_{i}^{(r-1)}$.     
Hence 
$\nu_{i}(z')\le \lceil \frac{   ((r-1)a_{i}+q)d_{i-1} }{2^{e_{i-1}}} \rceil -   \lceil \frac{(r-1)a_{i}d_{i-1}  }{2^{e_{i-1}       }   }           \rceil$.  Similarly, this time using 
Lemma \ref{whatever}e1 and the lower bound in (a3), we have 
$\nu_{i}(z'')\ge \lceil \frac{   ((s-1)a_{i}+q)d_{i-1} }{2^{e_{i-1}}} \rceil - 2 -   \lceil \frac{(s-1)a_{i}d_{i-1}  }{2^{e_{i-1}       }   }           \rceil$.  It follows that 
$\nu_{i}(z') - \nu_{i}(z'')\le 3$.  A symmetric argument interchanging the roles of $r$ and $s$ yields $\nu_{i}(z') - \nu_{i}(z'')\geq -3$, completing (c).

 Consider part (c), starting with (c1).  Let $x'\in \langle Y_{i} \rangle$ and $y'\in \langle Y_{i} \rangle$ be the stack addresses of the images $f_{i+1}(x)$ 
 and $f_{i+1}(y)$; that is $x' = f_{i+1}(x)_{1\rightarrow i}$ and $y' = f_{i+1}(x)_{1\rightarrow i}$.  
 Since $s-1\le P_{i}-1$, by Lemma \ref{whatever}d we have 
 $[s-1]_{i+1} -1 \le |Stack_{i+1}(x',s-1)|, |Stack_{+1}(y',s-1)|\le [s-1]_{i+1}$.  Since $I_{i}(f_{i}(x))\in S_{i}^{s}(G)$, we have 
 $f_{i+1}(x) = \sigma_{i}(I_{i}(f_{i}(x)))\in Stack_{i+1}(x',s) - Stack_{i+1}(x',s-1)$.  By Lemma \ref{whatever}c $f_{i+1}(x)$ is the only 
 member of $Stack_{i+1}(x',s) - Stack_{i+1}(x',s-1)$, and thus $|Stack_{i+1}(x',s)|\le |Stack_{i+1}(x',s-1)| +1$.  It follows that 
  $ [s-1]_{i+1} \le f_{i+1}(x)_{i+1} \le [s-1]_{i+1}+1 $, with the same inequality holding for $f_{i+1}(y)_{i+1}$.  Part (c1) follows.
  
  For (c2), suppose 
 without loss of generality that $s = t-1$.   
 Again using Lemma \ref{whatever}c, we get $[s]_{i+1}\le [t]_{i+1}\le [s]_{i+1}+1$.  
 Since $I_{i}(f_{i}(y)) \in S_{i}^{t}(G)$ we have  $f_{i+1}(y)_{i+1}\le [t]_{i+1}$ by definition of $[t]_{i+1}$.  Again since  $I_{i}(f_{i}(y)) \in S_{i}^{t}(G) = S_{i}^{s+1}(G)$ and using 
 Lemma \ref{whatever}d2 we have 
 $f_{i+1}(y)_{i+1} = 1 + |Stack_{i+1}(y',s)| \geq 1+ [s]_{i+1} -1 = [s]_{i+1}$.  Hence we have $[s]_{i+1}\le f_{i+1}(y)_{i+1} \le [s]_{i+1} +1$.      
 Since $s\le P_{i}-1$ and $f_{i+1}(x)$ is the topmost entry of $Stack_{i+1}(x',s)$, it follows by Lemma \ref{whatever}d2 that $[s]_{i+1}-1\le f_{i+1}(x)_{i+1}\le [s]_{i+1}$.  Part (c2) follows.

For (d), set $\alpha = I_{i}(f_{i}(x))$ and $\beta = I_{i}(f_{i}(y))$, and consider first (d1).  By part (a),   
$\alpha$ and $\beta$ lie in the same $i$-section or in successive $i$-sections $S_{i}^{s-1}$, $S_{i}^{s}$ of $Y_{i}$.
Hence (d1) follows directly from part (c).  For (d2), suppose $s = t-1$.  Again by part (a), $\alpha \in S_{i}^{s-1}\cup S_{i}^{s} = S_{i}^{t-2}\cup S_{i}^{t-1}$, 
and $\beta \in S_{i}^{t-1}\cup S_{i}^{t}$.  If $\alpha$ and $\beta$ belong to successive $i$-sections (e.g. $S_{i}^{t-2}$ and $S_{i}^{t-1}$, 
or $S_{i}^{t-1}$ and $S_{i}^{t}$ respectively), then $|f_{i+1}(x)_{i+1} - f_{i+1}(y)_{i+1}| \le 2$ by part (c).  So we can suppose that 
$\alpha \in S_{i}^{t-2}$ and $\beta \in S_{i}^{t}$.  
Since the projection $\bar \sigma_{i}$ is one to one on any one $i$-section (Lemma \ref{whatever}c), 
we have $[t]_{i+1} \le [t-2]_{i+1}+2$.  Since $t-1\le P_{i}-1$ we have 
$f_{i+1}(x)_{i+1} \geq [t-2]_{i+1}-1$ by Lemma \ref{whatever}d, while $f_{i+1}(y)_{i+1} \le [t]_{i+1}$ since $\beta \in S_{i}^{t}$.  Thus $|f_{i+1}(x)_{i+1} - f_{i+1}(y)_{i+1}| \le 3$, completing (d2).   \end{proof}

We now proceed to a bound on the dilation of our embedding.  In the proof we frequently apply Corollary \ref{zeros column index}, where 
 (in the language of its statement) we use $N_{u}(d)$ relative to the matrix $F^{X} = (f_{uv})$ with the settings in part (d) of that Corollary.  By the 
 construction given immediately after that Corollary, we have $F^{X} = F(i)$.  
 The value of $i$ will change from one application to another.  So in the proof which follows, let $N_{i,u}(d)$ denote the $N_{u}(d)$ of the corollary, 
 when in the application we intend to use $F^{X} = F(i)$.  So for example, in applying Corollary \ref{zeros column index} with $e=3$ and $F^{X} = F(2)$, we would conclude 
 that $|N_{2,u}(d+3) - N_{2,u}(d)| \le 2\cdot 3 +2 = 8$, using the more generous of the bounds (a) and (b).   
 \begin{theorem}\label{dilation bound}
 
 Let $G = [ a_{1}\times a_{2}\times \cdots \times a_{k}]$ be a $k$-dimensional grid with $a_{i}\geq 2^{22}$ for each $i$.  Then the embedding 
 $H^{k}:G\rightarrow Opt(G)$ from section \ref{higher construction} satisfies $dilation(H^{k})\le 3k$. 
 
 \end{theorem}
 
 \begin{proof}Let $x = (x_{1}, x_{2}, \ldots, x_{k})$ and $y = (y_{1}, y_{2}, \ldots, y_{k})$ be arbitrary adjacent points of $G$, and set 
 $Q = Opt(G)$ and $f = f_{k}$ for short.  Recall that 
 \newline $H^{k}(x) = (H^{k}(x)_{1}, H^{k}(x)_{2}, \ldots, H^{k}(x)_{j},\ldots, H^{k}(x)_{k})\in Q$, 
 where $H^{k}(x)_{j} = L_{e_{j}-e_{j-1}}^{-1}(f(x)_j)$ is the $(0,1)$ string of $Q_{e_{j}-e_{j-1}}$ equivalent 
 to $f(x)_{j}$ under the labeling $L_{e_{j}-e_{j-1}}$ of Corollary \ref{labeling}.  It suffices to show for each $j$, $1\le j \le k$, that $|f(x)_{j} - f(y)_{j}|\le 17$.  For then 
 by the definition of $H^{k}$ just given and Corollary \ref{labeling}b we have $ dist_{Q_{e_{j}-e_{j-1}}}(H^{k}(x)_{j}, H^{k}(y)_{j})\le 3$ for each $1\le j\le k$.  So we obtain 
 $dist_{Q}(H^{k}(x),H^{k}(y)) = \sum_{j=1}^{k}dist_{Q_{e_{j}-e_{j-1}}}(H^{k}(x)_{j}, H^{k}(y)_{j})\le 3k$, and our desired dilation bound follows.  Thus we are reduced to showing 
 that $|f(x)_{j} - f(y)_{j}|\le 17$ for each $j$.

 Let $i_{0}:= i_{0}(x,y)$, 
 $1\le i_{0}\le k$, be the unique coordinate at which the above adjacent points $x, y\in V(G)$ disagree, so $|x_{i_{0}} - y_{i_{0}}| = 1$.  Also let 
 $L(j,i_{0})=  max\{ |f(x)_{j} - f(y)_{j}|: xy\in E(G), |x_{i_{0}} - y_{i_{0}}| = 1 \}$.  For any given $j$, the maximum of  $L(j,i_{0})$ 
 over all $i_{0}$ serves as an upper bound for  $|f(x)_{j} - f(y)_{j}|$, over all edges $xy\in E(G)$.  So it suffices to show that that for each 
 $1\le j\le k$ this maximum is at most $17$.  Let $x(j)y(j)$ be an 
 edge at which this maximum occurs for a given $j$, and refer to this edge simply as $xy$ since $j$ will be 
 understood by context.

 We begin with $j=1$, starting with the case $i_{0} = 1$ for a bound on $L(1,1)$.  Since $i_{0} = 1$, $x$ and $y$ are corresponding points on successive chains 
 $C_{i}$ and  $C_{i+1}$ of $G(a_{1})$.  So we have $L(1,1) = |f(x)_{1} - f(y)_{1}| = |f_{2}(x)_{1} - f_{2}(y)_{1}|\le 3 < 17$ as required, where the second equality follows 
 from Lemma \ref{whatever}a and the inequality following that from Corollary \ref{2-dim dilation}a.   Now suppose $i_{0} \geq 2$.  Thus $x$ and $y$ agree in their first coordinate, so can be considered as 
 lying on the same chain of $G(a_{1})$.  So similarly by by Lemma \ref{whatever}a and 
 Theorem \ref{two dim properties}h we have  $L(1,i_{0}) = |f(x)_{1} - f(y)_{1}| = |f_{2}(x)_{1} - f_{2}(y)_{1}|\le 2 < 17$.  So max$\{ L(1,i_{0}): 1\le i_{0}\le k \} \le 17$.

 Now suppose that $j=2$.  If $i_{0} = 1$, then again $x$ and $y$ are corresponding points on successive chains of $G(a_{1})$.  Thus 
 $|f_{2}(x)_{2} - f_{2}(y)_{2}|\le 1$ by Theorem \ref{two dim properties}g.  Hence by Lemma \ref{whatever}b (with $e=1$) we get $L(2,1) = |f(x)_{2} - f(y)_{2}|\le 4 < 17$.

 If $i_{0} = 2$, then $x$  and $y$ are successive points on the same chain of $G(a_{1})$, so by Corollary \ref{2-dim dilation}a we have 
 $ |f_{2}(x)_{2} - f_{2}(y)_{2}|\le 1$.  So again by Lemma \ref{whatever}b, $L(2,2) = |f(x)_{2} - f(y)_{2}|\le  4 < 17$.

 Now suppose $i_{0}\geq 3$, still with $j=2$.  Then $x$ and $y$ belong to the same chain $C_{i}$ of $G(a_{1})$.  
 Since $x_{1} = y_{1}$, $x_{2} = y_{2}$, and $i_{0}\geq 3$, we can view $x$ and $y$  
 as corresponding points belonging to a pair of distinct 
 $2$-pages of $G$, say $x\in D_{2}^{u}$ and  $y\in D_{2}^{v}$.  Consider the segments 
 $T_1$ and $T_2$ of $C_{i}$ given by $T_{1} = \{(i,(u-1)a_{2} + t): 1\le t\le x_{2} \}$ and 
 $T_{2} = \{(i,(v-1)a_{2} + t): 1\le t\le x_{2} \}$.  We see that  $f_{2}(T_{1})$ (resp. $f_{2}(T_2)$) is the initial segment of $x_{2}$ points of the chain 
 $C_{i}$ in $D_{2}^{u}$ (resp. $D_{2}^{v}$), and that $f_{2}(x)$ and $f_{2}(y)$ are the last points of these segments respectively.  Let $c_{1}$ 
 and $c_{2}$ be the number of successive columns of $Y_{2}$ spanned by $f_{2}(T_1)$ and $f_{2}(T_2)$ respectively.  By Corollary  \ref{2-dim dilation}d we have 
 $|c_{1} - c_{2}|\le 1$.  The columns spanned by $f_{2}(T_{1})$ (resp. $f_{2}(T_{2})$) are transformed under the inflation map 
 $I_{2}$ into a corresponding set of $c_{1}$ (resp. $c_{2}$) successive nonblank columns in $(I_{2}\circ f_{2})(G)$. Corollary  \ref{2-dim dilation}e and Lemma \ref{blank arithmetic} 
 now imply that $(I_{2}\circ f_{2})(T_1)$ and $(I_{2}\circ f_{2})(T_2)$, being transforms of $f_{2}(T_{1})$ and $f_{2}(T_{2})$ respectively, 
 each begin in either the first nonblank column of their 
 section ($S_{2}^{u}$ for $(I_{2}\circ f_{2})(T_1)$ and $S_{2}^{v}$ for $(I_{2}\circ f_{2})(T_2)$) or the last nonblank column of their 
 preceding section ($S_{2}^{u-1}$ for $(I_{2}\circ f_{2})(T_1)$ and $S_{2}^{v-1}$ for $(I_{2}\circ f_{2})(T_2)$).  Now let $z' = (I_{2}\circ f_{2})(x)$ and 
 $z''= (I_{2}\circ f_{2})(y)$, and assume by symmetry that $c_{2}\geq c_{1}$.  Then  $|c_{1} - c_{2}|\le 1$ implies that 
 $\nu_{2}(z'') - \nu_{2}(z') = 0, 1$ or $2$.  Assume first that 
  $\nu_{2}(z'') - \nu_{2}(z') = 2$, and write $d = \nu_{2}(z')$ and $d+2 =  \nu_{2}(z'')$.  This case arises when $c_{2} - c_{1} = 1$, and $(I_{2}\circ f_{2})(T_2)$ begins in the first 
 nonblank column of $S_{2}^{v}$ while $(I_{2}\circ f_{2})(T_1)$ begins in the last nonblank column of $S_{2}^{u-1}$.  Then by Lemma \ref{whatever}a 
 and Corollary \ref{zeros column index}c with $e=2$, we get 
 $|f(x)_{2} - f(y)_{2}| = |f_{3}(x)_{2} - f_{3}(y)_{2}| =  |N_{2,u}(\nu_{2}(z')) - N_{2,v}(\nu_{2}(z''))| = |N_{2,u}(d) - N_{2,v}(d+2)|\le 2\cdot 2 + 4 = 8$. In the case $\nu_{2}(z'') - \nu_{2}(z') = 1$, we similarly obtain 
 (using $e = 1$ in Corollary \ref{zeros column index}c) that  $|f(x)_{2} - f(y)_{2}|\le 6$, and in the case 
 $\nu_{2}(z'') - \nu_{2}(z') = 0$ (using $e = 0$) we get $|f(x)_{2} - f(y)_{2}|\le 4$.  So we get $L(2,i_{0}) \le 8 < 17$ for $i_{0}\geq 3$, and overall max$\{ L(2,i_{0}): 1\le i_{0}\le k \} \le 17$. .

 Next suppose $j\geq 3$.  We argue according to the order relation between $j$ and $i_{0}$.  
 
 Suppose first that $j > i_{0}$.  Then 
 $x$ and $y$ belong to the same $(j-1)$-page and successive $(i_{0}-1)$-pages of $G$.  We obtain our bound on $L(j,i_{0})$ 
 independent of $i_{0}$.  Since $x$ and $y$ belong to the same $(j-1)$-page we have that     
 $I_{j-1}(f_{j-1}(x))$ and $I_{j-1}(f_{j-1}(y))$ belong to the same 
 $(j-1)$-section or to successive $(j-1)$-sections $S_{j-1}^{t}$ and $S_{j-1}^{t'}$, $|t - t'|\le 1$, by Corollary \ref{containment}a.  Applying Corollary \ref{containment}c, 
 in the first case (same $(j-1)$-section) we get $|f_{j}(x)_{j} - f_{j}(y)_{j}|\le 1$ by part (c1), while in the 
 second case (successive $(j-1)$-sections) we get $|f_{j}(x)_{j} - f_{j}(y)_{j}|\le 2$ by part (c2).  Now we bound $L(j,i_{0})$ by applying 
Lemma \ref{whatever}b (with $e=2$) to obtain $L(j,i_{0}) = |f(x)_{j} - f(y)_{j} | \le 6 < 17$, as required.

 Now suppose $j = i_{0}$. Then $x$ and $y$ are corresponding points in successive $(j-1)$-pages of $G$.  Hence by Corollary \ref{containment}d2 
 we get 
 $|f_{j}(x)_{j} - f_{j}(y)_{j}|\le 3 $.  Thus by Lemma \ref{whatever}b with $e = 3$ we get 
 $L(j,j) = |f(x)_{j} - f(y)_{j}|\le 8 < 17$.

It remains to consider the case $3\le j< i_{0}$.  Here $x$ and $y$ are corresponding points in successive $(i_{0} - 1)$-pages, say $x\in D_{i_{0}-1}^{r}$ and 
$y\in D_{i_{0}-1}^{r+1}$. Consider the $j$-pages containing $x$ and $y$, say $x\in D_{j}^{c}\subset D_{i_{0}-1}^{r}$ and $y\in D_{j}^{d}\subset D_{i_{0}-1}^{r+1}$ 
for suitable integers $c$ and $d$.  Since $x$ and $y$ agree in their first $j$ coordinates, they must also be corresponding points in these $j$-pages.  In particular, 
$x$ and $y$ belong to corresponding $(j-1)$-subpages $D_{j}^{c}(q)$ of $D_{j}^{c}$ and $D_{j}^{d}(q)$ of $D_{j}^{d}$ respectively for the same integer $q$.  
As before, let $z' = (I_{j}\circ f_{j})(x)$ and $z'' = (I_{j}\circ f_{j})(y)$ for brevity.  Then by Corollary \ref{containment}b we get 
 $|\nu_{j}(z') - \nu_{j}(z'') |\le 3 $.  Again applying Lemma \ref{whatever}a we have $f(x)_{j} = N_{j,c}(\nu_{j}(z'))$ and $f(y)_{j} = N_{j,d}(\nu_{j}(z''))$. 
 So applying Corollary \ref{zeros column index}c with $e = 3$ we get   
 $L(j,i_{0}) = |f(x)_{j} - f(y)_{j}|\le |N_{j,c}(\nu_{j}(z')) - N_{j,d}(\nu_{j}(z''))| \le 2\cdot 3 + 4 = 10 < 17$.  So for any $j\geq 3$ we have max$\{ L(j,i_{0}): 1\le i_{0}\le k\}\le 17$, completing the proof.  
 \end{proof}
 
 \section{Concluding Remarks}
 
 1. There is a routing of edge congestion $O(k)$ associated to our embedding $H^{k}$.  We outline the idea here, omitting full details of the 
 proof.
 
 As notation, for any graph $H$ and permutation $\pi: V(H)\rightarrow V(H)$, a $\pi$\textit{-routing} is an assignment $P: V(H)\rightarrow$ $\{$paths in $H \}$
 such that $P(x)$ is a path in $H$ from $x$ to $\pi(x)$.
 If $H$ is directed, then $P(x)$ is a directed path 
 from $x$ to $\pi(x)$.  The \textit{congestion} of a $\pi$-routing is the maximum of $\{ n(e): e\in E(H) \}$, where $n(e)$ is the number of 
 paths in the $\pi$-routing which use the edge $e$. 
 
 For background, let $\overrightarrow {Q_{n}}$ be the directed graph obtained from $Q_{n}$ by replacing each edge of $Q_{n}$ by $4$ directed edges, two pointing in one 
 direction and two in the opposite direction.  Using the classic Benes routing method, one can show \cite{GT} that for any permutation $\pi$ of $\overrightarrow {Q_{n}}$ 
 there exists a $\pi$-routing such that the paths of the routing are edge disjoint.  Consequently, for each permutation $\pi$ of the undirected $Q_{n}$ 
 there is a $\pi$-routing with congestion $O(1)$.
 
 For an undirected graph $H$, suppose there is a partition of $E(H)$ into $k$ sets, $E(H) = \bigcup_{i=1}^{k}E_{i}$, such that each $E_{i}$ is a vertex disjoint union 
 of cycles in $H$, where the vertices of each cycle are ordered in one of the two natural ways.  Further, let $g: H\rightarrow  Q_{n}$ be a one to one map.  
 Now consider the permutation $\pi_{i}$ on $Q_{n}$  as follows: if $v\notin g(H)$ then $\pi_{i}(v) = v$, 
 while if $v\in g(H)$ then $\pi_{i}(v) = w$, where $g^{-1}(v)g^{-1}(w)\in E_{i}$  
 and $g^{-1}(w)$ follows $g^{-1}(v)$ in the ordering of the vertices of the cycle of $E_{i}$ containing $g^{-1}(v)$ and $g^{-1}(w)$.  Then by the above result there is a $\pi_{i}$ routing on $Q_{n}$ of congestion 
 $O(1)$ for each $i$, $1\le i\le k$.  Since each $e\in E(H)$ lies in some $E_{i}$, it would follow that the edge congestion of $g$ (as defined in section \ref{subintro}) is $O(k)$.
 
It therefore suffices to 
 find a graph $G'\supset G = [a_{1}\times a_{2}\times \ldots \times a_{k}]$ with $V(G') = V(G)$ such that $G'$ has the required cycle partition of edges.  For then 
 (letting our map $H^{k}$ play the role of $g$ in the above paragraph) the map 
 $H^{k}: G'\rightarrow Opt(G)$ has edge congestion $O(k)$ by the above argument, so the same is true of its restriction  $H^{k}: G \rightarrow Opt(G)$.

The graph $G'$ is obtained from $G$ as follows.  
For each $i$, $1\le i\le k$, consider a 
 fixed $(k-1)$-tuple $\vec{c}(i) = (c_{1}, c_{2}, \cdots ,c_{i-1}, c_{i+1}, \cdots ,c_{k})$ of 
 integer entries, with $1\le c_{j}\le a_{j}$ for $j\ne i$.  Further, for $1\le t\le a_{i}$ let  
 $v(\vec{c}(i),t) = (c_{1}, c_{2}, \cdots ,c_{i-1}, t, c_{i+1}, \cdots ,c_{k})\in V(G)$.
We let $V(G') = V(G)$ and $E(G') = E(G)\cup E'$, where 
$E' = \{ v(\vec{c}(i),a_{i}) v(\vec{c}(i),1): 1\le i\le k,$ $\vec{c}(i)$ any $(k-1)$-tuple as above$\}$.  Thus $E'$ is just the set of 
``wraparound" edges not present in $G$ in each of the $k$ dimensions.  The required cycle partition of $E(G')$ is given by 
$E(G') = \bigcup_{i=1}^{k} E_{i}'$, where 
$E_{i}' = \{ v(\vec{c}(i),a_{i}) v(\vec{c}(i),1), v(\vec{c}(i),t) v(\vec{c}(i),t+1): 1\le t\le a_{i}-1, 1\le i\le k$, $\vec{c}(i)$ any $(k-1)$-tuple as above$\}$.

2. The lower bound requirement $a_{i}>2^{22}$ for our result can be relaxed to $a_{i}>2^{12}$, provided one can improve the conclusion of Corollary \ref{labeling}a 
only slightly to say $|L_{t}(x) - L_{t}(y)|\le 2r+4 \Rightarrow dist_{Q_{t}}(x,y)\le 3$.  Then using $r=3$ and $i=2$ we obtain 
$|L_{t}(x) - L_{t}(y)|\le 10 \Rightarrow dist_{Q_{t}}(x,y)\le 3$ for $t\geq 12$, so that $a_{i}>2^{12}$ suffices for our result.  The proof of Theorem \ref{dilation bound} is then reduced (as in its first paragraph) 
to proving the inequality $|f(x)_{j} - f(y)_{j}|\le 10$ for each $1\le j\le k$ and $xy\in E(G)$.  The proof of Theorem \ref{dilation bound} shows that this inequality does indeed hold.  The improvement in 
Corollary \ref{labeling}a may require a detailed study of the regular cyclic caterpillars we used.     

3. The question of finding good lower bounds for $B(G,Opt(G))$, for some class of multidimensional grids $G$ with $|V(G)|\rightarrow \infty$, remains 
open.  A nontrivial lower bound for all multidimensional grids is of course not possible, since if the $a_{i}$ are all powers of $2$, then $G$ is a spanning subgraph of 
$Opt(G)$, so $B(G,Opt(G)) = 1$ in that case.

4. The first author thanks Stephen Wright for many useful discussions on rounding of matrices.  We also acknowledge a useful discussion with Hal Kierstead which inspired our 
application of regular cyclic caterpillars.

\section{Appendix 1: Proofs of Theorem  \ref{two dim properties} and Corollary \ref{2-dim dilation}} 
 
\underline{Proof of Theorem \ref{two dim properties}}: Part (a) follows directly from steps 2a and 2b of the construction of $f$.  In part (b), the claim that $L_{r}(j)$ is an initial segment 
of  $Y_{2}^{j}$ follows directly from part (a), step 2 of the construction of $f$, and induction on $r$.  The claim on $|L_{r}(j)|$ follows 
from step 2 of this construction on noting that $C_i$ contributes $1$ (resp. $2$) to this sum if $R_{ij} = 0$ (resp. $R_{ij} = 1$).

Consider (c).  By part (a), $\sum_{t=1}^{j}R_{it}$ is the number of columns $Y_{2}^{t}$, $1\le t\le j$, such that 
$|f(C_{i})\cap Y_{2}^{t}| = 2$, while in the remaining columns $Y_{2}^{t}$, $1\le t\le j$, we have $|f(C_{i})\cap Y_{2}^{t}| = 1$.  
So $\pi(i,1,j) = N_{ij}$ as claimed.
Also observe that $\pi(i,r\rightarrow r+j) = N_{i,r+j} -  N_{i,r-1} =  j+1+ \sum_{t=r}^{r+j}R_{it} \in \{j+1+S_{j+1},  j+2+S_{j+1} \}$ 
by the formula for $N_{ij}$ and Theorem \ref{matrix R properties}b.     
The same equation 
with $r$ replaced by $s$  
yields the rest of (c). 
 
Part (d) follows from parts (b), (c), Lemma \ref{matrix R properties}b, and the fact that $S_{r+1} - S_{r} \le 1$.  For (e), by part (c) (taking $j=m$) and the 
definition of $G(a_{1},N_{im})$ we see that $f(G(a_{1},N_{im}))\subseteq Y_{2}^{(m)}$.  So it suffices to show that $|f(G(a_{1},N_{im}))| = |Y_{2}^{(m)}|$.  
We have $|f(G(a_{1},N_{im}))| = \sum_{i=1}^{a_1} \sum_{j=1}^{m}N_{ij} 
=  \sum_{i=1}^{a_1}(m + \sum_{t=1}^{m}R_{it}) = a_{1}m + m(2^{L_1}-a_{1}) = m2^{L_1} = |Y_{2}^{(m)}|$, using Lemma \ref{matrix R properties}a. 

Consider (f). Recall the notation $P_{1} = a_{2}a_{3}\cdots a_{k}$ for the number of $1$-pages in $G$. 
We use the fact that $G$ can be identified with the the subgraph of $G(a_1)$ induced by the union 
of initial segments $\bigcup_{i=1}^{a_1}C_{i}(P_1)$ via the map $\kappa$ of section \ref{the 2D mapping}.  
For $j = m-1$ or $m$ let $M_{j}' = min\{N_{ij}: 1\le i\le a_{1} \}$, 
and $M_{j}'' = max\{N_{ij}: 1\le i\le a_{1} \}$.  
We claim that $M_{m}' \geq  P_1$.  
If not, then $M_{m}' <  P_1$, so $M_{m}''\le P_1$ by (d).  It follows that 
$|Y_{2}^{(m)}| = |f(G(a_{1},N_{im}))| = \sum_{i=1}^{a_1}N_{im} < a_{1}P_{1} = |G| $.  But this contradicts 
$|Y_{2}^{(m)}| = m2^{e_1} \geq |G|$, proving that $M_{m}'\geq  P_1$.  Hence for each $1\le i\le a_1$
we have $N_{im}\geq P_{1}$, implying that 
$G\subseteq G(a_{1},N_{im})$ by the identification of $G$ above.  To show $Y_{2}^{(m-1)}\subset f(G)$, observe
from the definition of $m$ that $|Y_{2}^{(m-1)}| < |G|$.  Thus $\sum_{i=1}^{a_1}N_{i,m-1} = |Y_{2}^{(m-1)}| < |G| = a_{1}P_{1}$.  
So one of the terms in the last sum is less than $P_1$.  It follows that $M_{m-1}' < P_1$ and hence  $M_{m-1}'' \le P_1$ 
by (d).  This says that for each $i$, $f(i,P_{1} + 1)_{2}\geq m$, and hence that $f(i,j)_{2}\geq m$ for 
$j\ge P_{1} + 1$ by the monotonicity of $f$ from part (a).   
It follows that 
$Y_{2}^{(m-1)}\subset f( \bigcup_{i=1}^{a_1}C_{i}(P_{1})) = f(G)$, as required.

For (g), we induct on $r$.  The base case is clear since $f(i,1)_{2} = 1$ for all $i$.  Assume inductively 
that $|f(i,r)_{2} - f(j,r)_{2}| \le 1$ for some $r>1$, and let $c = f(i,r)_{2}$.  
If $f(j,r)_{2} = f(i,r)_{2} = c$, then we are done since $ f(i,r+1)_{2} = c$ or $c+1$ and the same holds for $f(j,r+1)_{2}$.  
So suppose $|f(i,r)_{2} - f(j,r)_{2}| = 1$, and without loss of generality that $f(j,r)_{2} = c+1$ (the case $f(j,r)_{2} = c-1$ being symmetric).  Then 
$N_{j,c} \le r-1$ by (c) and $\sum_{t=1}^{c}R_{it} = 1+ \sum_{t=1}^{c}R_{jt}$ by Lemma \ref{matrix R properties}.  
If $f(j,r+1)_{2} = f(j,r)_{2} = c+1$, then we are 
done since $f(i,r)_{2} \le f(i,r+1)_{2} \le f(i,r)_{2} + 1$.  So we may assume that $f(j,r+1)_{2} = c+2$.  If now 
$f(i,r+1)_{2} = c+1$, then we are done again.  So we can also assume that $f(i,r+1)_{2} = c$.  Then 
we get $N_{ic} \geq r+1 \geq N_{jc} +2$, contradicting $\sum_{t=1}^{c}R_{it} = 1+ \sum_{t=1}^{c}R_{jt}$.

Consider (h).  Set $ c = f(i,y)_2$ for arbitrary $y$, $1\le y\le N_{im}$.  It suffices to show that $f(i,y)_1$ 
must be one of three consecutive integers which depend only on $i$, but not on $y$.     
By (b), the set $L_{i-1}(c)$ is the initial segment of 
size $i-1 + \sum_{t=1}^{i-1} R_{tc}$ in $Y_{2}^{c}$.  So by Lemma \ref{matrix R properties}, $|L_{i-1}(c)|$ 
is either $i-1 + S_{i-1}$ or  $i + S_{i-1}$.  Now $f(C_i)\cap Y_{2}^{c}$ consists of either one or two successive points 
of $Y_{2}^{c}$ which immediately follow the point $(|L_{i-1}(c)|,c)$.  Hence $f(i,y)$ must be one of the 
three successive integers $i+S_{i-1}, i+1+S_{i-1}$, or $ i+2+S_{i-1}$, proving (h).

Finally consider (i).  The proof is based on $R^{j+1}$ 
being a downward shift of $R^{j}$ (with wraparound) for any $j$.  We refer to this property as the ``downward shift" property.  
The assumption $|f(C_{r})\cap Y_{2}^{j}| = 2$ says that $R_{rj} = R_{r-j+1,1} = 1$ (viewing subscripts modulo $a_1$), 
by the downward shift property.  Hence by part (b) and this same  property, 
we have $|L_{r}(j)| -  |L_{r}(j+1)| = R_{rj} - R_{1,j+1}\geq 0$, since $R_{rj} = 1$.  For the second claim, observe that 
$N_{r,j} = j +  \sum_{i=1}^{j}R_{ri} = j +  \sum_{i=r-j+1}^{r}R_{i1}$ by the downward shift property.  Similarly we have 
$N_{r+1,j} = j +  \sum_{i=r-j+2}^{r+1}R_{i1}$.  Therefore $N_{r,j} - N_{r+1,j} = R_{r-j+1,1} - R_{r+1,1}\geq 0$, since 
$ R_{r-j+1,1} = 1$.  

This completes the proof of the theorem.

\bigskip

\noindent \underline{Proof of Corollary \ref{2-dim dilation}}: As a convenience, we prove these properties with $f$ replacing $f_2$, to facilitate direct 
 reference to the construction above.  Of course $f_2$ then inherits these properties since it is a restriction of $f$.  
 For brevity let $d_{1} = |f(v)_{1} - f(w)_{1}|$ and $d_{2} = |f(v)_{2} - f(w)_{2}|$.
 
 Consider part (a).  We get $d_{2}\le 1$ by Theorem \ref{two dim properties}a if $v$ and $w$ are successive 
 points on the same chain, and by Theorem \ref{two dim properties}g otherwise (i.e., if $v = (i,j)$ and $w = (i+1,j)$ 
 for some $i$ and $j$).  We also get $d_{1}\le 2$ by Theorem \ref{two dim properties}h if $v$ and $w$ are successive 
 points on the same chain. So we are reduced to showing that $d_{1}\le 3$ 
when, say, $v = (i,j)$ and $w = (i+1,j)$ for some $i$ and $j$.  Let  $f(v)_{2} = c$.  If also  $f(w)_{2} = c$ 
 then since by Theorem \ref{two dim properties}a,b we have that $(f(C_i)\cup f(C_{i+1}))\cap Y_{2}^{c}$ is a set 
 of at most $4$ successive points of $ Y_{2}^{c}$, it follows in this case that $d_{1}\le 3$.  
 
 So let $f(w)_{2} = c'\ne c$, noting that then $|c - c'| = 1$ since
 we have already shown that $d_{2}\le 1$.  By Theorem  \ref{two dim properties}a,b we have 
 $f(v)_{1} = |L_{i}(c)|$ or $|L_{i}(c)|-1$, and similarly  $f(w)_{1} = |L_{i+1}(c'+1)|$ or $|L_{i}(c')|-1$.  By 
 Theorem  \ref{two dim properties}d we have $| |L_{i+1}(c')| - |L_{i}(c)| |\le 2$, so $d_{1}\le 3$ follows, completing part (a).         
 
 The proof of (b) is a tedious case analysis showing that the bound on $d_{1}$ from part (a) 
 can be strengthened to $d_{1}\le 2$, which combined with (a) yields (b).  The proof is omitted here for brevity. 
 We include (b) only for its possible interest, and do not use it later.  
   
Consider part (c).  Since  $m = \lceil \frac {|G|}{2^{e_{1}}} \rceil$, we see that $Y_2^{(m)}$ is a subgraph of the spanning subgraph 
$Y_{2}^{(2^{e_{k}-e_{1}})}$ of $Opt(G)$.  Thus by Theorem \ref{two dim properties}e,f we obtain 
$f_{2}(G)\subset f(H) = Y_{2}^{(m)}\subset Opt(G)$, as required.

Consider next part (d).  As notation, for $a\le b$ let $Y_{2}^{(a\rightarrow b)} = \bigcup_{t=a}^{b}Y_{2}^{t}$ be the 
union of columns $a$ through $b$ of $Y_2$.  Let then $Y_{2}^{(r\rightarrow r+c-1)}$ and $Y_{2}^{(s\rightarrow s+c'-1)}$ be the sets 
of columns of $Y_2$ spanned by $f_{2}(T)$ and $f_{2}(T')$ respectively.  
Suppose to the contrary that 
$|c - c'|\geq 2$.  By symmetry we can assume that $c'\geq c+2$.

Let $S = f_{2}(T)\cap Y_{2}^{(r\rightarrow r+c-1)}$, and let $S' =  f_{2}(T')\cap Y_{2}^{(s+1\rightarrow s+c)}$.  
Now for each $t, 1\le t\le s+c'-1,$ we have $f_{2}(T')\cap Y_{2}^{t} = C_{j}\cap Y_{2}^{t}$ with the possible exception 
when $t=s$ (resp. $t=s+c'-1$), where possibly $f_{2}(T')$ contains just the second (resp. first) of two points 
of $C_{j}\cap Y_{2}^{t}$.  So since $c< c'-1$ it follows that  
$S' = f_{2}(C_{j})\cap Y_{2}^{(s+1\rightarrow s+c)}$.  So $|S'| = \pi(j,s+1\rightarrow s+c)$.  Thus 
by Theorem \ref{two dim properties}c we have $||S'| - \pi(i,r\rightarrow r+c-1)|\le 1$.  Now $|S'|\le p-2$, 
since $S'$ omits at least the two endpoints of the path $f_{2}(T')$.  But also $\pi(i,r\rightarrow r+c-1)\geq |S|\geq p$, 
a contradiction.  

Finally we prove part (e).  The containment in (e) follows from (d) (also from Theorem \ref{two dim properties}e,f).  The bound on $|Y_{2}^{(r')} - f_{2}(D_{2}^{(r)})|$ 
follows from  $Y_{2}^{(r'-1)}\subset f_{2}(D_{2}^{(r)})$, since by that containment 
the set $Y_{2}^{(r')} - f_{2}(D_{2}^{(r)})$ is a proper subset of the column $Y_{2}^{r'}$.

\section{Appendix 2: Glossary of notation}

$\bullet$ $\mathbf{G = [a_{1}\times a_{2}\times \cdots \times a_{k}  ]}$: the  $k$-dimensional grid, a graph with   
vertex set $V(G) = \{ x = (x_{1}, x_{2}, \ldots, x_{k}): x_{i}$ an integer, $1\le x_{i}\le a_{i} \}$ and edge set 
$E(G) = \{ xy: \sum_{i=1}^{k}|x_{i} - y_{i}| =1   \}$

\noindent $\bullet$ $\mathbf{P(t)}$:  the path graph on $t$ vertices 

\noindent $\bullet$ $\mathbf{e_{i}} = \lceil log_{2}(a_{1}a_{2}\cdots a_{i})  \rceil$ for $1\le i\le k$, with $e_{0} = 0$ 

\noindent $\bullet$ $\mathbf{i}$-\textbf{page}:  a subgraph of $G$ obtained by fixing the last $k-i$ coordinates, and letting the first 
$i$ coordinates vary over all their possible values  

\noindent $\bullet$  $\mathbf{D_{i}^{r}}$, $\mathbf{D_{i}^{(r)}}$, $\mathbf{D_{i}^{r}(j)}$: $D_{i}^{r}$ is the $r$'th $i$-page of $G$ under the ordering of $i$-pages given as follows.  
Let $D_{i}$ and  $D_{i}'$ be two $i$-pages, with fixed last $k-i$ coordinate values $c_{i+1}, c_{i+2},\ldots ,c_{k}$ 
and $c_{i+1}', c_{i+2}',\ldots ,c_{k}'$ respectively.  Then $D_{i} \prec_{i} D_{i}' $ in this ordering 
if at the maximum index $t$, $i+1\le t\le k$, where $c_{t}\ne c_{t}'$ we have $c_{t} < c_{t}'$.  $D_{i}^{(r)} = \cup_{j=1}^{r}D_{i}^{j}$. 
$D_{i}^{r}(j) = D_{i-1}^{(r-1)a_{i}+j}$, and we regard $D_{i}^{r}(j)$ as the 
 $j$'th  $(i-1)$-subpage of $D_{i}^{r}$ under the ordering of $(i-1)$-subpages of $D_{i}^{r}$ induced by $\prec_{i-1}$.

\noindent $\bullet$  $\mathbf{ P_{i}} = a_{i+1}a_{i+2}\ldots a_{k} $, the number of $i$-pages in $G$.

\noindent $\bullet$  $\mathbf{\langle Y_{i} \rangle}$ for $1\le i\le k$: the $i$-dimensional grid given by
\newline $\langle Y_{i} \rangle = P(2^{e_{1}})\times P(2^{e_{2}-e_{1}})\times P(2^{e_{3}-e_{2}})\times \cdots \times P(2^{e_{i}-e_{i-1}}) .$  Note that 
$\langle Y_{i} \rangle$ is a spanning subgraph of $Q_{e_{i}}$, the hypercube of dimension $e_{i}$.

\noindent $\bullet$  $\mathbf{Y_{i}}$: the $i$-dimensional grid given by $Y_{i} = \langle Y_{i-1} \rangle \times P(l_{i})$, where $l_{i}$ is any 
 integer satisfying $l_{i}\geq 2^{e_{k}-e_{i-1}}$
 
 \noindent $\bullet$  $\mathbf{(i-1)}$-\textbf{level of} $\mathbf{Y_{i}}$: the subgraph of $Y_{i}$ consisting of all points with some fixed value $c$ in the $i$'th 
 coordinate, $1\le c\le l_{i} $.
 
 \noindent $\bullet$  $\mathbf{Y_{i}^{j}}$: $Y_{i}^{j}$ is the $(i-1)$-level of $Y_{t}$ consisting of all points with fixed coordinate value $j$ 
 in the $i$'th coordinate.       $(x_{1}, x_{2}, \dots, x_{t-1}, j )$.
  
 \noindent $\bullet$  $\mathbf{Y_{i}^{(r)}} = \bigcup_{j=1}^{r}Y_{i}^{j}$.
 
 \noindent $\bullet$  $\mathbf{i}$-\textbf{section of} $\mathbf{Y_{i}}$,    $\mathbf{S_{i}^{j}}$, $\mathbf{S_{i}^{j}(G)}$, $\mathbf{S_{i}^{j}(G)'}$, $\mathbf{S_{i}^{j}(c)}$: $S_{i}^{j} = \bigcup_{j=a}^{b}Y_{i}^{j}$, 
 where $a = 1+ (j-1)2^{e_{i}-e_{i-1}}$ and $b = j2^{e_{i}-e_{i-1}}$.  We call $S_{i}^{j}$ 
 the $j$'th ``$i$-section" of $Y_{i}$.  $S_{i}^{j}(G)$ is the set of points of $S_{i}^{j}$ lying in nonblank $(i-1)$-levels of $S_{i}^{j}$. $S_{i}^{j}(G)' = S_{i}^{j}(G)\cap (I_{i}\circ f_{i})(G)$, so 
 $S_{i}^{j}(G)'$ is the set of points lying in nonblank columns of $S_{i}^{j}$ which are also in the image of $I_{i}\circ f_{i}$.  $S_{i}^{j}(c)$ is the $c$'th nonblank $(i-1)$-level of 
 $S_{i}^{j}$ in order of increasing $i$'th coordinate.  We interpret $S_{i}^{j}(c)$ for certain $c$ using wraparound (see the comments preceding Theorem \ref{containment}).            
 
\noindent $\bullet$  $\mathbf{S_{i}^{(r)}}$, $\mathbf{S_{i}^{(r)}(G)}$, $\mathbf{S_{i}^{(r)}(G)'}$: $S_{i}^{(r)} =  \bigcup_{j=1}^{r}S_{i}^{j}$, the union of the first $r$ many $i$-sections of $Y_{i}$. $S_{i}^{(r)}(G)$ is the 
  set points in $S_{i}^{(r)}$ lying in nonblank $(i-1)$-levels of $S_{i}^{(r)}$. $S_{i}^{(r)}(G)' = S_{i}^{(r)}(G)\cap (I_{i}\circ f_{i})(G)$.   
  
\noindent $\bullet$  $\mathbf{s_{i}(j)}$: the number of $(i-1)$-levels in $S_{i}^{j}$ which are designated blank.  By construction we have 
$s_{i}(j) = 2^{e_{i} - e_{i-1}} - \lceil \frac{a_{1}a_{2}\cdots a_{i}}{2^{e_{i-1}}} \rceil + \lfloor  j\phi_{i} \rfloor - \lfloor  (j-1)\phi_{i} \rfloor $, 
where $\phi_{i} = \lceil \frac{a_{1}a_{2}\cdots a_{i}}{2^{e_{i-1}}} \rceil -  \frac{a_{1}a_{2}\cdots a_{i}}{2^{e_{i-1}}}$. 
  
\noindent $\bullet$  $\mathbf{F(i)}$: the $P_{i}\times 2^{e_{i}-e_{i-1}}$, $(0,1)$ matrix $F(i) = f_{uv}(i)$ that encodes which $(i-1)$-levels of $S_{i}^{(P_{i})}$ 
are to be designated blank as follows.  We have $ f_{uv}(i) = 1$ if the $v$'th $(i-1)$-level of $S_{i}^{u}$ (which is $Y_{i}^{v + (u-1)2^{e_{i}-e_{i-1}} }$) is blank, 
and $ f_{uv}(i) = 0$ otherwise.  In particular, the sum of entries in the $u$'th row of $F(i)$, being the number of blank $(i-1)$-levels in $S_{i}^{u}$, is $s_{i}(u)$.  Note that 
$F(i)$ is constructed in the procedure following Corollary \ref{zeros column index}.

\noindent $\bullet$ $\mathbf{N_{r}(d)}$, $\mathbf{N_{i,r}(d)}$: For a fixed $i$ and $1\le r\le P_{i}$, $N_{r}(d)$ is the column index of the $d'$th $0$ from the left of row $r$ of matrix $F(i)$.  When 
$i$ can vary, we let $N_{i,r}(d)$ be the $N_{r}(d)$ just defined relative to matrix $F(i)$.  Further we interpret the function $N_{r}(*)$ using "wraparound" (see the comments preceding 
Corollary \ref{zeros column index}).

\noindent $\bullet$  $\mathbf{Opt(G)}$, $\mathbf{Opt'(G)}$ for $G = [a_{1}\times a_{2}\times \cdots \times a_{k}  ]$: 
$Opt(G) = Q_{n}$ with $n = \lceil log_{2}(|V(G)|)   \rceil$. This is the 
hypercube of smallest dimension having at least as many vertices as $G$.  
$ Opt'(G) = P(2^{e_{1}})\times P(2^{e_{2}-e_{1}})\times P(2^{e_{3}-e_{2}})\times \cdots \times P(2^{e_{k}-e_{k-1}})$, a spanning subgraph of $Opt(G).$

\noindent $\bullet$ $\mathbf{u_{i} = u_{i}(G)}$: For $G$ a $k$-dimensional grid, we let $u_{i} = \lceil \frac{|G|}{2^{e_{i-1}}} \rceil$.  For each $i$, $2\le i\le k$, we constructed 
a map $f_{i}: G\rightarrow Y_{i}^{(u_{i})}$; that is, a map of $G$ into the first $u_{i}$ many $(i-1)$-levels of $Y_{i}$.  Each such $(i-1)$-level has size $2^{e_{i-1}}$, 
so  $u_{i}$ is the minimum number of $(i-1)$-levels required in the image of any one to one map $G\rightarrow Y_{i}$.

\noindent $\bullet$  $\mathbf{x_{i\rightarrow j}}$ $= (x_{i}, x_{i+1}, x_{i+2}, \ldots, , x_{j-1}, x_{j})$, for a $t$-tuple $x = (x_{1}, x_{2}, x_{3}, \ldots, x_{t})$, $1\le i,j\le t$.

\noindent $\bullet$  $\mathbf{I_{i}}$: the ``inflation" map $I_{i}: f_{i}(G)\rightarrow S_{i}^{(P_{i})}(G)$ defined in step $3$ (The Inflation Step) of the construction 
of $f_{i+1}$ as follows.  For any $z = (z_{1},z_{2},\ldots,z_{i})\in f_{i}(G)$, let $I_{i}(z) = (z_{1},z_{2},\ldots,z_{i-1}, z_{i}'),$ where 
$Y_{i}^{z_{i}'}$ is the $z_{i}$'th nonblank level in $S_{i}^{(P_{i})}$ in order of increasing $i$'th coordinate.

\noindent $\bullet$  \mbox{\boldmath$\sigma_{i}$}: the ``stacking" map $\sigma_{i}: S_{i}^{(P_{i})}(G)'\rightarrow S_{i}^{1}\times P(u_{i+1})\subseteq Y_{i+1}^{(u_{i+1})}$ 
defined in step $4$ (The Stacking Step) of the construction 
of $f_{i+1}$.  Informally, $\sigma_{i}$ stacks the sets $S_{i}^{r}(G)'$ ``on top of" $S_{i}^{1}\cong \langle Y_{i}\rangle$ in order of increasing $r$ as follows.  Suppose 
$x = (x_{1}, x_{2}, \ldots, x_{i})$ lies in $S_{i}^{r}(G)'$.  Then let   
 $\sigma_{i}(x) = (x_{1}, x_{2}, \ldots, x_{i-1}, \bar{x}_{i}, c)$, where $\bar{x}_{i}$, $1\le \bar{x}_{i}\le 2^{e_{i}-e_{i-1}}$,  is the integer congruent to $x_{i}$ mod $2^{e_{i}-e_{i-1}}$, and 
 $c = \sigma_{i}(x)_{i+1}$ is the number of images $\sigma_{i}(y)$ with $y\in S_{i}^{j}(G)$, $1\le j\le r$, and $\sigma_{i}(x)_{1\rightarrow i} = \sigma_{i}(y)_{1\rightarrow i} $. 

\noindent $\bullet$  \mbox{\boldmath${\bar \sigma_{i}}$}: ${\bar \sigma_{i}}$ is the projection of $\sigma_{i}$ onto the first $i$ coordinates.  So for $z\in S_{i}^{(P_{i})}(G)'$, we have 
${\bar \sigma_{i}}(z) = \sigma_{i}(z)_{1\rightarrow i}$.  By Lemma \ref{whatever}c, the restriction of ${\bar \sigma_{i}}$ to any one $i$-section is one to one.   

\noindent $\bullet$  $\mathbf{Stack_{i}(x,r)}$: for $x\in S_{i-1}^{1}\cong \langle Y_{i-1}  \rangle$ and $1\le r\le P_{i-1}$ (defined in the paragraph just preceding Lemma \ref{whatever}):         
$Stack_{i}(x,r) = \{ z = \sigma_{i-1}(y): z_{1\rightarrow i-1}= x$ and $y\in S_{i-1}^{(r)}(G)'  \}$.  We regard $Stack_{i}(x,r)$ as a stack, with ``height" extending 
into the $i$'th dimension, addressed by $x\in \langle Y_{i-1}  \rangle$.  Its elements are images 
$ \sigma_{i-1}(y)\in Y_{i}$, with $y\in S_{i-1}^{(r)}(G)' $, having projection $\sigma_{i-1}(y)_{1\rightarrow i-1} = x$ onto the first $i-1$ coordinates.     

\noindent $\bullet$  $\mathbf{[r]_{i}}$ $=$ max $\{|Stack_{i}(x,r)|: x\in \langle Y_{i-1} \rangle  \} $.  This is the maximum ``height" of $Stack_{i}(x,r)$, over all $x\in \langle Y_{i-1} \rangle  $.

\noindent $\bullet$ $\mathbf{v_{i,r}(S)}$: for $S\subseteq Y_{i}^{([r]_{i})}$, let $v_{i,r}(S) = |S \cap f_{i}(D_{i-1}^{(r)}) |$.

\noindent $\bullet$ $\mathbf{\nu_{i}(z)}$: For $z = (I_{i}\circ f_{i})(x)$, where $x\in D_{i}^{r}$, $\nu_{i}(z)$ is the integer such that $z\in S_{i}^{r}(\nu_{i}(z))$.  That is, 
$z$ belongs to the $\nu_{i}(z)$'th nonblank $(i-1)$-level (ordered by increasing $i$-coordinate) of the $i$-section $S_{i}^{r}$ containing $z$.  We interpret differences 
$\nu_{i}(z') - \nu_{i}(z'')$ using wraparound (see the comments preceding Theorem \ref{containment}).

 \section{ Appendix 3: Figures}

\begin{figure}[ht]
\centering
\includegraphics[scale=0.35]{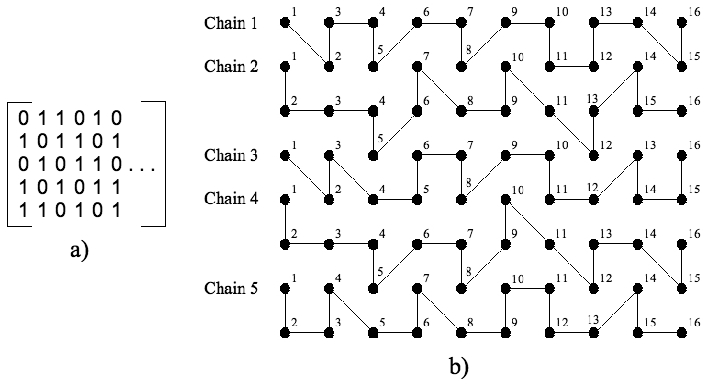}
\caption{ a) The matrix $R$, and b) the corresponding map $f:G(5)\rightarrow Y_{2}$ }\label{layout of G(5)}
\end{figure}

\begin{figure}[ht]
\centering
\includegraphics[scale=0.40]{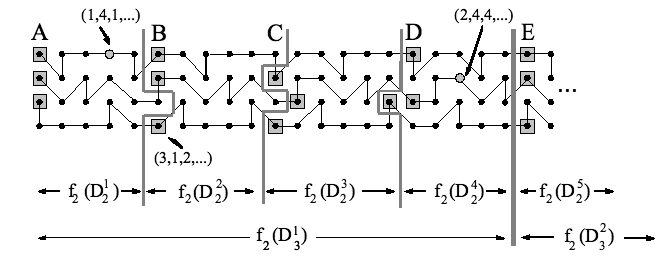}
\caption{ The map $f_{2}:[3\times 7\times 4\times a_{4}]\rightarrow Y_{2}$, with $a_{4}>1$ }\label{3by7}
\end{figure}

\begin{figure}[ht]
\centering
\includegraphics[scale=0.47]{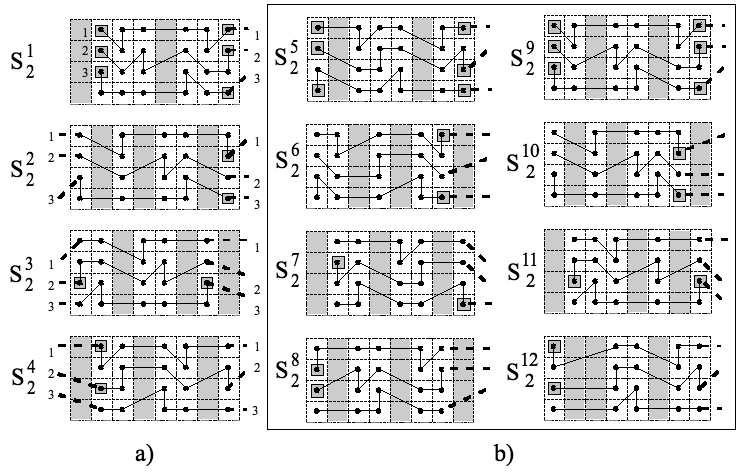}
\caption{ a) $I_{2}:Y_{2}^{(21)}\rightarrow S_{2}^{(4)}$, used for constructing $f_{3}([3\times 7\times 4])$; 
a) and b) combined $I_{2}:Y_{2}^{(63)}\rightarrow S_{2}^{(12)}$, used for constructing $f_{3}([3\times 7\times 4\times 3])$}\label{1thru12}
\end{figure}

\begin{figure}[ht]
\centering
\includegraphics[scale=0.43]{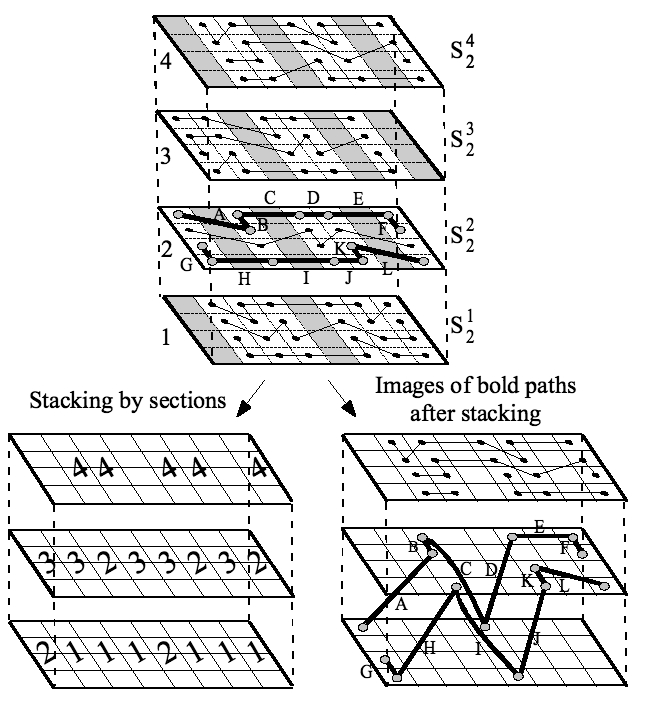}
\caption{The stacking map $\sigma_{2}:(I_{2}\circ f_{2})([3\times 7\times 4])\rightarrow Y_{3}^{(3)}$; stacking the four sections $S_{2}^{j}$, $1\le j\le 4$}\label{stack4sections}
\end{figure} 

\begin{figure}[ht]
\centering
\includegraphics[scale=0.40]{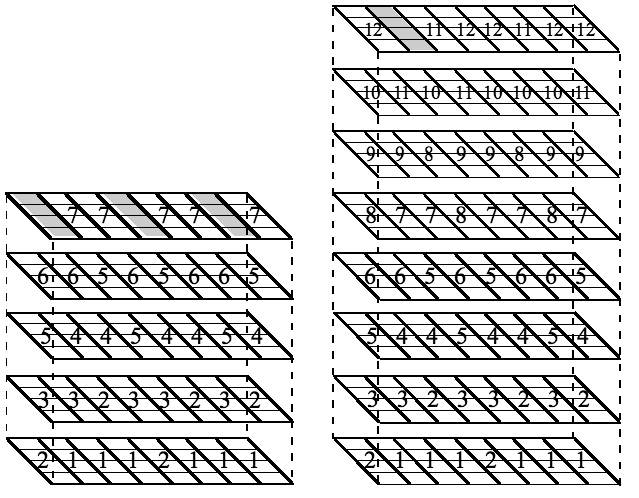}
\caption{ $\sigma_{2}$ applied to the first $7$ and all $12$ sections $S_{2}^{j}$, yielding the map 
$f_{3}:[3\times 7\times 4\times 3]\rightarrow Y_{3}^{(8)}$}\label{7and12stacking}
\end{figure}

\begin{figure}[ht]
\centering
\includegraphics[scale=0.48]{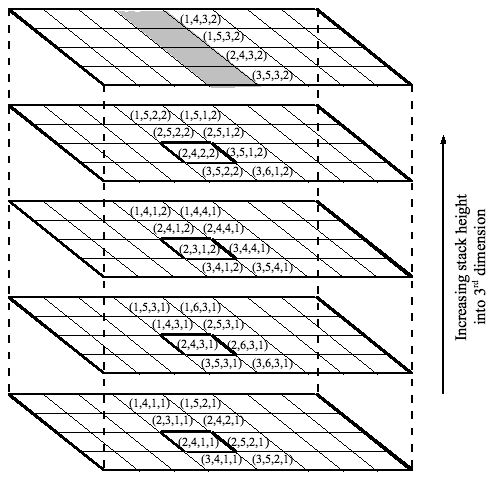}
\caption{Contents of stacks in columns 4 and 5 after stacking the first 7 sections $S_{2}^{j}$, $1\le j\le 7$; points are labeled by their preimages in $G$}\label{columns 4 and 5}
\end{figure}

\begin{figure}[ht]
\centering
\includegraphics[scale=0.43]{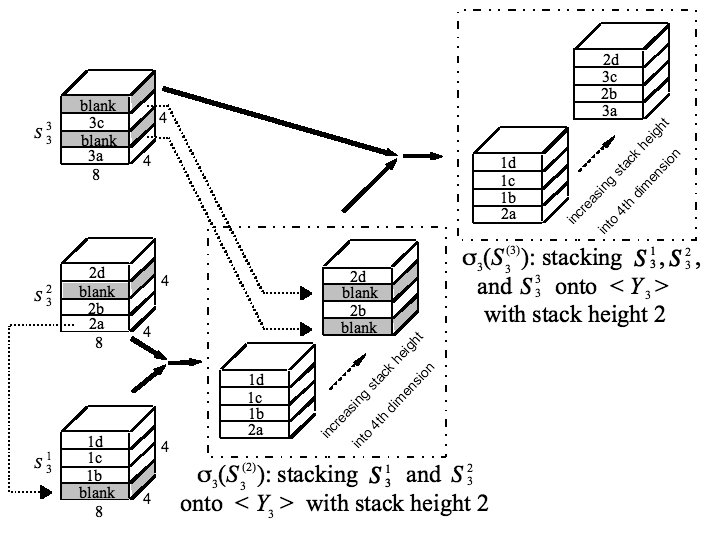}
\caption{The stacking map $\sigma_{3}(S_{3}^{(3)})\subseteq Y_{4}^{(2)}$, yielding the map $f_{4}:[3\times 7\times 4\times 3]\rightarrow Y_{4}^{(2)}$}\label{stacking in 4D}
\end{figure}

\begin{figure}[ht]
\centering
\includegraphics[scale=0.48]{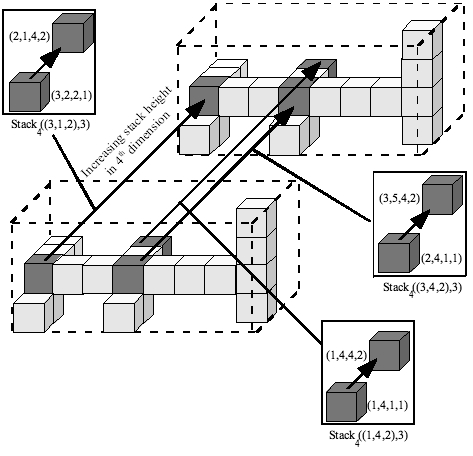}
\caption{Three individual stacks, each addressed by points of $S_{3}^{1}\cong \langle Y_{3} \rangle$, each of height $2$, extending into the $4$'th dimension}\label{detail in 4D}
\end{figure}


\begin{thebibliography}{99}

\bibitem{As1} T. Asano, T. Matsui, and T. Tokuyama: Optimal Roundings of Sequences and Matrices, \emph{Nordic J. Comput.} 
\textbf{7} (2000) 241-256.

\bibitem{As2} T. Asano, N. Katoh, K. Obokata, and T. Tokuyama: Matrix Rounding Under the $L_{p}$-Discrepancy Measure and Its
Application to Digital Halftoning, \emph{SIAM J. Comput.} \textbf{32} No.6 (2003) 1423-1435. 

\bibitem{B} U. Blass: Short Dominating Paths and Cycles in the Binary Hypercube: \emph{Annals of Combinatorics} \textbf{5} No. 1 (2001) 51-59.

\bibitem{Bar} Zs. Baranyai: On the Factorization of the Complete Uniform Hypergraph, in \emph{Infinite and Finite Sets, Proc. Coll. Keszthely, 1975, Colloquia Math. Soc. J\'{a}nos Bolyai} 
A. Hajnal, R. Rado, V.T. S\'{o}s (Eds.), Colloquia Math. Soc. J\'{a}nos Bolyai \textbf{10} North-Holland (1975) 91-108.

\bibitem{BMS} S. Bettayeb, Z. Miller, and I.H. Sudborough: Embedding Grids into Hypercubes, \emph{J. Comput. System Sci.} \textbf{45} No. 3
(1992) 340-366.  

\bibitem{BCLR} S. Bhatt, F.R.K. Chung, F.T. Leighton, and A. Rosenberg: Efficient Embeddings of Trees in Hypercubes,  \emph{SIAM J. Comput.} \textbf{21} No.1 (1992) 151-162.

\bibitem{CK} R. Caha and V. Koubek: Spanning Regular Caterpillars in Hypercubes, \emph{European J. Combinatorics} \textbf{18} (1997) 249-266.

\bibitem{CHL} Y. Cai, T.H. Lai, and C.C. Lin: Embedding Two-Dimensional Grids into Hypercubes with Dilation $2$ and Congestion $5$, Proc. 
Intern. Conf. on Parallel Processing, Oconomowoc, Wisc., III (1995) 105-112.  

\bibitem{Ch} M.Y. Chan: Embeddings of Grids into Optimal Hypercubes, \emph{SIAM J. Comput.} 
\textbf{20} No. 5 (1991) 834-864.   

\bibitem{D3} B. Doerr: Matrix Rounding with Respect to Small Submatrices,  
\emph{Random Structures \& Algorithms} \textbf{28} (2006) 107-112.

\bibitem{D1} B. Doerr, T. Friedrich, C. Klein, R. Osbild: Rounding of Sequences and Matrices, with Applications, 
\emph{Proceedings of the Workshop on Approximation and Online Algorithms (WAOA 2005)} T. Erlebach, G. Persiano (Eds.)
Vol. 3879 of Lecture Notes in Computer Science, Springer Verlag (2006) 96-109.

\bibitem{D2} B. Doerr, T. Friedrich, C. Klein, and R. Osbild: Unbiased Matrix Rounding,
\emph{Proceedings of the 10th Scandinavian Workshop on Algorithm Theory (SWAT)}, L. Arge, R. Freivalds (Eds.) 
Vol. 4059 of Lecture Notes in Computer Science, Springer Verlag (2006) 102-112.

\bibitem{DHLM} T. Dvorak, I. Havel, J-M Laborde, and M. Mollard: Spanning Caterpillars of a Hypercube, \emph{Journal of Graph Theory} \textbf{24} No. 1 (1997) 9-19.

\bibitem{GT} Q.-P. Gu and H. Tamaki: Routing a Permutation in the Hypercube by Two Sets of Edge Disjoint Paths, 
\emph{J. Parallel Distrib. Comput.} \textbf{44}(2) (1997) 147-152.

\bibitem{HL} I. Havel and P. Liebl:  Embedding the Polytomic Tree into the $n$-cube, \emph{Casopis Pest. Mat.} \textbf{98} (1973)
307-314.

\bibitem{HM} I. Havel and J. Moravek: $B$-Valuations of Graphs, \emph{Czech. Math. J.} \textbf{22} (1972) 338-351.

\bibitem{HLW} F. Harary, M. Lewinter, and W. Widulski: On Two-Legged Caterpillars 
which Span Hypercubes \emph{Congress. Numer.} (19'th Southeastern Conference on Comb., Graph Theory, and Computing) 
\textbf{66} (1988) 103-108.

\bibitem{KL} Y.M. Kim and T.H. Lai: The Complexity of Congestion-1 Embedding in a Hypercube, \emph{J. of Algorithms} 
\textbf{12} No. 2 (1991) 246-280.

\bibitem{KN} D.E. Knuth: Two-Way Rounding, \emph{SIAM J. Disc. Math} 
\textbf{8} No. 2 (1995) 281-290.

\bibitem{JL} J. Lee: On the Efficient Simulation of Networks by Hypercube Machines, Ph.D. Dissertation, Dept. of Electrical 
Engineering and Computer Science, Northwestern University (1990).

\bibitem{Lei} F.T. Leighton: \emph{Introduction to Parallel Algorithms and Architectures: Arrays, Trees, Hypercubes}, Morgan 
Kaufman Publishers, San Mateo, CA (1992).

\bibitem{LSt} M. Livingston and Q. Stout: Embeddings in Hypercubes, \emph{Mathematics and Computational Modelling} \textbf{11} 
(1988) 222-227. 

\bibitem{LSV} L. Lovasz, J. Spencer, and K. Vestergombi: Discrepancy of Set Systems and Matrices, \emph{European J. 
of Combin.} \textbf{7} (1986) 151-160.

\bibitem{MS} Z. Miller and I.H. Sudborough: Compressing Grids into small Hypercubes, \emph{Networks} \textbf{24} No. 6 
(1994) 327-357.

\bibitem{RH} A. L. Rosenberg, L. S. Heath: \emph{Graph Separators with Applications}, Kluwer Academic/Plenum Publishers, New York, New York (2001).

\bibitem{RS} M. Rottger, U. Schroeder: Embedding 2-Dimensional Grids into Optimal Hypercubes with Edge Congestion 1 or 2, \emph{Parallel Processing Letters} 
\textbf{8} No. 2 (1998) 231-242.

 
\bibitem{SP} J. Spencer: Ten Lectures on the Probabilistic Method: \emph{CBMS-NSF Regional Conference
Series in Applied Mathematics} \textbf{52} SIAM, Philadelphia (1987) 37-44.


\bibitem{W} D.B. West: \emph{Introduction to Graph Theory} (Second Edition) Prentice Hall (2001).

\bibitem{SW} S. E. Wright: Integer Matrices with Constraints on Leading Partial Row and Column Sums, \emph{Discrete Applied Mathematics} 
Vol. 158 Issue 16 (2010) 1838-1847.




\end{thebibliography}
\end{document}